\documentclass[10pt]{article}
\usepackage{geometry}
\usepackage{tabularx}
\usepackage{titlesec}
\usepackage{tcolorbox}
\usepackage{ulem}
\usepackage{graphicx} 
\usepackage{shapepar} 
\usepackage{url} 
\usepackage{color} 
\usepackage{float} 
\usepackage{amsmath, amsthm}
\usepackage{amssymb}
\usepackage{amsfonts}
\usepackage{mathrsfs}
\usepackage{hyperref}
\usepackage{caption} 
\usepackage{setspace}
\usepackage{diagbox}     
\usepackage{cite}

\newcommand{\thickhline}{%
    \noalign {\ifnum 0=`}\fi \hrule height 1pt
    \futurelet \reserved@a \@xhline
}

\newtheorem{theorem}{Theorem}[section]
\newtheorem{mydef}[theorem]{Definition}
\newtheorem{remark}[theorem]{Remark}
\newtheorem{lemma}[theorem]{Lemma}
\newtheorem{proposition}[theorem]{Proposition}
\newtheorem{cor}[theorem]{Corollary}

\numberwithin{equation}{section}

\makeatletter
\def\blfootnote{\xdef\@thefnmark{}\@footnotetext}
\makeatother

\def\dd{\,d}
\def\m{\mathbb}					
\def\lam{\lambda}  \def\eps{\epsilon}   \def\veps{\varepsilon}  
\def\a{\alpha}	\def\b{\beta}	\def\ls{\lesssim}
\def\p{\partial}  
\def\wh{\widehat}	\def\la{\langle}	\def\ra{\rangle}
\def\ls{\lesssim}	\def\gs{\gtrsim}	\def\wt{\widetilde}
\def\i{\int\limits}		
\def\be{\begin{equation}}     \def\ee{\end{equation}}
\def\bp{\begin{pmatrix}}	\def\ep{\end{pmatrix}} 
		\def\F{\mathscr{F}}

\title{Mean Effects on Critical Well-Posedness for  Majda–Biello Systems on the Torus}
\author{Ke Wang and Xin Yang}
\date{}
\begin{document}
\maketitle
\begin{abstract}
This paper studies how the mean of the initial data $u_0$ affects the critical indices concerning local well-posedness for the following Majda-Biello systems:
\[
	\left\{\begin{aligned}
		& u_t + u_{xxx} + vv_x = 0 , \\
		& v_t + \alpha v_{xxx} + (uv)_x = 0 , \\
		& (u,v) \mid_{t=0} = (u_0, v_0) \in H^s(\mathbb{T}) \times H^s(\mathbb{T}),
	\end{aligned}\right.
	\qquad x \in \mathbb{T}, \, t\in \mathbb{R},
\]
where $\mathbb{T}$ refers to the periodic torus and the dispersion coefficient $\alpha$ is restricted in $(0,4] \setminus \{1\}$ which corresponds to resonant cases.
Previously, under the zero-mean assumption on $u_0$, Oh (Int. Math. Res. Not., (18):3516-3556, 2009) determined the critical indices $s^{*}(\alpha)$ of the Sobolev regularity of the initial data for $C^3$ local well-posedness. In particular, Oh showed that 
\[
s^{*}(\alpha) = 
\left\{
\begin{array}{lll}
1, & \text{for $\alpha$ such that $\sqrt{12/\alpha - 3} \in \mathbb{Q}$ }, \\
\frac12, & \text{for a.e. $\alpha$ such that $\sqrt{12/\alpha - 3} \notin \mathbb{Q}$ }.
\end{array}\right.\]
In this paper, by allowing the mean of $u_0$ to be non-zero, we find that the critical index $s^{*}(\alpha)$ can be lowered from $1$ to $\frac12$ when $\sqrt{12/\alpha - 3} \in \mathbb{Q}$.
For other values of $\alpha$, except in a set of zero measure, we also justify the critical index $s^{*}(\alpha)$ to be $\frac12$ regardless of the mean of $u_0$. By subtracting the mean from $u_0$, the original Majda-Biello systems are slightly modified to contain first-order terms but with zero-mean initial data. The key ingredient in our proof is to introduce a refined Diophantine approximation theory to capture the essential resonance effect for the perturbed dispersive structure caused by these additional first-order terms. It turns out that only when $\sqrt{12/\alpha - 3} \in \mathbb{Q}$, the lower bound of the resonance function can be improved so that the critical index of local well-posedness can be smaller.
\end{abstract}

\blfootnote{\hspace{-0.25in} 2010 Mathematics Subject Classification. 35Q53; 35G55; 35L56}

\blfootnote{\hspace{-0.25in} Key words and phrases. KdV-KdV systems;   Majda-Biello systems; Local well-posedness; Fourier restriction spaces; Diophantine Approximation Theory.}

\begin{center}
\tableofcontents
\end{center}

\section{Introduction}
The initial value problem of Majda-Biello systems in the periodic domain reads as 
\begin{equation}\label{GeneralMajdaBiello}
	\left\{\begin{aligned}
		&u_t+u_{xxx}+vv_x= 0 , \\
		&v_t+ \alpha v_{xxx}+(uv)_x=0 , \\
		&(u,v)\mid_{t=0} =(u_0,v_0)\in H^s(\mathbb{T})\times H^s(\mathbb{T}),
	\end{aligned}\right.
\qquad x\in \mathbb{T}, \, t\in \mathbb{R},
\end{equation}
where $\mathbb{T} = \mathbb{R}/{(2\pi\mathbb{Z})}$ denotes the periodic torus with length $2\pi$, and the dispersion coefficient $\alpha\in \mathbb{R}\setminus\{0\}$, and $H^s$ refers to Sobolev spaces. This system was proposed by Majda and Biello in \cite{MajdaBiello} as a reduced asymptotic model describing the nonlinear resonant interactions between
planetary waves in rotating fluids such as the Earth's oceans and atmosphere. 

Majda-Biello systems are special coupled KdV systems. We first recall the initial value problem of the KdV equation:
\begin{equation}\label{KdV}
	\left\{
	\begin{aligned}
		u_t+u_{xxx}+uu_x = 0,\\
		u\mid_{t = 0} = u_0\in H^s(\mathbb{G}),
	\end{aligned}
	\right.
\end{equation}
where $\mathbb{G} = \mathbb{R}$ or $\mathbb{T}$. The problem of well-posedness problem (low regularity) asks for the smallest $s$ such that (\ref{KdV}) is well-posed. The study of this problem has produced satisfactory results after more than 50 years of development; see, e.g., \cite{BS75,BS76,Kat79,Kat83,KPV91,Bourgain93, KPV96SahrpIndex, Iterm03, ChristCollianderTao2003, KappelerTopalov06, GuoZiHua09, Kishimoto09, Molinet2011, Molinet2012, KillipRowanVicsan19}. In particular, we list some influential works that dealt with the cases when $s\leq 0$. Firstly, by introducing Fourier restriction spaces, Bourgain established the global well-posedness of (\ref{KdV}) in both $H^{s}(\mathbb{R})$ and $H^{s}(\mathbb{T})$ for $s\geq 0$. Then Kenig, Ponce and Vega \cite{KPV96SahrpIndex} further refined the result to justify the local well-posedness in $H^{-\frac{3}{4}+}(\mathbb{R})$ and $H^{-\frac{1}{2}}(\mathbb{T})$. The corresponding global well-posedness was justified by Colliander, Keel, Staffilani, Takaoka and Tao \cite{Iterm03} using the I-method. The endpoint case $s = -\frac{3}{4}$ in $H^{s}(\mathbb{R})$ was proved by Guo \cite{GuoZiHua09} and Kishimoto \cite{Kishimoto09}. Finally, utilizing the complete integrability structure of the KdV equation (\ref{KdV}), the global well-posedness of (\ref{KdV}) was established in $H^{s}(\mathbb{T})$ for $s\geq -1$ by Kappeler and Topalov \cite{KappelerTopalov06}, and in $H^{s}(\mathbb{R})$ for $s\geq -1$ by Killip and Visan \cite{KillipRowanVicsan19}. The index $-1$ for both the $\m{R}$ case and the $\m{T}$ case are justified to be sharp by Molinet \cite{Molinet2011, Molinet2012}. Furthermore, if the well-posedness is strengthened to be analytical well-posedness, which means the solution map is analytic rather than simply continuous, then the sharp index is found to be $-\frac34$ in the $\m{R}$ case and $-\frac12$ in the $\m{T}$ case, see \cite{ChristCollianderTao2003} by Christ, Colliander and Tao.

In applications, two or more KdV equations may be coupled to model complicated phenomena. The most widely used models include the Majda-Biello systems \cite{HirotaSatsuma1981}, the Hirota-Satsuma systems \cite{MajdaBiello}, and the Gear-Grimshaw systems \cite{GearGrimshaw}. Thanks to the development of the well-posedness theory for the single-KdV equation, Oh \cite{Oh2009} initiated the breakthrough on the coupled KdV systems when studying the Majda-Biello systems 
\be\label{MB}
	\left\{\begin{aligned}
		&u_t + u_{xxx} + vv_x= 0 , \\
		&v_t + \alpha v_{xxx} + (uv)_x=0 , \\
		&(u,v)\mid_{t=0} =(u_0,v_0)\in H^s(\mathbb{G})\times H^s(\mathbb{G}),
	\end{aligned}\right.
\qquad x\in \mathbb{G}, \, t\in \mathbb{R}.
\ee
Since Majda-Biello systems are not completely integrable anymore, then the most effective method is relying on a contraction mapping argument in Fourier restriction spaces, which leads to analytical well-posedness automatically. As a result, we will focus on the analytical well-posedness in the following and denote by $s^{*}_{\mathbb{G}}(\a)$ the smallest value such that (\ref{MB}) is locally analytically well-posed in $H^s(\mathbb{G})\times H^s(\mathbb{G})$ for any $s > s^{*}_{\mathbb{G}}(\a)$. 

When $\a = 1$, the main terms for both equations in (\ref{MB}) agree with that in the single KdV equation, so the well-posedness problem can be handled similarly. When $\a < 0$ and $\a > 4$, there is no essential resonance effect of the system (\ref{MB}), it turns out that $s^{*}_{\mathbb{G}}(\a)$ also matches the index in the single KdV case. Therefore, the challenging cases are $\a\in(0,4]\setminus\{1\}$. When $\m{G} = \m{R}$, Oh \cite{Oh2009} showed that $s^{*}_{\m{R}}(\a) = 0$ for any $\a\in(0,4)\setminus\{1\}$. Later, Yang and Zhang \cite{YangZhang} studied the end-point case $\a=4$ and found $s^{*}_{\m{R}}(4)=\frac34$. These critical indices are higher than that in the single-KdV case, thus revealing the essential difference between the single-KdV equation and coupled KdV systems (\ref{MB}) in the resonant cases $\a \in (0,4]\setminus\{1\}$. When $\m{G} = \m{T}$, by assuming $u_0$ has zero mean and by incorporating Diophantine approximation theory, Oh showed that $s^{*}_{\m{T}}(\a)\in [\frac12, 1]$ and $s^{*}_{\m{T}}(\a) = \frac12$ for almost every $\a \in (0, 4]\setminus\{1\}$. The assumption that $u_0$ has zero mean is necessary to establish the bilinear estimates \cite{KPV96SahrpIndex, Oh2009}. In the single KdV case, by subtracting the mean from its initial data, the new initial data automatically satisfies this assumption and the new equation adds a first-order term which does not affect the resonance structure of the equation. However, due to the interaction between $u$ and $v$ in coupled KdV systems, the appearance of first-order terms may bring additional complexity to the systems. Thus, the goal of this paper is to investigate whether and how the mean of $u_0$ affects the critical index $s^{*}(\a)$ for (\ref{MB}) with $\m{G} = \m{T}$ and $\a\in (0, 4]\setminus\{1\}$. 

For the initial data $u_0\in H^{s}(\m{T})$, its Fourier transform is defined as 
\[\widehat{u}_{0}(\xi) = \int_{0}^{2\pi} e^{-i x \xi} u_0(x) \dd x.\]
So the mean of $u_0$ can be represented as $\frac{1}{2\pi} \widehat{u}_{0}(0)$. 
For the case where $u_0$ has non-zero mean, i.e. $\widehat{u}_0(0)\neq 0$, we apply the following transformation: 
\begin{equation*}
\widetilde{u}(x,t) = u(x,t)-\frac{1}{2\pi}\widehat u_0(0),\quad \widetilde{v}(x,t) = v(x,t).
\end{equation*}
For convenience, we denote $\b = \frac{1}{2\pi}\widehat u_0(0)$. Then the target problem (\ref{GeneralMajdaBiello}) is converted to the following.
\begin{equation}\label{mmMajdaBirllo}
\left\{\begin{aligned}
&\widetilde u_t+\widetilde u_{xxx}+\widetilde{v}\widetilde v_x= 0 , \\
&\widetilde v_t+ \alpha \widetilde v_{xxx} + \b \widetilde v_x+(\widetilde u\widetilde v)_x=0 , \\
&(\widetilde u,\widetilde v)\mid_{t=0} =(u_0 - \b,v_0)\in H^s(\mathbb{T})\times H^s(\mathbb{T}),
\end{aligned}\right.
\qquad x\in \mathbb{T},t\in \mathbb{R}.
\end{equation}
The good side of (\ref{mmMajdaBirllo}) is that the mean of its initial data $\wt{u}(\cdot, 0)$ is automatically zero, and this zero-mean property is persistent with respect to time. The bad side of (\ref{mmMajdaBirllo}) is the extra first-order term $\b \wt{v}_{x}$. In the single KdV case on $\m{T}$, this extra term  does not cause any trouble as explained in \cite{Bourgain93}. But it does have essential impact on the critical index $s^{*}$ for coupled KdV systems for special coefficients, see \cite{YangLiZhang}. 

Based on the form of (\ref{mmMajdaBirllo}), we consider the following initial value problem with an extra first-order term $\b v_{x}$ and mean-zero initial data $u_0$.
\begin{equation}\label{mMajdaBiello}
	\left\{\begin{aligned}
		&u_t+u_{xxx}+vv_x= 0 , \\
		&v_t+ \alpha v_{xxx}+\beta v_x+(uv)_x=0 , \\
		&(u,v)\mid_{t=0} =(u_0,v_0)\in H^{s}_{0}(\mathbb{T})\times H^{s}(\mathbb{T}),
	\end{aligned}\right.
	\quad\quad x\in \mathbb{T}, \, t\in \mathbb{R},
\end{equation}
where 
\[
	H_{0}^{s}(\m{T}) := \bigg\{ f\in H^{s}(\m{T}): \int_{0}^{2\pi}f(x)\,\dd x=0\bigg\}.
\]
We denote by $s^{*}(\a, \b)$ the smallest value such that (\ref{mMajdaBiello}) is locally analytically well-posed in $H^{s}_{0}(\mathbb{T})\times H^{s}(\mathbb{T})$ for any $s > s^{*}(\a, \b)$. We first study the case $\a=4$ which corresponds to the most resonant effect, and we find that $s^{*}(4,\b)$ is lowered than $s^{*}(4,0)$ for most values of $\b$. For convenience of notation, we abbreviate the phrase “locally well-posed” to be “LWP”.

\begin{theorem}\label{thm1}
Let $\a = 4$ and assume $u_0$ has zero mean in problem \eqref{mMajdaBiello}. Then (\ref{mMajdaBiello}) is analytically LWP in $H^{s}_{0}(\mathbb{T})\times H^{s}(\mathbb{T})$ for any
\be\label{ci-4}
	\left\{\begin{array}{ll}
		s\geq 1, & \quad \text{for } \beta = 3n^2, n\in \mathbb{N}, \\
		s\geq 1/2, &\quad \text{for } \beta \neq 3n^2, n\in \mathbb{N}.
	\end{array}\right.
\ee
\end{theorem}
In the above theorem, the “analytically LWP” means the solution map 
\[
	\Phi: (u_0, v_0) \in H^{s}_{0}(\mathbb{T})\times H^{s}(\mathbb{T}) \longmapsto (u,v)\in C\big( [0,T]; H^{s}_{0}(\mathbb{T})\times H^{s}(\mathbb{T}) \big) 
\]
is analytic, where the lifespan $T$ depends on the norm of $(u_0,v_0)$. With this requirement, the indices in Theorem \ref{thm1} are sharp. Actually, the solution map even fails to be $C^{k}$ for certain $k$ if $s$ is below the threshold in Theorem \ref{thm1}.
\begin{theorem}\label{thm3}
Under the assumption in Theorem \ref{thm1}, the system (\ref{mMajdaBiello}) fails to be $C^{k}$ LWP in $H^{s}_{0}(\mathbb{T})\times H^{s}(\mathbb{T})$ if 
\begin{equation}
	\begin{cases}
		k\geq 2 \text{ and } s<1, & \quad \text{for $\beta = 3n^2$, $n\in\mathbb{N}$},\\
		k\geq 3 \text{ and } s<1/2, & \quad \text{for $\beta \neq 3n^2$, $n\in\mathbb{N}$}.  
	\end{cases}
\end{equation} 
\end{theorem}
Combining the results in Theorem \ref{thm1} and Theorem \ref{thm3}, we conclude that for the initial value problem (\ref{mMajdaBiello}) with $\a = 4$ in the space $H^{s}_{0}(\mathbb{T})\times H^{s}(\mathbb{T})$, the critical indices are 
\be\label{ci_4}
	s^{*}(4, \beta) :=\left\{\begin{array}{cl}
		1, & \quad \text{for } \beta = 3n^2, n\in \mathbb{N}, \\
		1/2, &\quad \text{for } \beta \neq 3n^2, n\in \mathbb{N}.
	\end{array}\right.
\ee
Now applying this result to (\ref{mmMajdaBirllo}) which is converted from (\ref{GeneralMajdaBiello}), we give a complete answer to how the mean of $u_0$ affects the critical index for the analytical well-posedness of (\ref{GeneralMajdaBiello}) when $\a = 4$. For any $\b\in\m{R}$, we introduce the following notation for the space of the initial data:
\be\label{space_id_nzm}
	H^s_{\b}(\mathbb{T})\times H^s(\mathbb{T}) = \Big\{ (u_0, v_0)\in H^s(\mathbb{T})\times H^s(\mathbb{T}): \frac{1}{2\pi}\wh{u}_0(0) = \b \Big\}.
\ee
\begin{cor}\label{Cor_mres}
	The critical index of the analytically LWP for the initial value problem \eqref{GeneralMajdaBiello} with $\alpha = 4$ in the space $H^s_{\b}(\mathbb{T})\times H^s(\mathbb{T})$ is 
	\begin{equation}
		s_{\b}^{*}(4) = \left\{\begin{array}{cl}
			1, \quad & \text{for } \b = 3 n^2, n\in \mathbb{N},\\
			1/2, \quad & \text{for } \b \neq 3 n^2, n\in \mathbb{N}.
		\end{array}\right.
	\end{equation}
\end{cor}

We remark that when $\a=4$ and $\b = 0$ in Corollary \ref{Cor_mres}, we recover the result in \cite{Oh2009} that the critical index is $1$. Compared with other cases in Corollary \ref{Cor_mres} with $\a=4$, the novelty here is that we lowered the critical index to  $\frac12$ when the initial data $u_0$ has non-zero mean $\b$ which does not take the form $3n^2$, $n\in\m{N}$.

Next, we turn to the case when $\a\in(0,4)\setminus\{1\}$. As discovered by Oh \cite{Oh2009} that the resonance effect in this case is closely related to the Diophantine approximation theory, the critical index of the well-posedness of (\ref{mMajdaBiello}) with $\b=0$ is determined by the minimal type index of 
the number 
\be\label{R_alpha}
	R_{\a} := \sqrt{12/\a - 3},
\ee
see Definition 1 and Definition 2 in \cite{Oh2009} for the meaning of the minimal type index. In this paper, the coefficient $\b$ is nonzero in (\ref{mMajdaBiello}), so we introduce biased minimal type indices to incorporate the effect of $\b$, see Definition \ref{Def, mbti} for more details. Concerning the well-posedness results for problem (\ref{mMajdaBiello}), the following index $s_{\a, \b}$ is essential:
\be\label{s_0}
	s_{\a,\b} := \max\{ \nu_{\lambda}(c_1), \, \nu_{\lambda}(c_2) \},
\ee
where 
\be\label{s_0_aux}
	c_1 = \frac12 + \frac{R_\a}{6}, \quad c_2 = \frac12 - \frac{R_\a}{6}, \quad \lambda = \frac{\b}{\a R_\a},
\ee
and $\nu_{\lambda}$ represents the minimal $\lambda$-biased type index as defined in Definition \ref{Def, mbti}. The explanation of why we introduce (\ref{s_0}) and (\ref{s_0_aux}) is provided after Corollary \ref{Cor_MBres_ci}.
Although there is no explicit formula for $s_{\a,\b}$ in terms of $\a$ and $\b$, we prove in Proposition \ref{munu} that $s_{\a,\b} = 0$ for almost every $\a\in(0,4)\setminus\{1\}$ and for all $\b\in\m{R}$.

\begin{theorem}\label{thm2} 
	Let $\alpha \in (0,4)\setminus\{1\}$, $\beta\neq 0$, and assume $u_0$ has zero mean in problem (\ref{mMajdaBiello}). Then the modified Majda-Biello system \eqref{mMajdaBiello} is analytically LWP in $H^{s}_{0}(\mathbb{T})\times H^{s}(\mathbb{T})$ for 
	\[\left\{
	\begin{array}{lcl}
		s \geq \frac12, &\text{if}& R_\a \in \m{Q}, \vspace{0.1in}  \\
		s\geq 1, &\text{if}& R_\a \notin \m{Q} \text{ and } s_{\a,\b} \geq 1,  \vspace{0.1in} \\
		s > \frac{1 + s_{\a,\b}}{2}, &\text{if}& R_\a \notin \m{Q} \text{ and } s_{\a,\b} < 1.
	\end{array}
	\right.\]
	Moreover, for almost every $\alpha \in (0,4)\setminus\{1\}$, \eqref{mMajdaBiello} is analytically LWP in $H^{s}_{0}(\mathbb{T})\times H^{s}(\mathbb{T})$ for $s > \frac12$.
\end{theorem}
The indices obtained above are also sharp if the solution map is required to be $C^3$, see the following theorem for precise statement.
\begin{theorem}\label{thm4}
	Under the assumptions in Theorem \ref{thm2}, the modified Majda-Biello system \eqref{mMajdaBiello} fails to be $C^3$ LWP in $H^{s}_{0}(\mathbb{T})\times H^{s}(\mathbb{T})$ for 
	\[\left\{
	\begin{array}{lcl}
		s < \frac12, &\text{if}& R_\a \in \m{Q},  \vspace{0.1in} \\
		s < 1, &\text{if}& R_\a \notin \m{Q} \text{ and } s_{\a,\b} \geq 1, \vspace{0.1in} \\
		s < \frac{1 + s_{\a,\b}}{2}, &\text{if}& R_\a \notin \m{Q} \text{ and } s_{\a,\b} < 1.
	\end{array}
	\right.\]
\end{theorem}

Combining the results in Theorem \ref{thm2} and Theorem \ref{thm4}, we conclude that for the initial value problem (\ref{mMajdaBiello}), with $\a\in(0,4)\setminus\{1\}$ and $\b\neq 0$, in the space $H^{s}_{0}(\mathbb{T})\times H^{s}(\mathbb{T})$, the critical indices are 
\be\label{ci_res}
	s^{*}(\a, \beta) := \left\{
	\begin{array}{lcl}
		\frac12, &\text{if}& R_\a \in \m{Q}, \vspace{0.1in} \\
		1, &\text{if}& R_\a \notin \m{Q} \text{ and } s_{\a,\b} \geq 1, \vspace{0.1in} \\
		\frac{1 + s_{\a,\b}}{2}, &\text{if}& R_\a \notin \m{Q} \text{ and } s_{\a,\b} < 1,
	\end{array}\right.
\ee
where $s_{\a,\b}$ is given in (\ref{s_0}). Meanwhile, for any fixed $\b\neq 0$, $s^{*}(\a, \b) = \frac12$ for almost every $\a\in(0,4)\setminus\{1\}$.

Now applying this result to (\ref{mmMajdaBirllo}), we demonstrate the effect of the mean of $u_0$ on the critical index for (\ref{GeneralMajdaBiello}) with $\a\in(0,4)\setminus\{1\}$ as below.

\begin{cor}\label{Cor_MBres_ci}
	Consider the initial value problem \eqref{GeneralMajdaBiello} with $\alpha \in (0,4)\setminus \{1\}$ in the space $H^s_{\b}(\mathbb{T})\times H^s(\mathbb{T})$, where $\b\neq 0$. Define $c_1$, $c_2$, $\lambda$ and $s_{\a,\b}$ as in (\ref{s_0_aux}) and (\ref{s_0}). Then the critical index of the analytically LWP for (\ref{GeneralMajdaBiello}) in the space $H^s_{\b}(\mathbb{T})\times H^s(\mathbb{T})$  is 
	\[
		s_{\b}^{*}(\a) := \left\{
		\begin{array}{lcl}
			\frac12, &\text{if}& R_\a \in \m{Q}, \vspace{0.1in} \\
			1, &\text{if}& R_\a \notin \m{Q} \text{ and } s_{\a,\b} \geq 1,  \vspace{0.1in}\\
			\frac{1 + s_{\a,\b}}{2}, &\text{if}& R_\a \notin \m{Q} \text{ and } s_{\a,\b} < 1.
		\end{array}\right.
	\]
	Moreover, for any fixed $\b\neq 0$, $s_{\b}^{*}(\a) = \frac12$ for almost every $a \in (0,4)\setminus\{1\}$.
\end{cor}

We would like to point out that when $\a\in(0,4)\setminus\{1\}$ and $\b=0$, Oh \cite{Oh2009} found that the critical index was $1$ if $R_\a\in\m{Q}$. Our result in Corollary \ref{Cor_MBres_ci} reduces the critical index to $\frac12$ when $R_\a\in\m{Q}$ and $\b\neq 0$. 

Next, we would like to briefly discuss the main ingredients in this paper. As the standard treatment for dispersive equations, the key is to study the resonance functions associated with the nonlinear terms $v v_x$ and $(uv)_{x}$ in (\ref{mMajdaBiello}). For example, adopting the notations in \cite{Oh2009}, the resonance function associated with the term $v v_x$ in (\ref{mMajdaBiello}) is 
\be\label{H_res_fn}
	H(\xi, \xi_1, \xi_2) = \xi^3 - (\alpha\xi_{1}^3 - \beta\xi_{1}) - (\alpha\xi_{2}^3 - \beta\xi_{2}), 
	\quad \forall\, (\xi, \xi_1, \xi_2) \in \mathcal{A},
\ee
where 
\[\mathcal{A} := \big\{ (\xi, \xi_1, \xi_2) \in \m{Z}^{3}: \xi = \xi_1 + \xi_2 \big\}.\]
Please see (\ref{res_fn_general}) for the general definition of a resonance function. Fixing $\xi$ and substituting $\xi_2 = \xi - \xi_1$, the function $H$ in (\ref{H_res_fn}) can be regarded as a function of $\xi_1$, denoted as $H^{\xi}(\xi_1)$:
\begin{align}
	H^{\xi}(\xi_{1}) &= -3\alpha\xi \left(\xi_{1}^2 - \xi \xi_1 + \frac{\alpha-1}{3\alpha}\xi^2 - \frac{\beta}{3\alpha}\right)  \label{H_xi} \\
    &= -3\a\xi\bigg[ (\xi_1 - c_1 \xi)(\xi_1 - c_2 \xi) - \frac{\b}{3\a} \bigg], \label{H_res_root}
\end{align}
where $c_1$ and $c_2$ are given in (\ref{s_0_aux}).

Similarly, the resonance function associated with the term $(uv)_{x}$ is 
\be\label{H_tilde}
	\wt{H}(\xi, \xi_1, \xi_2) = (\a\xi^3 - \b\xi) - \xi_{1}^3 - (\alpha\xi_{2}^3 - \beta\xi_{2}), 
	\quad \forall\, (\xi, \xi_1, \xi_2) \in \mathcal{A}.
\ee
For fixed $\xi_1$, $\wt{H}$ can be regarded as a function of $\xi$ and can be factored as 
\be\label{H_tilde_root}
	\wt{H}(\xi, \xi_1, \xi_2) = \wt{H}^{\xi_1}(\xi) := 3\a\xi_1 \bigg[ (\xi - c_1 \xi_1)(\xi - c_2 \xi_1) - \frac{\b}{3\a} \bigg].
\ee
We remark that preserving the same coefficients $c_1$ and $c_2$ in both (\ref{H_res_root}) and (\ref{H_tilde_root}) is the main reason why we fix $\xi$ for $H$ while fixing $\xi_1$ for $\wt{H}$. In fact, if we also fix $\xi$ for $\wt{H}$, then it will produce different coefficients $d_1$ and $d_2$ for the decomposition of $\wt{H}(\xi, \xi_1, \xi_2)$, see e.g. equation (19) in \cite{Oh2009}. Since $H^{\xi}(\xi_1)$ in (\ref{H_res_root}) and $\wt{H}^{\xi_1}(\xi)$ in (\ref{H_tilde_root}) share the same coefficients $c_1$ and $c_2$, the treatments for these two terms are very similar, so we will only focus on the analysis of $H^{\xi}(\xi_1)$ in the following illustration.

When $ \alpha = 4 $, $ h^{\xi}(\xi_{1}) $ has a repeated root: $\xi_1 = \xi/2$, which implies that $c_1 = c_2 = \frac12$ and 
\be\label{H_res_fn_cpt}
	H^{\xi}(\xi_{1}) = -3\xi\bigg[\left(2\xi_1 - \xi\right)^2 - \frac{\beta}{3}\bigg].
\ee
If $ \beta = 0 $, then there are infinitely many pairs $(\xi, \xi_1)\in\mathbb{Z}^2 $ such that the resonance function $H^{\xi}(\xi_{1})$ is zero, which makes it difficult to compensate for the loss of derivatives in the space. If $ \beta = 3n^2 $ with $ n\in\mathbb{N}\setminus\{0\}$, we have
\begin{equation}
H^{\xi}(\xi_{1}) = -3\xi(2\xi_1-\xi-n)(2\xi_1-\xi+n),
\end{equation}
then there also exist infinitely many pairs $(\xi, \xi_1)\in\mathbb{Z}^2 $ such that the resonance function vanishes. Only when $ \beta\neq 3n^2 $ for all $n\in\m{N}$, $\left(2\xi_1 - \xi\right)^2 - \frac{\beta}{3}$ never vanishes, which makes it possible to improve the well-posedness index of the problem, as stated in Theorem \ref{thm1}.

Next, when $ \alpha \in (0,4) \setminus \{1\} $, we write $H^{\xi}(\xi_1)$ in (\ref{H_res_root}) as $H^{\xi}(\xi_1) = -3\a\xi \big[ h^{\xi}(\xi_1) - \frac{\b}{3\a} \big]$,
where 
\[
h^{\xi}(\xi_1):= (\xi_1 - c_1 \xi)(\xi_1 - c_2\xi).
\]
Suppose $|\xi| = N \gg 1$. If $\xi_1$ is away from both $c_1\xi$ and $c_2\xi$, then $|h^{\xi}(\xi_1)| \sim N^2 \gg \frac{|\b|}{3\a}$, so $|H^{\xi}(\xi_1)|$ possesses a large lower bound $\sim N^3$. If $\xi_1$ is near $c_1\xi$, then $|\xi_1 - c_2 \xi| \sim (c_1 - c_2)N$ and meanwhile, it follows from the key step in \cite{Oh2009} that 
\[
    |\xi_1 - c_1\xi| = |\xi|\Big| c_1 - \frac{\xi_1}{\xi} \Big| \gs |\xi| \, \frac{1}{|\xi|^{2+\nu(c_1)+\eps}} \sim N^{-1-\nu(c_1)-\veps},
\]
where $\nu(c_1)$ is called the minimal type index of $c_1$ (see Definition 1 and 2 in \cite{Oh2009}). From Diophantine approximation theory, $\nu(c_1)\geq 0$ and $\nu(c_1) = 0$ for almost every $c_1$. When $\nu(c_1) = 0$, $|h^{\xi}(\xi_1)| \gs N^{-\eps}$ which provides a lower bound $N^{1-\veps}$ for $|H^{\xi}(\xi_1)|$ if $\b = 0$. However, in the current situation where $\b\neq 0$, the lower bound $N^{-\veps}$ of $|h^{\xi}(\xi_1)|$ does not dominate $\frac{|\b|}{3\a}$, so this argument can not provide a lower bound for $|H^{\xi}(\xi_1)|$. In order to resolve this issue, one has to discover a finer structure of the resonance function $H^{\xi}(\xi_1)$. In fact, when $|\xi|$ is so large that
\be\label{large_xi}
    |\xi|^2 \geq \frac{24 |\b|}{12-3\a}, 
\ee
then we can incorporate $\frac{\b}{3\a}$ into $h^{\xi}(\xi_1)$ to decompose the resonance function $H^{\xi}(\xi_1)$ below: 
\be\label{H_xi_decom2}
	H^{\xi}(\xi_{1}) = -3\alpha\xi(\xi_1-x_1)(\xi_{1}-x_2),
\ee
where
\begin{equation}
	x_1 = \frac{1}{2}\xi+\frac{1}{6}\sqrt{(R_{\alpha}\xi)^2+\frac{12\b}{\alpha}}, 
	\qquad 
	x_2 = \frac{1}{2}\xi-\frac{1}{6}\sqrt{(R_{\alpha}\xi)^2+\frac{12\b}{\alpha}}, 
	\qquad 
	R_{\alpha} = \sqrt{12/\alpha-3}.
\end{equation}
The case $\xi>0$ and the case $\xi<0$ can be handled similarly, so let us focus on the case $\xi>0$. Due to the assumption (\ref{large_xi}), we can expand $x_1$ and $x_2$ in terms of the order of $\xi$ as follows:
\be\label{xi_decom}
    x_1 = c_1\xi + \frac{\lam}{\xi} + Q_1(\xi), \qquad
    x_2 = c_2\xi - \frac{\lam}{\xi} + Q_2(\xi),
    \qquad 
    \lam := \frac{\b}{\a R_\a}, 
\ee
where $c_1 = \frac12 + \frac{R_\a}{6}$ and $c_2 = \frac12 - \frac{R_\a}{6}$ are as defined in (\ref{s_0_aux}), and 
\[
    |Q_{j}(\xi)| \leq \frac{12\lam^2}{R_\a |\xi|^3} = O(|\xi|^{-3}), \qquad\, j=1,2.
\]

The challenging part in the analysis is near the resonance set for $H^{\xi}$ in (\ref{H_xi_decom2}), that is when $\xi_1$ is near $x_1$ or $x_2$. Without loss of generality, we consider the region where $\xi_1$ is near $x_1$. In this region, $\xi_1$ is also near $c_1 \xi$ based on (\ref{xi_decom}), so 
\[
    |\xi_1 - x_2| \sim |c_1 - c_2| \xi \sim R_\a \xi.
\]
Therefore, $H^{\xi}(\xi_1) \sim \xi^2 (\xi_1 - x_1)$. Then according to the expression for $x_1$ in (\ref{xi_decom}), 
\[
    H^{\xi}(\xi_1) \sim \xi^3\bigg[\frac{\xi_{1}}{\xi} - c_1 - \frac{\lam}{\xi^2} + O\bigg( \frac{1}{\xi^4} \bigg)\bigg]. 
\]
This inspires us to consider the estimate of the lower bound of $ \big|c_1-\frac{\xi_{1}}{\xi}+\frac{\lam}{\xi^2}\big| $; whereas when $ \beta = 0 $, we only need to consider $ \big|c_1-\frac{\xi_{1}}{\xi}\big| $. When $\b\neq 0$, this is similar to the Diophantine approximation but with an extra term $ \frac{\lam}{\xi^2} $. Thus, we need to obtain the best lower bound for $\big|c_1 - \frac{\xi_1}{\xi} + \frac{\lam}{\xi^2}\big|$ with respect to $\xi_1,\xi\in\m{Z}$, where $ \lam\neq 0$ and $|\xi|$ is large. Therefore, based on the concept of the minimal type index, see e.g. (Definition 1 and Definition 2 in \cite{Oh2009}), we generalize that concept by incorporating biases as follows. 

\begin{mydef}\label{Def, mbti}
	A real number $\rho$ is said to be of $\gamma$-biased type $\nu$ if there exist positive constants $K = K(\rho, \nu, \gamma)$ and $N = N(\rho,\gamma)$ such that the inequality
	\[
		\left|\rho - \frac{m}{n}+\frac{\gamma}{n^2}\right| \ge \frac{K}{|n|^{2+\nu}}
	\]
	holds for all $(m,n) \in \mathbb{Z}^2$ with $|n|>N$. In addition, 
	\[
		\nu_{\gamma}(\rho) := \inf\{\nu \in \mathbb{R} : \rho \text{ is of $\gamma$-biased type } \nu\} 
	\]
	is called the minimal $\gamma$-biased type index of $\rho$, where the infimum is understood as $\infty$ if $\{\nu \in \mathbb{R} : \rho \text{ is of $\gamma$-biased type } \nu\}$ is empty.
\end{mydef}

In Proposition \ref{munu}, we find that for any fixed $\gamma$, $\nu_\gamma(\rho) = 0$ for almost every $ \rho \in \mathbb{R}$. This implies that regardless of the value of the mean $\b$ of $u_0$, the critical indices are the same for almost every $\a\in(0,4)\setminus\{1\}$. 
However, when $R_\alpha \in \mathbb{Q} $, compared with the known fact $\nu_{0}(c_1) = \nu_{0}(c_2) = \infty$, we show that $\nu_{\lambda}(c_1)  = \nu_{\lambda}(c_2)= 0$ as long as $\lam\neq 0$ (which is equivalent to $\b\neq 0$ since $\lam = \frac{\b}{\a R_\a}$). This leads to an enhancement of the critical well-posedness indices as shown in Theorem \ref{thm2} and Corollary \ref{Cor_MBres_ci}.

This paper is organized as follows. Section \ref{Sec, pre} introduces the notation and presents some useful linear estimates. Next, we prove Theorem \ref{thm1} for the most resonance case $\a = 4$ in Section \ref{Sec, pf_thm1}, and prove Theorem \ref{thm2} for the other resonance cases $\a\in(0,4)\setminus\{1\}$ in Section \ref{Sec, pf_thm2}. The key ingredients in the proofs in Section \ref{Sec, pf_thm1} and Section \ref{Sec, pf_thm2} are the bilinear estimates and the theory of the biased minimal type indices. Finally, the ill-posedness results, namely Theorem \ref{thm3} and Theorem \ref{thm4}, are justified in Section \ref{Sec, ip}.

\section{Preliminary}\label{Sec, pre}
It is well-known that the well-posedness problems of KdV type equations in $H^{s}$ are subcritical if $s>-\frac32$. In such cases, we can perform some invariant scaling to the equations \eqref{mMajdaBiello} such that the scaled initial data is small in $H^{s}$, which can be helpful in establishing the well-posedness.

Let $\sigma\geq 1$ and perform the scaling to the equations \eqref{mMajdaBiello} by $(u,v) \longmapsto (u^{\sigma}, v^{\sigma})$ defined as follows:
\begin{equation}\label{Scaling_form}
(u^{\sigma}, v^{\sigma})(x,t)= \frac{1}{\sigma^2} (u, v)\left(\frac{x}{\sigma},\frac{t}{\sigma^3}\right).
\end{equation}
Then $u^{\sigma}$ and $v^{\sigma}$ are defined on $\mathbb{T}_{\sigma} := [0,2\pi\sigma)$, and $(u^{\sigma}, v^{\sigma})$ satisfy the following equations
\begin{equation}\label{mMajdaBielloScaling}
\left\{\begin{aligned}
&u_t^{\sigma}+u_{xxx}^{\sigma}+v^{\sigma}v_x^{\sigma}= 0 , \\
&v_t^{\sigma}+ \alpha v_{xxx}^{\sigma}+\beta_{\sigma} v_x^{\sigma}+(u^{\sigma}v^{\sigma})_x=0 , \\
&(u^{\sigma},v^{\sigma})\mid_{t=0} = (u^{\sigma}_{0}, v^{\sigma}_{0}) \in H^s(\mathbb{T}_{\sigma})\times H^s(\mathbb{T}_{\sigma}),
\end{aligned}\right.
\quad\quad x\in \mathbb{T}_{\sigma}, \, t\in \mathbb{R},
\end{equation}
where 
\[ 
    \beta_{\sigma} := \frac{\beta}{\sigma^2} \qquad \text{and} \qquad
    (u^{\sigma}_{0}, v^{\sigma}_{0})(x) := \left(\frac{1}{\sigma^2}u_0\left(\frac{x}{\sigma}\right),\frac{1}{\sigma^2}v_0\left(\frac{x}{\sigma}\right)\right). 
\]

Based on $\m{T}_{\sigma}$, we denote its frequency space as
$\m{Z}_{\sigma} :=\big\{k \,\big|\, k=n/\sigma\,\,\, \text{for some $n\in\m{Z}$}\big\}$. 
For $1 \leq p < \infty$, we say $f \in L^p(\mathbb{Z}_\sigma)$ if 
\begin{equation*}
	\|f\|_{L^p(\mathbb{Z}_\sigma)} = \left(\int_{\mathbb{Z}_\sigma} |f(\xi)|^p \dd\xi\right)^{1/p} := \left(\frac{1}{2\pi\sigma}\sum_{\xi \in \mathbb{Z}_\sigma}|f(\xi)|^p\right)^{1/p} < \infty.
\end{equation*}
Then we define the Fourier transform on $\m{T}_{\sigma}$ to be 
\begin{equation*}
	\widehat{f}(\xi) = \int_0^{2\pi \sigma} e^{-ix\xi}f(x) \dd x, \quad\forall\,\xi\in\m{Z}_{\sigma}.
\end{equation*} 
Combined with the Fourier inversion transform, we can rewrite $f$ as follows:
\begin{equation*}
	f(x) = \int_{\mathbb{Z}_{\sigma}} e^{ix\xi}\widehat{f}(\xi) \dd \xi = \frac{1}{2\pi \sigma}\sum_{\xi \in \mathbb{Z}_\sigma} e^{ix\xi}\widehat{f}(\xi).
\end{equation*}

Since (\ref{mMajdaBielloScaling}) is a nonlinear system and (\ref{mMajdaBielloScaling})$_{2}$ contains a first order term $\beta_{\sigma} v^{\sigma}_{x}$, we first study the following homogeneous KdV equation (\ref{linear eq}) with a first-order term $\beta_\sigma w_{x}$:
\be\label{linear eq}
\left\{\begin{array}{ll}
	\p_{t}w + \alpha w_{xxx} + \beta_\sigma w_{x} = 0, & x\in\m{T}_{\sigma},\, t\in\m{R},\\
	w(0)=w_{0}\in H^{s}(\m{T}_{\sigma}),
\end{array}\right.\ee
where $\alpha \in \m{R}\setminus\{0\}$, $\b\in\m{R}$, $\sigma\geq 1$ and $\b_\sigma := \b/\sigma^2$.
The solution to (\ref{linear eq}) is given explicitly by 
\be\label{semigroup op}
w(x,t)=\int_{\m{Z}_{\sigma}}e^{i\xi x}e^{i(\alpha \xi^3-\beta_\sigma\xi)t}\widehat{w_{0}}(\xi)\dd \xi =: S_{\a,\b_\sigma}(t)w_{0}(x).\ee
When \( \alpha = 1 \) and \( \beta = 0 \), the corresponding semigroup operator $S_{\a,\b_\sigma}$ is simply denoted as \( S(t) \).

The solutions of the KdV equation (\ref{KdV}) are usually studied in the Fourier restriction spaces which were originally introduced in \cite{Bourgain93, KPV96SahrpIndex}. For the case on the real line, the KdV equation was investigated in the Fourier restriction spaces $X^{s,b}$ with $b > \frac{1}{2}$. But for the periodic case, $b$ has to be chosen as $\frac{1}{2}$, which unfortunately does not guarantee $X^{s,\frac12}$ to lie in $C(\mathbb{R}_t; H^s_x )$. Thus, we need to consider the problem in $Y^s$, which is an adaptation of $X^{s,\frac12}$ such that it belongs to the space $C(\mathbb{R}_t; H^s_x)$, see e.g. \cite{Iterm03}.

Now for the KdV-type equation (\ref{linear eq}) with a first-order term $\beta_\sigma w_{x}$, we
introduce generalized function spaces in the following definition which take the first order term into effect.
\begin{mydef}[\cite{Bourgain93,KPV96SahrpIndex,Iterm03}]
\label{Def, FR space}
	For any $\a, \b, \,s,\,b,\,\sigma\in\m{R}$ with $\a\neq 0$ and $\sigma\geq 1$, the Fourier restriction spaces, $X^{s,b}_{\a,\b_\sigma}$, are defined to be the completion of the Schwartz space $\mathscr{S}(\m{T}_{\sigma}\times\m{R})$ with respect to the norm
	\be\label{FR norm}
\|\omega\|_{X^{s,b}_{\alpha,\beta_\sigma}} := \left\| \langle \xi \rangle^s \langle \tau - \alpha \xi^3 + \beta_\sigma \xi \rangle^b \wh{w}(\xi,\tau) \right\|_{L^2(\mathbb{Z}_\sigma \times \mathbb{R})},\ee
	where $\wh{w}$ refers to the space-time Fourier transform of $w$ and     
    \[\b_\sigma := \b/\sigma^2.\]  
    Similarly, the spaces $Y^s_{\alpha,\beta_\sigma}$ are defined in the same way with respect to the following norm:
\begin{equation}\label{work_space}
	\|\omega\|_{Y^s_{\alpha,\beta_\sigma}} := \|\omega\|_{X^{s,\frac{1}{2}}_{\alpha,\beta_\sigma}} + \|\langle\xi\rangle^s \widehat{\omega}(\xi,\tau)\|_{L^2_\xi(\mathbb{Z}_\sigma)L^1_\tau(\mathbb{R})}.
\end{equation}
\end{mydef}
For technical needs in controlling the $Y^s_{\alpha,\beta_\sigma}$ norm of the Duhamel terms (see \eqref{Duhamel_Estimate}), we define the spaces $Z^s_{\alpha,\beta_\sigma}$ via the following norm
\begin{equation}\label{auxillaryspace}
    \|\omega\|_{Z^s_{\alpha,\beta_\sigma}} := \|w\|_{X^{s,-\frac{1}{2}}_{\alpha,\beta_\sigma}} + \|\langle\xi\rangle^s\langle\tau-\alpha\xi^3+\beta_\sigma\xi\rangle^{-1}\widehat{\omega}(\xi,\tau)\|_{L^2_\xi(\mathbb{Z_\sigma})L^1_\tau(\mathbb{R})}.
\end{equation}

For notational convenience, $\|\cdot\|_{X^{s,b}_{1,0}}$ is denoted as $\|\cdot\|_{X^{s,b}_{\sigma}}$ in what follows, likewise, the same holds for $\|\cdot\|_{Y^s_{\sigma}}$and $\|\cdot\|_{Z^s_{\sigma}}$. On the other hand, we fix $\eta \in C^\infty_c(\mathbb{R})$ to be a smooth cutoff function supported on $[-2,2]$ with $\eta \equiv 1$ on $[-1,1]$. Furthermore, we write $A\lesssim B$ to mean $A\le C B$, where the constant $C$ may depend only on the parameters $s$, $\alpha$ and $\beta$.

Next, we list some necessary and standard lemmas about linear estimates, the proofs can be found in \cite{Bourgain93,Iterm03,Oh2007}.

\begin{lemma}\label{lemma1}
For any $s,\a,\b, \sigma\in\m{R}$ with $\a\neq 0$ and $\sigma\geq 1$, 
\begin{equation}\label{First_Linear_Estimate}
 \|\eta(t)S_{\alpha,\beta_\sigma}(t)f\|_{Y^{s}_{\alpha,\beta_\sigma}} \le C_1\|f\|_{H^s_x(\m{T}_{\sigma})},
\end{equation}
and for the Duhamel terms, 
\begin{equation}\label{Duhamel_Estimate}
\left\|\eta(t) \int_0^t S_{\alpha,\beta_\sigma}(t-t') g(t') \,dt'\right\|_{Y^s_{\alpha,\beta_\sigma}} \le C_2 \|g\|_{Z^s_{\alpha,\beta_\sigma}}.
\end{equation}
where the constants $C_1$ and $C_2$ depend only on $\eta$.
\end{lemma}

\begin{lemma}\label{lemma3}
For any $s,\a,\b, \sigma\in\m{R}$ with $\a\neq 0$ and $\sigma\geq 1$, there exists a universal constant $C$ such that 
\begin{equation}\label{L4L2_X0_14}
\|f\|_{L^4_t L^2_x(\m{T}_\sigma \times \m{R})} \le C\|f\|_{X^{0,\frac{1}{4}}_{\alpha,\beta_\sigma}}.
\end{equation}
\end{lemma}

\begin{lemma}\label{lemma4}
For any $s,\a,\b, \sigma\in\m{R}$ with $\a\neq 0$ and $\sigma\geq 1$, there exists a constant $C=C(\a)$ such that 
\begin{equation}\label{L4L4_X013}
  \|f\|_{L^4_{x,t}(\m{T}_\sigma \times \m{R})} \le C  \|f\|_{X^{0,\frac{1}{3}}_{\alpha,\beta_\sigma}}.
\end{equation} 
\end{lemma}
By the Plancherel identity and duality, we handle bilinear estimates in the sense of $ L^2 $ (see e.g. \cite{Tao01}).

\begin{lemma}\label{lemma5}
Given $s$, $b$, $r$, $\{(\alpha_i, \beta_i)\}_{1 \leq i \leq 3}$ and $\sigma\geq 1$, the bilinear estimate 
\begin{equation*}
\left\| \partial_x (w_1 w_2) \right\|_{Z^{s}_{\alpha_3, \beta_3}} \le C\sigma^r\| w_1 \|_{Y^{s}_{\alpha_1, \beta_1}} \| w_2 \|_{Y^{s}_{\alpha_2, \beta_2}}, \quad \forall w_1, w_2 \in \mathscr{S}(\m{T}_\sigma\times\mathbb{R}).
\end{equation*}
holds if the following two estimates hold,
\begin{equation}\label{Duality_Equ_Form}
\int_{\Lambda} \frac{\xi_{3}\left\langle\xi_{3}\right\rangle^{s} \prod_{i=1}^{3} f_{i}\left(\xi_{i}, \tau_{i}\right)}{\left\langle\xi_{1}\right\rangle^{s}\left\langle\xi_{2}\right\rangle^{s}\left\langle L_{1}\right\rangle^{\frac{1}{2}}\left\langle L_{2}\right\rangle^{\frac{1}{2}}\left\langle L_{3}\right\rangle^{\frac{1}{2}}} \le C\sigma^r \prod_{i=1}^{3}\left\|f_{i}\right\|_{L^2}, \quad \forall f_{1}, f_{2}, f_{3} \in \mathscr{S}\left(\m{T}_\sigma\times\mathbb{R}\right),
\end{equation}
and 
	\begin{equation}
    \label{weighted l2, general, aux}
	\bigg\| \frac{1}{\la L_{3}\ra}\i_{\Lambda(\xi_{3},\tau_{3})}\frac{\xi_{2}\la \xi_{3}\ra^{s}}{\la \xi_{1}\ra^{s}\la \xi_{2}\ra^{s}}\prod_{i=1}^{2}\frac{f_{i}(\xi_{i},\tau_{i})}{\la L_{i}\ra^{\frac{1}{2}}}\bigg\|_{L^{2}_{\xi_{3}}L^{1}_{\tau_{3}}}
	\leq C\sigma^{r}\prod_{i=1}^{2}\| f_{i}\|_{L^{2}}, \quad\forall\, f_1,f_2\in\mathscr{S}\left(\m{T}_\sigma\times\mathbb{R}\right),
    \end{equation}
where
\begin{equation*}
L_{i}=\tau_{i}-\phi^{\alpha_{i}, \beta_{i}}\left(\xi_{i}\right) = \tau_i-\alpha_i\xi_i^3+\beta_{i}\xi_{i}, \quad i=1,2,3,
\end{equation*}
and
\begin{equation*}
\Lambda := \Big\{ (\vec{\xi}, \vec{\tau}) \in \mathbb{R}^6 : \sum_{i=1}^{3} \xi_i = \sum_{i=1}^{3} \tau_i = 0 \Big\},\quad \Lambda(\xi_3,\tau_3) := \left\{(\xi_1,\xi_2,\tau_1,\tau_2)\in\m{T}^2_\sigma\times \m{R}^2:  (\vec{\xi}, \vec{\tau}) \in \Lambda  \right\}.
\end{equation*} 
\end{lemma}

\section{Proof of Theorem \ref{thm1}}
\label{Sec, pf_thm1}
In this section, we first present the crucial bilinear estimate for proving Theorem \ref{thm1}, and then give a brief proof of Theorem \ref{thm1} by applying the standard contraction mapping method, which establishes the analytically local well-posedness of the system \eqref{mMajdaBiello}. After that, we complete the proof by justifying the crucial bilinear estimates. Since Theorem \ref{thm1} deals with the system \eqref{mMajdaBiello} which is equivalent to (\ref{mMajdaBielloScaling}) after scaling, the following bilinear estimate in Proposition \ref{Prop1} is presented based on the scaled system (\ref{mMajdaBielloScaling}).

\begin{proposition}\label{Prop1} 
Let $\b\in\m{R}$ and $\sigma\geq 1$. Define $s^{*}(4, \b)$ as 
\begin{equation}\label{Prop1_Index}
s^{*}(4, \beta)=\left\{\begin{array}{cl}
	1, & \quad \beta = 3n^2, n\in \mathbb{N}, \\
	1/2, &\quad others.
\end{array}\right.
\end{equation}
Then for any $s \ge s^*(4,\beta)$, the bilinear estimate (\ref{Bilinear1}) holds for any $w_1$ and $w_2$ in the Schwartz space $\mathscr{S}(\mathbb{T}_\sigma \times \mathbb{R})$.
\begin{equation}\label{Bilinear1}
\|\partial_x(w_1w_2)\|_{Z^s_\sigma} \le C_3\sigma\|w_1\|_{Y^s_{4,\beta_{\sigma}}}\|w_2\|_{Y^s_{4,\beta _{\sigma}}}, 
\end{equation}
where $C_3 = C_3(s,\beta)$.
Moreover, by further assuming the mean zero condition for $w_1$, that is, $\mathscr{F}_{x}{w_1}(0, t) = 0$ for any $t$, then the bilinear estimate (\ref{Bilinear2}) also holds for any $s \ge s^*(4,\beta)$.
\begin{equation}\label{Bilinear2}
\|\partial_x(w_1w_2)\|_{Z^s_{4,\beta_{\sigma}}} \le C_3\sigma\|w_1\|_{Y^s_\sigma}\|w_2\|_{Y^s_{4,\beta_{\sigma}}},
\end{equation}
where $C_3 = C_3(s,\beta)$.
\end{proposition}

\subsection{Contraction Mapping Argument}\label{Sec, CM}

The proof of the well-posedness result via the contraction mapping principle is standard; see \cite{KPV96SahrpIndex,Iterm03}, so we will only sketch the main steps in the proof of Theorem \ref{thm1}. In addition, the proofs for the case when $\beta = 3n^2$ are very similar to that when $\beta\neq 3n^2$, with the only difference being the use of different cases in (\ref{Prop1_Index}) in Proposition \ref{Prop1}. Hence, we will present only the proof of Theorem \ref{thm1} for the case $\beta \neq 3n^2, n\in \mathbb{N}$.

For $\beta \neq 3n^2$ and $\sigma \geq 1$, we study the LWP of (\ref{mMajdaBielloScaling}) in $H^{s}_{0}(\mathbb{T})\times H^{s}(\mathbb{T})$ for any $s\geq \frac{1}{2}$. For convenience, we denote 
\[
    r = \|(u^\sigma_0,v^\sigma_0)\|_{H^s_0(\mathbb{T}_\m{\sigma})\times H^s(\mathbb{T}_\m{\sigma})}.
\]
Then by direct computation, we have
\be\label{r_small} 
    r \leq \sigma^{-\frac32} \|(u_0, v_0)\|_{H^s_0(\mathbb{T})\times H^s(\mathbb{T})},
\ee
so $r$ can be arbitrarily small as long as $\sigma$ is chosen to be sufficiently large.

Define an operator $\Gamma$ on $Y_\sigma^{s}\times Y^s_{4,\beta_\sigma}$ to be 
$\Gamma(u^\sigma,v^\sigma) = (\Gamma_1(u^\sigma,v^\sigma),\Gamma_2(u^\sigma,v^\sigma))$, where
\begin{equation*}
	\left\{
	\begin{aligned}
		&\Gamma_1(u^\sigma,v^\sigma) = \eta(t)S(t)u^\sigma_0 - 2\eta(t)\int_0^t S(t-t')v^\sigma v^\sigma_x(t')\dd t', \\
		&\Gamma_2(u^\sigma,v^\sigma) = \eta(t)S_{4,\beta_\sigma}(t)v^\sigma_0 - \eta(t)\int_0^t S_{4,\beta_\sigma}(t-t')(u^\sigma v^\sigma)_x(t')\dd t'.
	\end{aligned}\right.
\end{equation*}
From Lemma \ref{lemma1}, we have
\begin{equation}\label{Tempcon}
\|\Gamma(u^{\sigma},v^{\sigma})\|_{Y_\sigma^{s}\times Y^{s}_{4,\beta_\sigma}}\le 2C_1r+3C_2\left(\|\partial_x(v^{\sigma} v^{\sigma})\|_{Z^s(\mathbb{T}_\sigma \times \mathbb{R})}+\|\partial_x(u^\sigma v^\sigma)\|_{Z^s_{4,\beta_{\sigma}}(\mathbb{T}_\sigma \times \mathbb{R})}\right),
\end{equation}
where $C_1$ comes from \eqref{First_Linear_Estimate} and $C_2$ comes from \eqref{Duhamel_Estimate}. Then, using the second case in (\ref{Prop1_Index}) in Proposition \ref{Prop1}, we deduce that
\[
\text{LHS of } \eqref{Tempcon} \le 2C_1r+3C_2C_3\sigma\|(u^{\sigma},v^\sigma)\|_{Y_\sigma^s\times Y^{s}_{4,\beta_\sigma}}^2,
\]
where $C_3$ is from Proposition \ref{Prop1}. 

Consider the ball
\[
B(3C_1r) = \left\{ (u^{\sigma},v^\sigma)\in Y_\sigma^s\times Y^{s}_{4,\beta_\sigma} : \|(u^{\sigma},v^\sigma)\|_{Y_\sigma^s\times Y^{s}_{4,\beta_\sigma}} \leq 3C_1r \right\}.
\]
Thanks to (\ref{r_small}), there exists a constant $\sigma_0$ such that for any $\sigma\geq \sigma_0$, 
\be\label{small_sigr}
    27C_1C_2C_3\sigma r\le 1,
\ee
which implies that 
\begin{equation*}
\|\Gamma(u^{\sigma},v^{\sigma})\|_{Y_\sigma^{s}\times Y^{s}_{4,\beta_\sigma}}\le 2C_1r+3C_2C_3\sigma(3C_1r)^2 = 2C_1r+27C_1^2C_2C_3r^2\le 3C_1r.
\end{equation*}
This guarantees $\Gamma(B(3C_1r))\subseteq B(3C_1r)$.

Meanwhile, for any $(u^\sigma,v^\sigma), (\widetilde{u}^\sigma,\widetilde{v}^\sigma) \in B(3C_1r)$, it is readily seen that
\[
\|\Gamma(u^\sigma,v^\sigma) - \Gamma(\widetilde{u}^\sigma,\widetilde{v}^\sigma)\|_{Y_\sigma^s\times Y^{s}_{4,\beta_\sigma}}\leq 18C_1 C_2 C_3\sigma r\|(u^\sigma,v^\sigma) - (\widetilde{u}^\sigma,\widetilde{v}^\sigma)\|_{Y_\sigma^s\times Y^{s}_{4,\beta_\sigma}}.
\]
Thanks to the estimate (\ref{small_sigr}), we verify that $\Gamma$ is a contraction mapping, which implies the analytically local well-posedness of the system (\ref{mMajdaBielloScaling}) when $\sigma$ is sufficiently large. Finally, by scaling back to the system (\ref{mMajdaBiello}), we proved Theorem \ref{thm1}.

\subsection{Proof of Proposition \ref{Prop1}}
\label{Sec, pf_bilin_dres}

The proofs of (\ref{Bilinear1}) and (\ref{Bilinear2}) in Proposition \ref{Prop1} are very similar, so we will only provide details for (\ref{Bilinear1}) and then  illustrate why (\ref{Bilinear2}) needs the extra assumption that $w_1$ has zero mean.

For (\ref{Bilinear1}), by the definition of the $Z^{s}_\sigma$ space in (\ref{auxillaryspace}), the bilinear estimate \eqref{Bilinear1} is decomposed into the following two estimates:
\begin{eqnarray}
\|\partial_x(w_1w_2)\|_{X_\sigma^{s, -\frac{1}{2}}} \le C_3\sigma\|w_1\|_{Y^s_{4,\beta_{\sigma}}}\|w_2\|_{Y^s_{4,\beta _{\sigma}}}, \label{Bilinear1, p1} \\
\|\langle\xi\rangle^s\langle\tau-\alpha\xi^3+\beta_\sigma\xi\rangle^{-1}\widehat{\p_x(w_1 w_2)}(\xi,\tau)\|_{L^2_\xi(\mathbb{Z_\sigma})L^1_\tau(\mathbb{R})}
\le C_3\sigma\|w_1\|_{Y^s_{4,\beta_{\sigma}}}\|w_2\|_{Y^s_{4,\beta _{\sigma}}}. \label{Bilinear1, p2}
\end{eqnarray}
Next, we will first prove (\ref{Bilinear1, p1}) and then justify (\ref{Bilinear1, p2}) by taking advantage of the established (\ref{Bilinear1, p1}).

For (\ref{Bilinear1, p1}), since the $X^{s,\frac12}_{4, \b_\sigma}$ norm is smaller than the $Y^{s}_{4, \b_\sigma}$ norm, it suffices to show that
\begin{equation}\label{BilinearFirstPart}
\|\partial_x(w_1w_2)\|_{X^{s,-\frac{1}{2}}_\sigma} \lesssim \sigma\|w_1\|_{X_{4,\beta_{\sigma}}^{s,\frac{1}{2}}}\|w_2\|_{X_{4,\beta_{\sigma}}^{s,\frac{1}{2}}}.
\end{equation}
Define the bilinear operator:
\begin{equation}\label{Bilinear_Operator}
\mathscr{B}_{s}(f_1,f_2)(\xi,\tau) = \frac{\xi \langle\xi\rangle^s}{\langle L\rangle^{1/2}}\iint\limits_{\substack{\xi_1+\xi_{2} = \xi \\ \tau_1+\tau_2 = \tau}}\frac{f_1(\xi_{1},\tau_1)f_2(\xi_{2},\tau_2)}{\langle\xi_1\rangle^s\langle\xi_2\rangle^s\langle L_1\rangle^{1/2}\langle L_2\rangle^{1/2}}\, d\xi_1 \, d\tau_1, \quad \forall\, \xi\in \m{Z}/\sigma, \, \tau\in\m{R},
\end{equation}
where $f_{1}, f_{2} \in \mathscr{S}\left(\m{T_\sigma}\times\mathbb{R}\right)$, and
\begin{equation}\label{Spital_Time_Forms_u_3n}
L = \tau - \xi^3, \quad L_1 = \tau_1 - 4 \xi_1^3+\frac{\beta}{\sigma^2} \xi_1,\quad L_2 = \tau_2 - 4\xi_2^3+\frac{\beta}{\sigma^2} \xi_2,
\end{equation}
By the definition of $X^{s,-\frac{1}{2}}_\sigma$ and making a suitable change of variable, the inequality \eqref{BilinearFirstPart} is equivalent to:
\begin{equation}\label{BilinearL2}
\|\mathscr{B}_s(f_1,f_2)\|_{L^2_{\xi,\tau}}\lesssim \sigma\|f_1\|_{L_{\xi_{1},\tau_1}^2}\|f_2\|_{L_{\xi_{2},\tau_2}^2}.
\end{equation}
Without loss of generality, we assume that both $f_1$ and $f_2$ are non-negative. According to (\ref{Spital_Time_Forms_u_3n}) and the constraint $\tau_1 + \tau_2 = \tau$, all temporal frequency variables will disappear if we subtract $L$ from $L_1 + L_2$, that is:
\[ L_1 + L_2 - L = \Big( -4\xi_1^3 + \frac{\beta}{\sigma^2} \xi_1 \Big) + \Big( -4\xi_2^3 + \frac{\beta}{\sigma^2} \xi_2\Big) + \xi^3 := H_{\sigma}(\xi, \xi_1, \xi_2),\]
where the function $H_{\sigma}$ is called the resonance function by convention.
By substituting $\xi_2 = \xi - \xi_1$, we obtain
\[
    H_{\sigma}(\xi, \xi_1, \xi_2) = -3\xi\left[(2\xi_{1}-\xi)^2-\frac{\beta}{3\sigma^2}\right].
\]
When $\xi$ is fixed, the above expression can be regarded as a function in $\xi_1$. We denote such a function to be $H_{\sigma}^{\xi}$ which can be rewritten as below:
\begin{equation*}
H_{\sigma}^{\xi}(\xi_{1})= - 3\xi\left[h^{\xi}(\xi_1)-\frac{\beta}{3\sigma^2}\right], \quad \text{where} \quad h^{\xi}(\xi_1) := (2\xi_1-\xi)^2.
\end{equation*}
When $\beta <0$, then $h^{\xi}(\xi_{1})-\frac{\beta}{\sigma^2}>0$ for all $\xi$ and $\xi_{1}$; When $\beta>0$ and $\beta \neq 3n^2$, $h^{\xi}(\xi_{1})-\frac{\beta}{\sigma^2}$ has two roots on $\m{R}$, neither of which belongs to $\m{Z}_{\sigma}$ since $ \sqrt{\frac{\beta}{3}} \notin \mathbb{Z} $. These are the key observations in the proof for the case $\b\neq 3n^2$. For the remaining case $\b = 3n^2$, there are infinitely many pairs of $(\xi,\xi_1)\in \m{Z}/\sigma\times\m{Z}/\sigma$ such that the resonance function $H_\sigma^{\xi}(\xi_1)$ is zero, which leads to the conclusion that the critical index needs to be $1$ instead of $\frac12$. 

\subsubsection{Case 1: \texorpdfstring{$\beta \neq 3n^2$, $n\in \mathbb{N}$}{Case 1: beta not equal to 3n squared, n in N}}
\label{Sec, pos_beta_nr}

\noindent\textbf{Proof of (\ref{Bilinear1, p1}) in Case 1}

\noindent\textbf{Case 1.1: $\beta>0$ and $\beta \neq 3n^2$.}
When $\b>0$, $H_\sigma^\xi(\xi_1)$ can be factored as follows:
\be\label{res_fun_fac}
    H_\sigma^\xi(\xi_1) = -3\xi\left[(2\xi_{1}-\xi)^2-\frac{\beta}{3\sigma^2}\right] = -3\xi \left( 2\xi_1 - \xi + \frac{1}{\sigma}\sqrt{\frac{\beta}{3}} \right)\left( 2\xi_1 - \xi - \frac{1}{\sigma}\sqrt{\frac{\beta}{3}} \right).
\ee

\noindent\textbf{Region (1):} $|\xi|\le 100{\sqrt{\beta}}+1$.

For $s \ge 0$, the following estimate holds for any $(\xi_1, \xi_2, \xi)\in (\m{Z}/\sigma)^{3}$ such that $\xi = \xi_1 + \xi_2$,
\be\label{small_weight}
\frac{|\xi|\langle\xi\rangle^s}{\langle\xi_{1}\rangle^s\langle\xi_2\rangle^s\langle L\rangle^{1/2}} \lesssim \frac{|\xi|\langle\xi\rangle^s}{\langle \xi\rangle^s\langle L\rangle^{1/2}} \lesssim 1. 
\ee
Then by H\"older inequality and Lemma \ref{lemma4}, we have
\begin{equation}\label{Region1}
\begin{aligned}
\|\mathscr{B}_s(f_1,f_2)\|_{L^2_{\xi,\tau}}&\lesssim
\Big\|\iint\limits_{\substack{\xi_1+\xi_{2} = \xi \\ \tau_1+\tau_2 = \tau}}\frac{f_1(\xi_{1},\tau_1)f_2(\xi_{2},\tau_2)}{\langle L_1\rangle^{1/2}\langle L_2\rangle^{1/2}}\dd\xi_1\dd\tau_1\Big\|_{L^2_{\xi,\tau}} = \|F_1F_2\|_{L^2_{x,t}}\\
&\le \|F_1\|_{L^4_{x,t}}\|F_2\|_{L^4_{x,t}}\lesssim\|F_1\|_{X^{0,\frac{1}{3}}_{4,\beta_{\sigma}}} \|F_2\|_{X^{0,\frac{1}{3}}_{4,\beta_{\sigma}}}\le \|f_1\|_{L^2_{x,t}} \|f_2\|_{L^2_{x,t}},
\end{aligned}
\end{equation}
where 
\begin{equation}\label{inverse1}
{F_1}(x,t) =\mathscr{F}^{-1}\{ \langle L_1\rangle^{-\frac{1}{2}} {f_1}(\xi_1,\tau_1)\},\quad {F_2}(x,t) = \mathscr{F}^{-1}\{\langle L_2\rangle^{-\frac{1}{2}} {f_2}(\xi_2,\tau_2)\}.
\end{equation}

\noindent\textbf{Region (2):} $|\xi|>1$ and $\Big|2\xi_1-\xi-\frac{1}{\sigma}\sqrt{\frac{\beta}{3}}\Big|\geq \frac{1}{10}|\xi|$.

In this case, we have 
$$
\Big|2\xi_1-\xi+\frac{1}{\sigma}\sqrt{\frac{\beta}{3}}\Big|\ge \frac{1}{10}|\xi|-\frac{2}{\sigma}\sqrt{\frac{\beta}{3}}\geq \frac{1}{20}|\xi|.
$$
Let
\begin{equation}\label{MAX_1_And_Resonance_3n}
    MAX_1 := \max\{\langle L\rangle,\langle L_1\rangle,\langle L_2\rangle\}.
\end{equation}
Then it follows from the decomposition (\ref{res_fun_fac}) that
\begin{equation*}
MAX_1\ge \frac{1}{3} \langle -L+L_1+L_2\rangle\ge |H_\sigma^\xi(\xi_1)| \gtrsim |\xi|^3.
\end{equation*}

\noindent \textbf{Region (2.1).} $\langle L\rangle = MAX_1$.

In this case, 
\[\frac{|\xi|\langle\xi\rangle^s}{\langle\xi_1\rangle^s\langle \xi_2\rangle^s} \frac{1}{\langle L\rangle^\frac{1}{2}} \lesssim \frac{\langle\xi\rangle^s}{\langle\xi_1\rangle^s\langle \xi_2\rangle^s} \le 1.\] 
Then the rest computation is similar to Region (1).

\noindent \textbf{Region (2.2).} $\langle L_1\rangle= MAX_1$ or $\langle L_2\rangle= MAX_1$.

Without loss of generality, we only discuss the case $\langle L_1\rangle= MAX_1$ since the other case that $\langle L_2\rangle= MAX_1$ is similar. By Lemma \ref{lemma5}, \eqref{BilinearFirstPart} is equivalent to
\begin{equation}\label{Duality3}
\left|\int \frac{\xi\langle\xi\rangle^s g(\xi,\tau)f_1(\xi_{1},\tau_1)f_2(\xi_{2},\tau_2)}{\langle\xi_1\rangle^s\langle\xi_2\rangle^s\langle L\rangle^{1/2}\langle L_1\rangle^{1/2}\langle L_2\rangle^{1/2}}\,d\xi_1 \,d\tau_1 \,d\xi \,d\tau \right|\lesssim \sigma \|f_1\|_{L^2_{\xi_1,\tau_1}}\|f_2\|_{L^2_{\xi_2,\tau_2}}\|g\|_{L^2_{\xi,\tau}}. 
\end{equation}
Then similar to the argument in Region (2.1), we have 
$\frac{|\xi|\langle\xi\rangle^s}{\langle\xi_1\rangle^s\langle \xi_2\rangle^s} \frac{1}{\langle L_1\rangle^\frac{1}{2}} \lesssim 1$. 
Therefore, by denoting 
\be\label{inverseL2}
F_1(x,t) = \mathscr{F}^{-1}\{f_1({\xi_1,\tau_1})\}, \quad {F_2}(x, t) =\mathscr{F}^{-1}\{\langle L_2\rangle^{-\frac{1}{2}}{f_2}(\xi_2, \tau_2)\},\quad {G}(x, t) = \mathscr{F}^{-1}\{\langle L\rangle^{-\frac{1}{2}}{g}(-\xi, -\tau)\},
\ee
it follows from Lemma \ref{lemma4} that
\be\label{st_est}
\begin{split}
LHS \text{ of } \eqref{Duality3}&\lesssim 
\left|\iint F_1(x,t)F_2(x,t)G(x,t) \,dx \,dt\right| 
\le \|F_1\|_{L^2_{x,t}} \|F_2\|_{L^4_{x,t}} \|G\|_{L^4_{x,t}}\\
&		
\le  \|F_1\|_{L^2_{x,t}} \|F_2\|_{X^{0,\frac{1}{3}}_{4,\beta_{\sigma}}} \|G\|_{X_\sigma^{0,\frac{1}{3}}} \le \|f_1\|_{L^2_{\xi_1,\tau_1}}\|f_2\|_{L^2_{\xi_2,\tau_2}}\|g\|_{L^2_{\xi,\tau}}.
\end{split}
\ee

\noindent\textbf{Region (3):} $|\xi|>1$ and $\big|2\xi_1-\xi-\frac{1}{\sigma}\sqrt{\frac{\beta}{3}} \big| \le \frac{1}{10}|\xi|$.

In this domain, it turns out that
\be\label{comp_size}
    \langle\xi_1\rangle \sim \langle\xi_2\rangle \sim \langle\xi\rangle
\ee
then we divide this domain into the following cases.

\noindent\textbf{Region (3.1):} $|2\xi_1-\xi-\frac{1}{\sigma}\sqrt{\frac{\beta}{3}}|\geq \frac1\sigma$ and $|2\xi_1-\xi+\frac{1}{\sigma}\sqrt{\frac{\beta}{3}}|\geq \frac1\sigma$.

At this time, the resonance function $H_\sigma^{\xi}(\xi_{1})$ satisfies
\[|H_\sigma^{\xi}(\xi_{1})|\geq \frac{1}{\sigma^2}|\xi|,\] 
so when $s\geq \frac{1}{2}$, we take advantage of (\ref{comp_size}) to deduce that
\begin{equation*}
\frac{|\xi|\langle\xi\rangle^s}{\langle\xi_1\rangle^s\langle\xi_2\rangle^s}\frac{1}{ {MAX_1}^{\frac{1}{2}}}\lesssim \sigma\frac{|\xi|^{\frac{1}{2}}}{\langle \xi\rangle^{\frac{1}{2}}}\le \sigma.
\end{equation*}
Then, similar to the argument for Region (2) in the previous discussion, dividing Region (3.1) into three subregions depending on whether $\la L\ra$, $\la L_1\ra$ or $\la L_2\ra$ attains $MAX_1$, we can verify the desired estimate.

\noindent\textbf{Region (3.2):} $\big|2\xi_1-\xi-\frac{1}{\sigma}\sqrt{\frac{\beta}{3}}\big|\le \frac1\sigma$ or $\big|2\xi_1-\xi+\frac{1}{\sigma}\sqrt{\frac{\beta}{3}}\big|\le \frac1\sigma$.

The resonance function can be rewritten as
\begin{equation*}
H_\sigma^{\xi}(\xi_1) = -\frac{3\xi}{\sigma^2}\left(2\widetilde{\xi}_1-\widetilde\xi-\sqrt{\frac{\beta}{3}}\right)\left(2\widetilde{\xi}_1-\widetilde\xi+\sqrt{\frac{\beta}{3}}\right),
\end{equation*}
where $\widetilde\xi_1 = \sigma\xi_{1},\widetilde{\xi}=\sigma\xi\in\mathbb{Z}$. Since $\beta \neq 3n^2$, which implies that $\sqrt{\frac{\beta}{3}}\notin \mathbb{Z}$, we have
\[ \theta\Big(\sqrt{\frac{\beta}{3}}\Big) \gs 1, \]
where $\theta(x)$ denotes the distance between $x$ and its nearest integer.
Hence,
\begin{equation*}
|H_\sigma^{\xi}({\xi_{1}})|\gtrsim \theta^2\Big(\sqrt{\frac{\beta}{3}}\Big) \frac{|\xi|}{\sigma^2}\gtrsim \frac{|\xi|}{\sigma^2}.
\end{equation*} 
The rest of the argument is the same as that for Region (3.1).

\bigskip
\noindent\textbf{Case 1.2: $\beta<0$.}

The proof for this case is basically similar to that for Case 1.1. We still divide the proof into three sub-cases.

\noindent\textbf{Region (1):} $|\xi|\le 1$.

In this case, it is easy to get that $\frac{|\xi|\langle \xi\rangle^s}{\langle \xi_1\rangle^s\langle\xi_2\rangle^s}\le 1$. Then the rest proof is the same as that for Region (1) in Case 1.1. We emphasize that the bound in this region is an absolute constant $1$ while the bound in Region (1) in Case 1.1 depends on $\b$.

\noindent\textbf{Region (2):} $|\xi|> 1$ and $|2\xi_1-\xi|\geq \frac{1}{10}|\xi|$.

Since $\beta<0$, it holds that 
\[
|H^{\xi}_\sigma(\xi_1)| = 3|\xi| \left[h^{\xi}(\xi_1)-\frac{\beta}{3\sigma^2} \right]\geq 3|\xi|(2\xi_1-\xi)^2\gtrsim |\xi|^3.
\]
Thanks to the negativity of $\b$, the above lower bound is valid for any $\xi$ and $\xi_1$. Compared with the estimate in Region (2) in Case 1.1, the lower bound of $|H^{\xi}_\sigma(\xi_1)|$ there requires that $|\xi|$ is larger than $\sqrt{\b}$. This is the reason why we choose different upper bounds for $\xi$ in Region (1) for Case 1.1 and Case 1.2.

\noindent\textbf{Region (3):} $|\xi|> 1$ and $|2\xi_1-\xi|< \frac{1}{10}|\xi|$.

In this case, $\langle\xi\rangle\sim\langle\xi_1\rangle\sim\langle\xi_2\rangle$ holds, and we can also obtain the lower bound for the resonance function
\begin{equation*}
|H_\sigma^\xi(\xi_1)| = \left|-3\xi\left[(2\xi_{1}-\xi)^2-\frac{\beta}{3\sigma^2}\right]\right|\gtrsim \frac{|\beta\xi|}{\sigma^2}.
\end{equation*}
Again, here we take advantage of the negativity of $\b$ to attain an effective lower bound without any restrictions. Although we drop the term $\xi(2\xi_1 - \xi)^2$ entirely in the above estimate, this does not lose much since $|2\xi_1 - \xi|$ is very small in Region (3). Once the lower bound of $|H_\sigma^\xi(\xi_1)|$ is achieved, the rest argument is similar to that for Region (3) in Case 1.1, hence is omitted.
\bigskip

\noindent \textbf{Proof of (\ref{Bilinear1, p2}) in Case 1}

For the second part, $L_\xi^2(\mathbb{Z}/\sigma)L_\tau^1(\mathbb{R})$, we will first give the proof for the case when $\b > 0$ but $\b\neq 3n^2$. Then as we discussed in Case 1.2, the argument for the case when $\b < 0$ is similar and actually simpler, and is therefore omitted. 

Since $X_{4,\beta_{\sigma}}^{s,\frac{1}{2}}$ norm is smaller than $Y_{4,\beta_{\sigma}}^{s}$ norm, (\ref{Bilinear1, p2}) boils down to 
\[
\|\langle \xi\rangle^{s}\langle L\rangle^{-1}\mathscr{F}_x(\partial_x(w_1w_2))\|_{L_\xi^2(\mathbb{Z}/\sigma)L_\tau^1(\mathbb{R})}\lesssim \sigma\|w_1\|_{X_{\alpha,\beta_{\sigma}}^{s,\frac{1}{2}}}\|w_2\|_{X_{\alpha,\beta_{\sigma}}^{s,\frac{1}{2}}},
\]
which can be reformulated as
\begin{equation}\label{Part21neq3n}
\|\langle L\rangle^{-\frac{1}{2}}\mathscr{B}_{s}(f_1,f_2)\|_{L_\xi^2(\mathbb{Z}/\sigma)L_\tau^1(\mathbb{R})}\lesssim \sigma\|f_1\|_{L_{\xi_{1},\tau_1}^2}\|f_2\|_{L_{\xi_{2},\tau_2}^2}.
\end{equation}
This part of the proof is standard, which relies on the Cauchy–Schwarz inequality to reduce the estimate (\ref{Part21neq3n}) to the established estimate (\ref{BilinearL2}). In the following, we choose $\eps$ as a small constant in $\big(0, \frac{1}{100}\big)$.

\noindent\textbf{Region (1)}: $|\xi|\leq 100\sqrt{\beta}+1$.

Since $0<\varepsilon <\frac{1}{100}$, we have
\begin{align*}
LHS \text{ of } \eqref{Part21neq3n} 
& \le \|\langle L\rangle^{-\frac{1}{2}-\varepsilon}\|_{L^{\infty}_{\xi} L^2_\tau} 
\|\langle L\rangle^{\varepsilon}\mathscr{B}_s(f_1,f_2)(\xi,\tau)\|_{L^{2}_{\xi} L^2_\tau}
\\
&\lesssim_{\varepsilon} \|\langle L\rangle^{\varepsilon}\mathscr{B}_s(f_1,f_2)(\xi,\tau)\|_{L^2_{\xi,\tau}}.
\end{align*}
So it remains to prove 
\be\label{reduce1}
\|\langle L\rangle^{\varepsilon}\mathscr{B}_s(f_1,f_2)(\xi,\tau)\|_{L^2_{\xi,\tau}}
\lesssim \sigma\|f_1\|_{L_{\xi_{1},\tau_1}^2}\|f_2\|_{L_{\xi_{2},\tau_2}^2},
\ee
which is analogous to (\ref{BilinearL2}) with the only difference being the extra term $\la L\ra^{\eps}$ on the left hand side of (\ref{reduce1}).
When $|\xi| \leq 100\sqrt{\b}+1$, the following inequality holds \begin{equation*}
\frac{|\xi|\langle\xi\rangle^s\langle L\rangle^{\varepsilon}}{\langle\xi_{1}\rangle^s\langle\xi_2\rangle^s\langle L\rangle^{\frac{1}{2}}} \lesssim \frac{|\xi|\langle\xi\rangle^s}{\langle \xi\rangle^s\langle L\rangle^{\frac{1}{2}-\varepsilon}} \lesssim \frac{1}{\langle L \rangle^{\frac{1}{3}}}\lesssim 1.
\end{equation*}
Compared with (\ref{small_weight}), the above estimate is stronger in that the left-hand side contains an extra term $\la L\ra^{\eps}$. Therefore, by the similar argument for Region 1 in Case 1.1, we obtain the desired result (\ref{reduce1}).

\noindent\textbf{Region (2)}: $|\xi|> 100\sqrt{\beta}+1$, and $MAX_1 = \langle L_1\rangle$ or $MAX_1 = \langle L_2 \rangle$.

When $|\xi|>1$ and $MAX_1 = \langle L_1\rangle$, the proof of \eqref{reduce1} is similar to that of \eqref{Bilinear1, p1} when $MAX_1 = \langle L_1\rangle$ in Region (2.2) and Region (3). The only difference is that we need to adjust the term $\la L\ra^{\frac12}$ in \eqref{Duality3} and \eqref{inverseL2} to be $\la L\ra^{\frac12 - \eps}$ due to the extra term $\la L\ra^{\eps}$ in \eqref{reduce1}. Since $\eps<\frac{1}{100}$, then $\frac12 - \eps > \frac13$ and the estimate \eqref{st_est} is still valid, which justifies \eqref{reduce1}.

When $MAX_1 = \langle L_2\rangle$, the proof for (\ref{reduce1}) is almost the same as that for $MAX_1 = \langle L_1\rangle$, and thus omitted.

\noindent\textbf{Region (3)}: $|\xi|>1$ and $MAX_1 = \langle L \rangle$.

\noindent\textbf{Region (3.1)}: $\langle L_1\rangle \geq \frac{1}{2}\langle L\rangle^{6\varepsilon}$ or $\langle L_2\rangle \geq \frac{1}{2}\langle L\rangle^{6\varepsilon}$.

In this case, we have the following
\[
    \frac{\la L \ra^{\veps}}{\la L \ra^{1/2} \la L_1 \ra^{1/2}\la L_2 \ra^{1/2}}  \ls \frac{1}{\la L \ra^{1/2} \la L_1 \ra^{1/3}\la L_2 \ra^{1/3}}.
\]
The purpose of the above inequality is to eliminate the term $\la L \ra^{\veps}$ in the numerator by paying the price of lowering the powers of $\la L_1\ra$ and $\la L_2 \ra$ from $1/2$ to $1/3$. We point out that the power $1/3$ suffices to obtain the desired result due to Lemma \ref{lemma4}. For example, the estimate \eqref{Region1} is still valid if the terms $\langle L_1\rangle^{\frac{1}{2}}$ and $\la L_2\ra^{\frac{1}{2}}$ in \eqref{Region1} and \eqref{inverse1} are replaced with $\la L_1\ra^{\frac{1}{3}}$ and $\la L_2\ra^{\frac{1}{3}}$ respectively.

\noindent\textbf{Region (3.2)}: $\langle L_1\rangle \le \frac{1}{2}\langle L\rangle^{6\varepsilon}$ and $\langle L_2\rangle \le \frac{1}{2}\langle L\rangle^{6\varepsilon}$.

Recalling that $H_\sigma^\xi(\xi_1) = L_1+L_2-L$, hence 
\[\tau - \xi^3 = L = -H_\sigma^\xi(\xi_1) + L_1 + L_2.\]
Due to the assumption that $\langle L_1\rangle \le \frac{1}{2}\langle L\rangle^{6\varepsilon}$ and $\langle L_2\rangle \le \frac{1}{2}\langle L\rangle^{6\varepsilon}$, we know 
\begin{equation}\label{Lattice_Form_3n}
\tau - \xi^3 = -H_\sigma^\xi(\xi_1) + o((\tau - \xi^3)^{10\varepsilon}) = -H_\sigma^\xi(\xi_1)+o(|H_\sigma^\xi(\xi_1)|^{10\varepsilon}).
\end{equation}
For fixed $\xi$, let 
\[
    \Omega(\xi) = \big\{\eta \in \mathbb{R} : \eta = -H_\sigma^\xi(\xi_1)+o(|H_\sigma^\xi(\xi_1)|^{10\varepsilon}) \text{ for some } \xi_1 \in \mathbb{Z}/\sigma\big\}.
\]
Next, if the inequality:
\begin{equation}\label{Count1}
|\Omega(\xi) \cap \{| \eta | \sim M\}| \lesssim  \sigma M^{\frac{2}{3}},
\end{equation} 
holds for any $|\xi| \ge 1$, it then follows from this inequality that
\begin{equation*}
\begin{aligned}
\text{LHS of } \eqref{Part21neq3n} &\le \|\langle \tau - \xi^3\rangle^{-\frac{1}{2}}\chi_{\Omega(\xi)}(\tau-\xi^3)\|_{L^{\infty,2}_{\xi,\tau}} \|\mathscr{B}_s(f_1,f_2)\|_{L^2_{\xi,\tau}}\\
&	 
= \sup_{\xi:|\xi|> 1} \left(\int_{|\mu| < \sigma^{100}} \langle\mu\rangle^{-1} \dd\mu + \sum_{M: M\geq\sigma^{100} } \int_{|\mu|\sim M} \langle\mu\rangle^{-1}\chi_{\Omega(\xi)}(\mu)\dd\mu\right)^{\frac12}
\|\mathscr{B}_s(f_1,f_2)\|_{L^2_{\xi,\tau}}\\
&\le
\left( 100 \ln{(\sigma+1)} + \sum_{M:M\geq\sigma^{100} (\text{dyadic})} M^{-1}M^{\frac{2}{3}+\frac{1}{100}}\right)^{\frac12}
\|\mathscr{B}_s(f_1,f_2)\|_{L^2_{\xi,\tau}}\\
&\lesssim \sigma\|\mathscr{B}_s(f_1,f_2)\|_{L^2_{\xi,\tau}} = \text{RHS of } \eqref{Part21neq3n}.
\end{aligned}
\end{equation*}

Finally, we prove \eqref{Count1}. Without loss of generality, we assume that $\xi$ is positive. Fix $\xi_1$, we have
\begin{equation*}
\left|\left\{\eta \in \mathbb{R} : |\eta| \sim M, \eta = -H_\sigma^\xi(\xi_1)+o(|H_\sigma^\xi(\xi_1)|^{10\varepsilon})\right\}\right| \sim M^{10\varepsilon}. 
\end{equation*}
Now, we estimate the number of possible values of $\xi_1 \in \mathbb{Z}/\sigma$ such that
\begin{equation*}
|H_\sigma^\xi(\xi_1)+o(|H_\sigma^\xi(\xi_1)|^{10\varepsilon})| \sim M.
\end{equation*}
Let $|\xi| \sim N\geq1$ be dyadic. Since 
\[H_\sigma^{\xi}(\xi_1) = -\xi \left[h^{\xi}(\xi_1)-\frac{\beta}{\sigma^2}\right],\] 
where $h^{\xi}(\xi_{1}) = 3(-2\xi_1+\xi)^2$ for $\alpha = 4$, it holds that
\begin{equation*}
\#\{\xi_{1}\in \mathbb{Z}/\sigma: |H_\sigma^\xi(\xi_{1})|\sim M\}\le \#\left\{\xi_{1}\in \mathbb{Z}/\sigma:  \Big|h^{\xi}(\xi_1)-\frac{\beta}{\sigma^2}\Big|\sim  M\right\} = \#\left\{\widetilde\xi_{1}\in \mathbb{Z}:  \Big| h^{\widetilde\xi}(\widetilde\xi_1)-\beta\Big|\sim \sigma^2 M\right\}
\end{equation*}
where $\widetilde{\xi} = \sigma\xi, \widetilde{\xi}_1 = \sigma \xi_{1} \in \mathbb{Z}$. Choosing $\sigma\geq1$ such that $\sigma^{100}>2\beta$, then $M\geq \sigma^{100}>2\beta$ and
\begin{equation*}
\#\left\{\widetilde\xi_{1}\in \mathbb{Z}:  \Big| h^{\widetilde\xi}(\widetilde\xi_1)-\beta\Big|\sim \sigma^2 M\right\}\le \#\left\{\widetilde\xi_{1}\in \mathbb{Z}:  0\le(-2\widetilde{\xi}_1+\widetilde\xi)^2\le\sigma^2 \frac{3M}{2}\right\},
\end{equation*}
so the number of $\xi_{1}$ is at most $\sim\sigma M^{\frac{1}{2}}$. Above all, the contribution
to \eqref{Count1} is at most $\sigma M^{\frac{1}{2}+10\varepsilon}\le\sigma M^{\frac{2}{3}}$. 

Hence, we finished the proof for the bilinear estimate (\ref{Bilinear1}) in Case 1.

\noindent\textbf{Proof of the bilinear estimate \eqref{Bilinear2} in Case 1}

The proof of the bilinear estimate \eqref{Bilinear2} follows exactly the same line as that of \eqref{Bilinear1}. However, it should be noted that we require the mean zero condition on $w_1$ in this case. Recalling the resonance function \eqref{H_tilde_root} for $\alpha = 4$:
\[
 \wt{H}^{\xi_1}(\xi) := 3\xi_1 \bigg[ (2\xi -  \xi_1)^2 - \frac{\b}{3} \bigg].
\]
When $\xi_1$ is fixed, we regard $\wt{H}^{\xi_1}(\xi)$ as a function of $\xi$, which is similar to the resonance function \eqref{H_res_fn_cpt} for \eqref{Bilinear1}. Hence the proof of the second bilinear estimate \eqref{Bilinear2} are
analogous to \eqref{Bilinear1}. The major difference here is the extra singularity induced by $\xi_1$ for \eqref{Bilinear2}. To ensure the estimate is valid when $\xi_1 = 0$, the assumption that $\mathscr{F}_{x}{w_1}(0, t) = 0$ for any $t$, i.e., the mean value of $w_1$ is zero, is necessary.

\subsubsection{Case 2: \texorpdfstring{$\beta= 3n^2, n\in \mathbb{N}$}{Case 2: beta equal to 3n squared, n in N}}
\label{Sec, pos_beta_r}

For the case $\beta =3n^2, n\in \m{N}^+$, there are infinitely many points $(\xi,\xi_1)\in \m{Z}/\sigma\times\m{Z}/\sigma$ such that the resonance function $H_\sigma^{\xi}(\xi_1)$ is zero, which leads to the conclusion that the critical index needs to be $1$ instead of $\frac12$. Recalling the proof for the case where $\beta\neq 3n^2$, we observed that the proof for the bilinear estimate \eqref{Bilinear2} is entirely analogous to that for \eqref{Bilinear1}. 
Thus, we will only carry out the proof for the bilinear estimate \eqref{Bilinear1} and omit that for the bilinear estimate (\ref{Bilinear2}). 

In addition, from Case 1 in Section \ref{Sec, pos_beta_nr}, we have seen that (\ref{Bilinear1}) is split to be (\ref{Bilinear1, p1}) and (\ref{Bilinear1, p2}), and once (\ref{Bilinear1, p1}) is established, the estimate (\ref{Bilinear1, p2}) can be justified by taking advantage of (\ref{Bilinear1, p1}). Therefore, in the current Case 2, we will only verify (\ref{Bilinear1, p1}) and leave (\ref{Bilinear1, p2}) to the readers.

For (\ref{Bilinear1, p1}), recalling the discussion at the beginning of Section \ref{Sec, pf_bilin_dres}, it suffices to justify (\ref{BilinearL2}), that is to prove 
\[
    \|\mathscr{B}_s(f_1,f_2)\|_{L^2_{\xi,\tau}}\lesssim \sigma\|f_1\|_{L_{\xi_{1},\tau_1}^2}\|f_2\|_{L_{\xi_{2},\tau_2}^2}.
\]
 
\noindent \textbf{Region (1).} $|\xi| \le  100n$.

For $s \ge 0$, we have
\begin{equation*}
	\frac{|\xi|\langle\xi\rangle^s}{\langle\xi_{1}\rangle^s\langle\xi_2\rangle^s} \lesssim \frac{|\xi|\langle\xi\rangle^s}{\langle \xi\rangle^s} \le 100n,\quad \text{and}\quad
\langle \tau-\xi^3\rangle^\frac{1}{2} \ge 1.
\end{equation*}
Then by H\"older's inequality and Lemma \ref{lemma4}, it follows that 
\begin{equation*}
\begin{aligned}
\|\mathscr{B}_s(f_1,f_2)\|_{L^2_{\xi,\tau}}&\lesssim
\Big\|\iint\limits_{\substack{\xi_1+\xi_{2} = \xi \\ \tau_1+\tau_2 = \tau}}\frac{f_1(\xi_{1},\tau_1)f_2(\xi_{2},\tau_2)}{\langle L_1\rangle^{1/2}\langle L_2\rangle^{1/2}}\dd\xi_1\dd\tau_1\Big\|_{L^2_{\xi,\tau}} = \|F_1F_2\|_{L^2_{x,t}}\\
&\le \|F_1\|_{L^4_{x,t}}\|F_2\|_{L^4_{x,t}}\lesssim\|F_1\|_{X^{0,\frac{1}{3}}_{4,\beta_{\sigma}}} \|F_2\|_{X^{0,\frac{1}{3}}_{4,\beta_{\sigma}}}\le \|f_1\|_{L^2_{x,t}} \|f_2\|_{L^2_{x,t}},
\end{aligned}
\end{equation*}
where 
\begin{equation*}
{F_1}(x,t) =\mathscr{F}^{-1}\{ \langle L_1\rangle^{-\frac{1}{2}} {f_1}(\xi_1,\tau_1)\},\quad{F_2}(x,t) = \mathscr{F}^{-1}\{\langle L_2\rangle^{-\frac{1}{2}} {f_2}(\xi_2,\tau_2)\}.
\end{equation*}

\noindent \textbf{Region (2).} $|\xi|>100n$ and $|2\xi_1 - \xi-\frac{n}{\sigma}| \ge \frac{1}{10}|\xi|$.

In this case, we have 
\begin{equation*}
\left|2\xi_1-\xi+\frac{n}{\sigma} \right|\ge \frac{1}{10}|\xi|-\frac{2n}{\sigma}\geq \frac{1}{20}|\xi|,
\end{equation*}
then 
\begin{equation*}
MAX_1\ge \frac{1}{3} \langle -L+L_1+L_2\rangle\ge |H_\sigma^\xi(\xi_1)| = |3\xi|\left|2\xi_1-\xi+\frac{n}{\sigma}\right|\left|2\xi_1 - \xi-\frac{n}{\sigma}\right|\gtrsim |\xi|^3.
\end{equation*}

\noindent \textbf{Region (2.1).} $\langle L\rangle = MAX_1$.

In this case, we have $\frac{|\xi|\langle\xi\rangle^s}{\langle\xi_1\rangle^s\langle \xi_2\rangle^s} \frac{1}{\langle\tau-\xi^3\rangle^\frac{1}{2}} \lesssim \frac{\langle\xi\rangle^s}{\langle\xi_1\rangle^s\langle \xi_2\rangle^s} \le 1$ for $s \ge 0$. The rest argument is similar to that for Region (1).

\noindent \textbf{Region (2.2).} $\langle L_1\rangle= MAX_1$ or $\langle L_2\rangle= MAX_1$.

Since \eqref{BilinearL2} is symmetric with respect to 
$L_1$ and $L_2$, we only consider the case $\langle L_1\rangle= MAX_1$. By duality, \eqref{BilinearFirstPart} is equivalent to
\begin{equation}\label{Duality1}
\Bigg|\int\limits_{\substack{\xi_1+\xi_{2} +\xi = 0 \\ \tau_1+\tau_2 + \tau = 0}} \frac{\xi\langle\xi\rangle^s g(\xi,\tau)f_1(\xi_{1},\tau_1)f_2(\xi_{2},\tau_2)}{\langle\xi_1\rangle^s\langle\xi_2\rangle^s\langle L\rangle^{1/2}\langle L_1\rangle^{1/2}\langle L_2\rangle^{1/2}}\dd\xi_1\dd\tau_1\dd\xi\dd\tau \Bigg|\lesssim \sigma \|f_1\|_{L^2_{\xi_1,\tau_1}}\|f_2\|_{L^2_{\xi_2,\tau_2}}\|g\|_{L^2_{\xi,\tau}}. 
\end{equation}
In this case, similarly, we have $\frac{|\xi|\langle\xi\rangle^s}{\langle\xi_1\rangle^s\langle \xi_2\rangle^s} \frac{1}{\langle L_1\rangle^\frac{1}{2}} \lesssim 1$. Then, using H\"older inequality and Lemma \ref{lemma4}, it holds that
\begin{equation*}
\begin{aligned}
LHS \text{ of } \eqref{Duality1}&\lesssim \left|\iint F_1(x,t)F_2(x,t)G(x,t)dx dt\right| \le \|F_1\|_{L^2_{x,t}} \|F_2\|_{L^4_{x,t}} \|G\|_{L^4_{x,t}}\\
&		
\le  \|F_1\|_{L^2_{x,t}} \|F_2\|_{X^{0,\frac{1}{3}}_{4,\beta_{\sigma}}} \|G\|_{X_\sigma^{0,\frac{1}{3}}} \le  \|f_1\|_{L^2_{\xi_1,\tau_1}}\|f_2\|_{L^2_{\xi_2,\tau_2}}\|g\|_{L^2_{\xi,\tau}}.
\end{aligned}
\end{equation*}
where 
\begin{equation*}
F_1(x,t) = \mathscr{F}^{-1}\{f_1({\xi_1,\tau_1})\},\quad{G}(x, t) =\mathscr{F}^{-1}\{\langle L_2\rangle^{-\frac{1}{2}}{f_2}(\xi_2, \tau_2)\},\quad{F_2}(x, t) = \mathscr{F}^{-1}\{\langle L\rangle^{-\frac{1}{2}}{g}(-\xi, -\tau)\}.
\end{equation*}

\noindent\textbf{Region (3).} $|\xi|>100n$ and $|2\xi_1 - \xi-\frac{n}{\sigma}| \le \frac{1}{10}|\xi|$.

In this case, we have
$\langle\xi_1\rangle \sim \langle\xi_2\rangle \sim \langle\xi\rangle$.
Since $s\geq 1$,
\begin{equation*}
\frac{|\xi|\langle\xi\rangle^s}{\langle\xi_1\rangle^s\langle\xi_2\rangle^s}\sim \frac{|\xi|}{\langle \xi\rangle}\le 1.
\end{equation*}
The rest of the calculation is exactly the same as that for Region (1).

\section{Proof of Theorem \ref{thm2}}
\label{Sec, pf_thm2}

\subsection{\texorpdfstring{Minimal $\gamma$-biased type index $\nu_{\gamma}$ and $\gamma$-biased irrational measure $\mu_{\gamma}$}{Minimal gamma-biased type index nu gamma and gamma-biased irrational measure mu gamma}}

\subsubsection{\texorpdfstring{Definitions of $\nu_{\gamma}$ and $\mu_{\gamma}$}{Definitions of nu\_gamma and mu\_gamma}}
Before proceeding to the proof of Theorem \ref{thm2} and its associated bilinear estimates, we first recall the concept of the minimal $\gamma$-biased type index, which was defined in Definition \ref{Def, mbti} in the introduction.
\begin{mydef}[i.e. Definition \ref{Def, mbti}]\label{Def, mbti2}
Let $\gamma \in \mathbb{R}$. A real number $\rho$ is said to be of $\gamma$-biased type $\nu$ if there exist positive constants $K = K(\rho, \nu, \gamma)$ and $N = N(\rho,\gamma)$ such that the inequality
\be\label{def_nu}
    \left|\rho - \frac{m}{n}+\frac{\gamma}{n^2}\right| \ge \frac{K}{|n|^{2+\nu}}
\ee
holds for all $(m,n) \in \mathbb{Z}^2$ with $|n|>N$. In addition, 
\begin{equation}
    \nu_{\gamma}(\rho) := \inf\{\nu \in \mathbb{R} : \rho \text{ is of $\gamma$-biased type } \nu\} 
\end{equation} 
is called the minimal $\gamma$-biased type index of $\rho$, where the infimum is understood as $\infty$ if $\{\nu \in \mathbb{R} : \rho \text{ is of $\gamma$-biased type } \nu\}$ is empty.
\end{mydef}

Compared with Definition 1 in \cite{Oh2009} where the bias $\gamma = 0$ and (\ref{def_nu}) holds for all $(m,n)\in \m{Z}^2$ with $n\neq 0$, we impose an extra lower bound requirement $|n|>N$ in the above Definition \ref{Def, mbti2} for the reasons below.
\begin{itemize}
	\item[(1)] Firstly, when $\gamma = 0$, then for any irrational number $\rho$, the left hand side of (\ref{def_nu}) never vanishes, which makes it possible to find positive $K$ for (\ref{def_nu}) to hold for all $(m,n)\in\m{Z}^2$ with $n \neq 0$.
	
	\item[(2)] Secondly, when $\gamma\neq 0$, then for any irrational number $\rho$, it is not clear whether the left hand side of (\ref{def_nu}): $\rho - \frac{m}{n} + \frac{\gamma}{n^2}$, never vanishes or not, so (\ref{def_nu}) may not hold for any $(m,n)\in\m{Z}^2$ with $n\neq 0$. Meanwhile, once the left hand side of (\ref{def_nu}) vanishes for some $(m_{*}, n_{*})$, then there are exactly two pairs of integers: $(m_{*}, n_{*})$ and $(-m_{*}, -n_{*})$ such that the left hand side of (\ref{def_nu}) vanishes. So the lower bound requirement $|n| > N$ in (\ref{def_nu}) is introduced to exclude these two pairs of integers. With this being said, the lower bound $N$ can be taken as $|n_{*}| + 1$ which only depends on $\rho$ and $\gamma$.
\end{itemize}

Since we study both well-posedness and ill-posedness in this paper, for convenience in proofs and notations, we introduce the following companion definition, namely, the $\gamma$-biased irrational measure.
\begin{mydef}\label{Def, bim}
	Let $\gamma \in \mathbb{R}$. The $\gamma$-biased irrational measure $\mu_{\gamma}$ of a real number $\rho$ is defined as
	\begin{equation}\label{Defofirrmea}
		\mu_{\gamma}(\rho) := \sup\left\{\mu \in \mathbb{R} : 0 < \left|\rho- 	\frac{m}{n}+\frac{\gamma}{n^2}\right| < \frac{1}{|n|^\mu}  \text{ holds for infinitely many } (m,n) \in \mathbb{Z} \times \mathbb{Z}^* \right\}.
	\end{equation}
\end{mydef} 
Now we demonstrate the connection between these two indices.

\subsubsection{\texorpdfstring{Relations between $\nu_{\gamma}$ and $\mu_{\gamma}$}{Relations between nu\_gamma and mu\_gamma}}
\begin{proposition}\label{muequnuplus2}
    Let $\gamma\neq0$. Then $\mu_{\gamma}(\rho) = 2+ \nu_{\gamma}(\rho)$ for all $\rho \in\m{R}$. In particular, $\nu_{\gamma}(\rho) = 0$ and $\mu_{\gamma}(\rho) = 2$ for all $\rho \in\m{Q}$.
\end{proposition}

\begin{proof}

We split the proof into two cases: (1) $\rho$ is rational, and (2) $\rho$ is irrational. 

{\bf Case (1): $\rho\in\m{Q}$}. 

Since $\rho\in \m{Q}$, we can write
\[
\rho = \frac{q}{p}, \quad \text{where\ }p,q\in\m{Z}\setminus\{0\}\ \text{and}\  \gcd(p,q) = 1,
\]
which implies that 
\[
     \left|\rho-\frac{m}{n}+\frac{\gamma}{n^2}\right| = \left|\frac{qn-mp}{np}+\frac{\gamma}{n^2}\right|.
\]
Hence, for any $|n| > |p| + |p\gamma|$,
\[\begin{split}
\Big|\frac{qn-mp}{np}+\frac{\gamma}{n^2}\Big|
&\geq \left\{\begin{array}{lll}
    |\frac{\gamma}{n^2}|, & \text{if} & qn-mp = 0,\\
    |\frac{1}{np}|-|\frac{\gamma}{n^2}|\geq \frac{1}{n^2}, & \text{if} & qn-mp \neq 0,
\end{array}\right. \\
&\geq \frac{1}{n^2}\min\{1, |\gamma|\},
\end{split}\]
 which implies that $\nu_{\gamma}(\rho)\le 0$. 

 If $\nu_{\gamma}(\rho) < 0$, then there exists some \(\varepsilon > 0\) such that \(\nu_{\gamma}(\rho) < -\varepsilon\), which implies there exist positive \(K(\rho, \nu, \gamma) \) and \(N(\rho, \gamma)\) such that
 \[
  \left|\rho-\frac{m}{n}+\frac{\gamma}{n^2}\right|\geq \frac{K}{|n|^{2-\varepsilon}} , \quad \text{for any } n>N.
 \]
 Meanwhile, by choosing $(m,n) = k(q,p), k\in\m{Z}^*$, we find for any $n>N$,
    \[
    0<\frac{K}{|n|^{2-\varepsilon}}\le\left|\rho-\frac{m}{n}+\frac{\gamma}{n^2}\right| = \frac{|\gamma|}{n^2}, 
    \]
which is a contradiction when $n\to \infty$. Therefore, $\nu_{\gamma}(\rho)\geq 0$. Combining with the fact that $\nu_{\gamma}(\rho)\leq 0$ we verified earlier, we conclude that $\nu(\rho) = 0$ for $\rho\in \mathbb{Q}$.

Similar to the above proof, we can also show that $\mu_\gamma(\rho) = 2$ for $\rho\in \mathbb{Q}$. Therefore, $\mu_{\gamma}(\rho) =\nu_{\gamma}(\rho)+2$ when $\rho\in\m{Q}$.

{\bf Case (2):} $\rho\in \m{R}\setminus\m{Q}$. 

Firstly, we prove $\nu_{\gamma}(\rho)+2\geq \mu_{\gamma}(\rho)$ for $\rho \in \mathbb{R}\setminus\mathbb{Q}$. According to the definition of $\mu_\gamma$, for any $\varepsilon_1 > 0$, there exist infinitely many $(m,n) \in \m{Z}\times \m{Z}^*$ such that 
\begin{equation*}
0<\left|\rho -\frac{m}{n}+\frac{\gamma}{n^2}\right|< \frac{1}{n^{\mu_{\gamma}(\rho)-\varepsilon_1}}.
\end{equation*}
So for any $\varepsilon_2 > 0$, there does not exist positive numbers $K$ and $N$ such that 
\begin{equation*}
\left|\rho-\frac{m}{n}+\frac{\gamma}{n^2}\right|\geq\frac{K}{n^{\mu_{\gamma}(\rho)-\varepsilon_1-\varepsilon_2}},
\end{equation*} 
for all $(m,n)\in\m{Z}^2$ with $n>N$. Hence, we have $\nu_{\gamma}(\rho)+2\geq \mu_{\gamma}(\rho)-\varepsilon_1-\varepsilon_2$, which implies $\nu_{\gamma}(\rho)+2\geq \mu_{\gamma}(\rho)$. 

Conversely, given $\mu_{\gamma}(\rho)$, then for any $\varepsilon > 0$, there exist at most finitely many $(m,n)\in\m{Z}\times\m{Z}^*$ such that 
\begin{equation*}
0<\left|\rho-\frac{m}{n}+\frac{\gamma}{n^2}\right|< \frac{1}{|n|^{\mu_{\gamma}(\rho)+\varepsilon}}.
\end{equation*}
Meanwhile, the equation 
\be\label{exp_eq}
    \rho-\frac{m}{n}+\frac{\gamma}{n^2} = 0,
\ee
can have at most two solutions $(m,n)\in\m{Z}\times\m{Z}^*$. 
Noting $(m,n)$ satisfies (\ref{exp_eq}) if and only if $(-m,-n)$ satisfies (\ref{exp_eq}), so it is equivalent to prove (\ref{exp_eq}) has at most one solution $(m,n)\in\m{Z}\times\m{Z}^*$ with $n>0$. In fact, if 
there exist $(m_1,n_1)$ and $(m_2, n_2)$, with positive $n_1$ and $n_2$, such that (\ref{exp_eq}) holds and $(m_1, n_1) \neq (m_2, n_2)$, then $n_1$ has to be different from $n_2$ and $n_1^2 - n_2^2 \neq 0$. Moreover, 
\begin{equation*}
\rho-\frac{m_1}{n_1}+\frac{\gamma}{n_1^2}=\rho-\frac{m_2}{n_2}+\frac{\gamma}{n_2^2}=0,
\end{equation*}
so
\begin{equation*}
\gamma = \frac{n_1n_2(m_2 n_1 - m_1 n_2)}{n_1^2-n_2^2}\in\mathbb{Q}.
\end{equation*}
As a result,
\begin{equation*}
\rho = \frac{m_1}{n_1}-\frac{\gamma}{n_1^2}\in \mathbb{Q},
\end{equation*} 
which leads to a contradiction with $\rho \in \mathbb{R}\setminus\mathbb{Q}$. Hence, we proved that given any $\rho\in\m{R}\setminus\{\m{Q}\}$ and $\gamma\in\m{R}$, the equation (\ref{exp_eq}) has at most two solutions $(m,n)$ which only depend on $\rho$ and $\gamma$. So there exists a positive integer $N = N(\rho,\gamma)$ such that (\ref{exp_eq}) does not admit a solution $(m,n)$ with $|n|>N$. 

So far, we demonstrated that for $(m,n)\in\m{Z}\times \m{Z}^*$ with $|n|>N$, except for at most finitely many $(m,n)$ such that 
\be\label{pos_small_diff}
0< \left|\rho-\frac{m}{n}+\frac{\gamma}{n^2}\right|< \frac{1}{|n|^{\mu_{\gamma}(\rho)+\varepsilon}},
\ee
all the other $(m,n)$ satisfy 
\[\left|\rho-\frac{m}{n}+\frac{\gamma}{n^2}\right| \geq \frac{1}{|n|^{\mu_{\gamma}(\rho)+\varepsilon}}.\]
So there exists a positive number $K_1 = K_1(\rho,\gamma,\mu_{\gamma}(\rho))$ such that for those finitely many $(m,n)$ that satisfy (\ref{pos_small_diff}), it holds that 
\[
\left|\rho-\frac{m}{n}+\frac{\gamma}{n^2}\right| \geq \frac{K_1}{|n|^{\mu_{\gamma}(\rho)}}.
\]
Consequently, for any $(m,n)\in\m{Z}\times \m{Z}^*$ with $|n|>N$, 
\[
\left|\rho-\frac{m}{n}+\frac{\gamma}{n^2}\right| \geq \frac{K}{|n|^{\mu_{\gamma}(\rho)+\eps}},
\]
where $K:=\min\{K_1, 1\}$.
According to Definition \ref{Def, mbti2}, this means $ \mu_{\gamma}(\rho)+\varepsilon \geq \nu_{\gamma}(\rho) +2$, which further implies that $ \mu_{\gamma}(\rho) \geq \nu_{\gamma}(\rho) +2$. 

Therefore, $ \mu_{\gamma}(\rho) = \nu_{\gamma}(\rho) +2$ for any $\rho\in\m{R}\setminus\m{Q}$ as well.
\end{proof}

\begin{remark}
We point out that Proposition \ref{muequnuplus2} is not valid when $\gamma = 0$ and $\rho\in\m{Q}$. In fact, when $\gamma = 0$, we have $\mu_0(\rho) = 1$ and $\nu_0(\rho) = \infty$ for $\rho \in \mathbb{Q}$. 
\end{remark}

\subsubsection{\texorpdfstring{Properties of $\nu_{\gamma}$ and $\mu_{\gamma}$}{Properties of nu\_gamma and mu\_gamma}}
Next, we derive some properties of $\nu_{\gamma}$ and $\mu_{\gamma}$.
\begin{proposition}\label{munu}
Let $\gamma\in\m{R}\setminus\{0\}$. Then the following properties hold.
\begin{itemize}
\item[(1)] For any $\rho\in\m{R}$ and $k\in\m{Z}$, 
$\nu_{\gamma}(\rho + k) = \nu_{\gamma}(\rho)$ and $\mu_{\gamma}(\rho + k) = \mu_{\gamma}(\rho)$.

\item[(2)] For any $\rho\in\m{R}$, 
$\nu_{\gamma}(\rho)\geq0$ and $\mu_{\gamma}(\rho)\geq 2 $.

\item[(3)] For a.e. $\rho\in\m{R}$, 
$\nu_{\gamma}(\rho)=0$ and $\mu_{\gamma}(\rho)=2$.
\end{itemize}
\end{proposition}
\begin{proof}

Thanks to Proposition \ref{muequnuplus2}, we only need to prove the above properties for \(\mu_\gamma\).\

\begin{itemize}
\item[(1)] This part is obvious.

\item[(2)] If $\rho\in\m{Q}$, then it has already been shown that $\nu_{\gamma}(\rho) = 0$ and $\mu_{\gamma}(\rho) = 2$. So it remains to study the case when $\rho \in \mathbb{R}\setminus\mathbb{Q}$. In this case, it is well known that $\mu_{0}(\rho) \geq 2$, see e.g. Proposition 1 on page 13 in \cite{YZ22DCDS}. So for any $\eps>0$, there exists infinitely many $(m,n)$ such that 
\be\label{ie_0}
    0 < \Big|\rho - \frac{m}{n} \Big| < \frac{1}{|n|^{2-\varepsilon}}.
\ee
For any fixed $n$, the above inequality can hold only for finitely many $m$. In addition, (\ref{ie_0}) holds for a pair $(m,n)$ if and only if it holds for the pair $(-m, -n)$. So we can find infinitely many pairs of $(m,n)$ such that (\ref{ie_0}) holds, where $n>0$ is so large that
$n^{\varepsilon}\geq 2 + |\gamma|$. As a result, 
\begin{equation*}
\Big|\rho -\frac{m}{n}+\frac{\gamma}{n^2}\Big|\le \Big|\rho -\frac{m}{n}\Big|+\Big|\frac{\gamma}{n^2}\Big| < \frac{1}{n^{2-2\varepsilon}}.
\end{equation*}  
Meanwhile, since $\rho\in\m{R}\setminus\m{Q}$, there exists at most one pair of $(m,n)$ such that $n>0$ and $\rho -\frac{m}{n}+\frac{\gamma}{n^2} = 0$. Hence, there exist infinitely many $(m,n)$ such that $n>0$ and
\[ 0 < \Big|\rho -\frac{m}{n}+\frac{\gamma}{n^2}\Big| < \frac{1}{n^{2-2\varepsilon}},\]
which implies that $\mu_{\gamma}(\rho)\geq 2-2\varepsilon$. Sending $\varepsilon \to 0^{+}$, we have $\mu_{\gamma}(\rho)\geq 2$. 

\item[(3)] Since $ \mu_{\gamma}(\rho)\geq 2$ and $\mu_{\gamma}$ is invariant under integer translation, then we only need to prove that $m(A_{\varepsilon})= 0$ for any $\eps>0$, where 
\[A_{\varepsilon} =\{\rho\in[0,1):\mu(\rho)>2+\varepsilon\}.\] 

According to the definition and the discussion in part (2), for any $\rho \in A_\varepsilon$, there exist infinitely many pairs of integers $(m,n)$ such that $n>0$ and
\begin{equation*}
\left|\rho-\frac{m}{n}+\frac{\gamma}{n^2}\right|<\frac{1}{n^{2+\varepsilon}}.
\end{equation*}
Note that the set of $n$ satisfying the above inequality has an infinite supremum, otherwise only a finite number of pairs $(m,n)$ could satisfy the above condition. So given any integer $K>|\gamma|>0$, there must exist an $n>K$ and an $m$ such that
\begin{equation*}
\rho \in I_{\gamma}(m,n)=\left(\frac{m}{n}-\frac{\gamma}{n^2}-\frac{1}{n^{2+\varepsilon}}, \, \frac{m}{n}-\frac{\gamma}{n^2}+\frac{1}{n^{2+\varepsilon}}\right).
\end{equation*}
Since $0\leq \rho < 1$, we have 
\[ \frac{m}{n}-\frac{\gamma}{n^2}-\frac{1}{n^{2+\varepsilon}} < 1, \quad \text{ and }\quad \frac{m}{n}-\frac{\gamma}{n^2}+\frac{1}{n^{2+\varepsilon}} > 0,\]
which implies that $0\leq m\leq n+1$. Hence, 
\[
\rho \in \bigcup_{n = K+1}^{\infty} \bigcup_{m=0}^{n+1} I_{\gamma}(m,n).
\]
Consequently,
\begin{equation*}
\begin{aligned}m(A_{\varepsilon})&\leq
\sum_{n = K+1}^{\infty} \sum_{m=0}^{n+1} |I_{\gamma}(m,n)| =
\sum_{n = K+1}^{\infty}\sum_{m=0}^{n+1}
\frac{2}{n^{2+\varepsilon}} \\
&\leq 4\sum_{n=K+1}^{\infty}\frac{1}{n^{1+\varepsilon}}
\le 4\int_{K}^{\infty}\frac{1}{x^{1+\varepsilon}}\dd x
=
\frac{4}{\varepsilon K^{\varepsilon}}.
\end{aligned}
\end{equation*}
Let $K \to \infty$, then $m(A_{\varepsilon}) = 0$, which completes the proof.
\end{itemize}
\end{proof}

\subsection{Key bilinear estimates}
Next, we turn to prove Theorem \ref{thm2}. Based on the contraction mapping argument in the proof of Theorem \ref{thm1} in Section \ref{Sec, CM}, it 
reduces to verifying the following two bilinear estimates.

\begin{proposition}\label{Prop2} 
Let $\a \in (0,4)\setminus\{1\}$, $\b\neq 0$ and $\sigma\geq 1$. Let $s^*(\alpha,\beta)$ be defined as in (\ref{ci_res}). Then for any $s > s^*(\alpha,\beta)$, the bilinear estimate (\ref{Bilinear3}) holds for any $w_1$ and $w_2$ in the Schwartz space $\mathscr{S}(\mathbb{T}_\sigma \times \mathbb{R})$.
\begin{equation}\label{Bilinear3}
\|\partial_x(w_1 w_2)\|_{Z_\sigma^s(\mathbb{T}_\sigma \times \mathbb{R})} \le C \sigma\|w_1\|_{Y^s_{\alpha,\beta_{\sigma}}(\mathbb{T}_\sigma \times \mathbb{R})}\|w_2\|_{Y^s_{\alpha,\beta _{\sigma}}(\mathbb{T}_\sigma \times \mathbb{R})},
\end{equation}
where $C = C(s,\a,\b)$. 
Moreover, under the additional mean-zero condition on $w_1$, i.e., $\mathscr{F}_{x}{w_1}(0, t) = 0$ for any $t$, the following bilinear estimate (\ref{Bilinear4}) also holds for any $s > s^*(\alpha,\beta)$.
\begin{equation}\label{Bilinear4}
\|\partial_x(w_1w_2)\|_{Z^s_{\alpha,\beta_{\sigma}}(\mathbb{T}_\sigma \times \mathbb{R})} \le C \sigma\|w_1\|_{Y_\sigma^s(\mathbb{T}_\sigma \times \mathbb{R})}\|w_2\|_{Y^s_{\alpha,\beta_{\sigma}}(\mathbb{T}_\sigma \times \mathbb{R})},
\end{equation}
where $C = C(s,\a,\b)$. 
In addition, when $R_{\alpha}\in\mathbb{Q}$, the ranges for $s$ in both (\ref{Bilinear3}) and (\ref{Bilinear4}) can be extended to include the endpoint, i.e. $s \geq \frac12$.
\end{proposition}
The proofs of (\ref{Bilinear3}) and (\ref{Bilinear4}) are very similar, so we will only provide details for (\ref{Bilinear3}) and then illustrate why (\ref{Bilinear4}) needs the extra assumption that $w_1$ has zero mean. 

For (\ref{Bilinear3}), according to the definition of the space $Z_\sigma^s(\mathbb{T}_\sigma \times \mathbb{R})$, it is equivalent to justify the following two estimates:
\begin{eqnarray}
\|\partial_x(w_1w_2)\|_{X_\sigma^{s, -\frac{1}{2}}} \le C_1\sigma\|w_1\|_{Y^s_{\a,\beta_{\sigma}}}\|w_2\|_{Y^s_{\a,\beta _{\sigma}}}, \label{Bilinear3, p1} \\
\|\langle\xi\rangle^s\langle\tau-\alpha\xi^3+\beta_\sigma\xi\rangle^{-1}\widehat{\p_x(w_1 w_2)}(\xi,\tau)\|_{L^2_\xi(\mathbb{Z_\sigma})L^1_\tau(\mathbb{R})}
\le C_1\sigma\|w_1\|_{Y^s_{\a,\beta_{\sigma}}}\|w_2\|_{Y^s_{\a,\beta _{\sigma}}}. \label{Bilinear3, p2}
\end{eqnarray}

\subsubsection{Proof of the main part of the bilinear estimate}

\begin{proof}[Proof of \eqref{Bilinear3, p1}]
Firstly, based on the definition of the space $Y^s_{\a,\beta _{\sigma}}$ in Definition \ref{Def, FR space}, it suffices to show 
\begin{equation}\label{BilinearFirstPart2}
\|\partial_x(w_1w_2)\|_{X_\sigma^{s,-\frac{1}{2}}} \lesssim \sigma\|w_1\|_{X_{\a,\beta_{\sigma}}^{s.\frac{1}{2}}}\|w_2\|_{X_{\a,\beta_{\sigma}}^{s,\frac{1}{2}}}.
\end{equation}
Inspired by the proof for (\ref{BilinearFirstPart}), we define the following bilinear operator:
\begin{equation*}
\mathscr{B}_{s}(f_1,f_2)(\xi,\tau) = \frac{\xi \langle\xi\rangle^s}{\langle L\rangle^{1/2}}\iint\limits_{\substack{\xi_1+\xi_{2} = \xi \\ \tau_1+\tau_2 = \tau}}\frac{f_1(\xi_{1},\tau_1)f_2(\xi_{2},\tau_2)}{\langle\xi_1\rangle^s\langle\xi_2\rangle^s\langle L_1\rangle^{1/2}\langle L_2\rangle^{1/2}}\dd \xi_1 \dd\tau_1, \quad \forall\, \xi\in \m{Z}/\sigma, \, \tau\in\m{R},
\end{equation*}
where
\be\label{L-res}
L = \tau - \xi^3,\quad L_1 = \tau_1 - \alpha \xi_1^3+\frac{\beta}{\sigma^2} \xi_1,\quad L_2 = \tau_2 - \alpha\xi_2^3+\frac{\beta}{\sigma^2} \xi_2.
\ee
Thus, (\ref{BilinearFirstPart2}) is converted to be
\begin{equation}
\|\mathscr{B}_s(f_1,f_2)\|_{L^2_{\xi,\tau}}\lesssim \sigma\|f_1\|_{L_{\xi_{1},\tau_1}^2}\|f_2\|_{L_{\xi_{2},\tau_2}^2}.
\end{equation}

According to (\ref{L-res}), all temporal frequency variables will disappear if we
subtract $L$ from $L_1 + L_2$, that is:
\be\label{res_fn_general}
    L_1 + L_2 - L = \Big( -\alpha \xi_1^3 + \frac{\beta}{\sigma^2} \xi_1 \Big) + \Big( -\alpha \xi_2^3 + \frac{\beta}{\sigma^2} \xi_2\Big) + \xi^3 := H_{\sigma}(\xi, \xi_1, \xi_2),
\ee
where the function $H_{\sigma}$ is called the resonance function. By substituting $\xi_2 = \xi-\xi_1$, we obtain 
\[
    H_{\sigma}(\xi, \xi_1, \xi_2) = -3\a\xi\xi_1^2 + 3\a\xi^2\xi_1 + (1-\a)\xi^3 + \frac{\b}{\sigma^2}\xi.
\]
When $\xi$ is fixed, the above expression can be viewed as a function in $\xi_1$. We denote this function as $H_{\sigma}^{\xi}$ which can be rewritten below: 
\be\label{res_fn_sigma}
    H_{\sigma}^{\xi}(\xi_1) = -3\a\xi \Big(\xi_1^2 - \xi\xi_1 + \frac{\a-1}{3\a}\xi^2 - \frac{\b}{3\a\sigma^2} \Big).
\ee

When $\alpha\in (0,4)\setminus \{1\}$, on the one hand, 
\[\xi_1^2 - \xi\xi_1 + \frac{\a-1}{3\a}\xi^2 = 0\]
has two roots $\xi_1 = c_1\xi$ and $c_2\xi$, where   
\be\label{c1c2}
    c_1 = \frac12 + \frac{R_\a}{6}, \qquad
    c_2 = \frac12 - \frac{R_\a}{6}, \qquad 
    R_{\a} = \sqrt{12/\a - 3},
\ee
so $H_{\sigma}^{\xi}(\xi_1)$ can be expressed as 
\be\label{H_factor}
    H_{\sigma}^{\xi}(\xi_1) = -3\a\xi (\xi_1 - c_1\xi)(\xi_1-c_2\xi) + \frac{\b\xi}{\sigma^2}.
\ee
On the other hand, if $|\xi|$ is so large that 
\be\label{large_xi2}
    |\xi|^2 \geq \frac{24 |\b|}{(12-3\a) \sigma^2} = \frac{24|\b|}{\a R_\a^2 \sigma^2},
\ee
then the resonance function $H_{\sigma}^{\xi}$ can be factored below:
\begin{equation}\label{Resonance04}
    H_\sigma^{\xi}(\xi_{1}) = -3\alpha\xi\left(\xi_{1}-x_1\right)\left(\xi_1-x_2\right),
\end{equation}
where 
\begin{equation*}
x_1 = \frac{1}{2}\xi+\frac{1}{6}\sqrt{(R_{\alpha}\xi)^2+\frac{12}{\alpha}\frac{\beta}{\sigma^2}},\quad 
x_2 = \frac{1}{2}\xi-\frac{1}{6}\sqrt{(R_{\alpha}\xi)^2+\frac{12}{\alpha}\frac{\beta}{\sigma^2}}. 
\end{equation*}
We can further expand $x_1$ and $x_2$ in terms of the order of $|\xi|$ to obtain
\[\left\{\begin{array}{l}
    x_1 = \frac12 \xi + \frac16 R_\a |\xi| + \frac{\lam}{\sigma^2 |\xi|} + Q_1(\xi), \vspace{0.1in}\\
    x_2 = \frac12 \xi - \frac16 R_\a |\xi| - \frac{\lam}{\sigma^2 |\xi|} + Q_2(\xi),
\end{array}\right.
\qquad \lam:= \frac{\b}{\a R_\a},\]
where 
\be\label{decom_xi2}
    |Q_{j}(\xi)| \leq \frac{12\lam^2}{R_\a \sigma^4} |\xi|^{-3}, \qquad j=1,2, \qquad \text{$\forall\, \xi$ satisfies (\ref{large_xi2})}.
\ee
The case when $\xi>0$ and the case when $\xi<0$ are similar, so let us focus on the former case. 

In the remaining proof, we assume $\xi>0$. Then for positive $\xi$ which satisfies (\ref{large_xi2}), it follows from (\ref{decom_xi2}) that $x_1$ and $x_2$ can be rewritten as 
\be\label{xi_decom2}
    x_1 = c_1\xi + \frac{\lambda}{\sigma^2 \xi} + Q_1(\xi), \qquad
    x_2 = c_2\xi - \frac{\lambda}{\sigma^2\xi} + Q_2(\xi).
\ee
Since $\a\neq 1$, neither $c_1$ nor $c_2$ is zero. 

\noindent \textbf{Region (1).} $0 < \xi \le E_{\alpha,\beta,s}$, where $E_{\a,\b,s}$ is a constant which only depends on $\a$, $\b$ and $s$. The specific choice of $E_{\a,\b,s}$ will be determined later.

For $s \ge 0$, we have
\begin{equation*}
\frac{\xi\langle\xi\rangle^s}{\langle\xi_{1}\rangle^s\langle\xi_2\rangle^s\langle \tau-\xi^3\rangle^\frac{1}{2}} \le \frac{\xi\langle\xi\rangle^s}{\langle \xi\rangle^s\langle \tau-\xi^3\rangle^\frac{1}{2}} \le  E_{\alpha,\beta,s}.
\end{equation*}
Then by H\"older's inequality and Lemma \ref{lemma4}, we have
\begin{equation}\label{TEMPF1}
\begin{aligned}
\|\mathscr{B}_s(f_1,f_2)\|_{L^2_{\xi,\tau}}&\lesssim 
\Big\|\iint\limits_{\substack{\xi_1+\xi_{2} = \xi \\ \tau_1+\tau_2 = \tau}}\frac{f_1(\xi_{1},\tau_1)f_2(\xi_{2},\tau_2)}{\langle L_1\rangle^{1/2}\langle L_2\rangle^{1/2}}\dd\xi_1\dd\tau_1\Big\|_{L^2_{\xi,\tau}} = \|F_1F_2\|_{L^2_{x,t}}\\
&\le  \|F_1\|_{L^4_{x,t}}\|F_2\|_{L^4_{x,t}}\lesssim\|F_1\|_{X^{0,\frac{1}{3}}_{4,\beta_{\sigma}}} \|F_2\|_{X^{0,\frac{1}{3}}_{4,\beta_{\sigma}}}\le \|f_1\|_{L^2_{x,t}} \|f_2\|_{L^2_{x,t}}
\end{aligned}
\end{equation}
where 
\begin{equation}\label{TEMPD1}
{F_1}(x,t) =\mathscr{F}^{-1}\{ \langle L_1\rangle^{-\frac{1}{2}} {f_1}(\xi_1,\tau_1)\},\quad{F_2}(x,t) = \mathscr{F}^{-1}\{\langle L_2\rangle^{-\frac{1}{2}} {f_2}(\xi_2,\tau_2)\}.
\end{equation}

\noindent \textbf{Region (2).} $\xi > E_{\alpha,\beta,s}$, $ |\xi_{1}-c_1\xi|\geq \frac{1}{2\sigma}$ and $|\xi_{1}-c_2\xi|\geq \frac{1}{2\sigma}$, where $c_1$ and $c_2$ are as defined in \eqref{c1c2}.

By choosing $E_{\a,\b,s}$ such that 
\be\label{const_M1}
    E_{\a,\b,s} \geq \frac{4}{R_\a}\bigg(1+\frac{|\b|}{\a}\bigg),
\ee
then for any $\xi > E_{\a,\b,s}$, (\ref{large_xi2}) is satisfied and moreover,  
\[
    \frac{\lam}{\sigma^2\xi} + \frac{12\lam^2}{R_a\sigma^4}\xi^{-3} \leq \frac{1}{2\sigma}.
\]
So it follows from (\ref{xi_decom2}) and (\ref{decom_xi2}) that
\be\label{loc_x12}
    x_1 \in \Big[c_1\xi - \frac{1}{2\sigma}, c_1\xi + \frac{1}{2\sigma}\Big], \qquad x_2 \in \Big[c_2\xi - \frac{1}{2\sigma}, c_2\xi + \frac{1}{2\sigma}\Big],
\ee
where $x_1$ and $x_2$ are the two roots of the quadratic function $H_{\sigma}^{\xi}(\xi_1)$ in (\ref{Resonance04}). Then due to the constraint $|\xi_1 - c_k \xi|\geq \frac{1}{2\sigma}$ for $k=1,2$, we know
\[
    \big| H_{\sigma}^{\xi}(\xi_1)\big| \geq \min\bigg\{
    \Big| H_{\sigma}^{\xi}\Big( c_1\xi - \frac{1}{2\sigma}\Big)\Big|, \, \Big|H_{\sigma}^{\xi}\Big( c_1\xi + \frac{1}{2\sigma}\Big)\Big|, \, \Big|H_{\sigma}^{\xi}\Big( c_2\xi - \frac{1}{2\sigma}\Big)\Big|, \, \Big|H_{\sigma}^{\xi}\Big( c_2\xi + \frac{1}{2\sigma}\Big)\Big|\bigg\}.
\]
By direct computation, it follows from (\ref{H_factor}) that 
\[
    \begin{aligned}
        \Big|H_\sigma^{\xi}\Big(c_1\xi-\frac{1}{2\sigma}\Big)\Big| &= \Big| \frac{3\a\xi}{2\sigma}\Big[ (c_1 - c_2)\xi - \frac{1}{2\sigma} \Big] + \frac{\b\xi}{\sigma^2} \Big| = \Big|\frac{\a R_\a}{2\sigma}\xi^2 - \frac{3\a}{4\sigma^2}\xi + \frac{\b}{\sigma^2}\xi\Big|.
    \end{aligned}
\]
Since $\xi > E_{\a,\b,s}$, where $E_{\a,\b,s}$ has a lower bound as that in (\ref{const_M1}), then it follows from the above equality that 
\[
    \Big|H_\sigma^{\xi}\Big(c_1\xi-\frac{1}{2\sigma}\Big)\Big| \geq \frac{\a R_\a}{4\sigma}\xi^2.
\]
Similarly, we can justify that $\frac{\a R_\a}{4\sigma}\xi^2$ is the common lower bound of $|H_{\sigma}^{\xi}(c_k\xi \pm \frac{1}{2\sigma})|$ for $k=1,2$, which implies that 
\[
    \big| H_{\sigma}^{\xi}(\xi_1)\big| \geq \frac{\a R_\a}{4\sigma}\xi^2.
\]

Let $MAX_1 := \max\{\langle L\rangle,\langle L_1\rangle,\langle L_2\rangle\} $, then it follows from the relation $(\ref{res_fn_general})$ that 
\begin{equation*}
MAX_1 \geq \frac13 |H_\sigma^{\xi}(\xi_{1})| \gtrsim \frac{\xi^2}{\sigma}.
\end{equation*}

\noindent \textbf{Region (2.1).} $\langle L\rangle = MAX_1$.

In this case, we have $\la \tau-\xi^3 \ra = \la L \ra \gs \xi^2/\sigma$, so
\begin{equation*}
\frac{|\xi|\langle\xi\rangle^s}{\langle\xi_1\rangle^s\langle \xi_2\rangle^s} \frac{1}{\langle\tau-\xi^3\rangle^\frac{1}{2}} \lesssim \sigma^{\frac{1}{2}} \frac{\langle\xi\rangle^s}{\langle\xi_1\rangle^s\langle \xi_2\rangle^s} \lesssim \sigma^{\frac{1}{2}},
\end{equation*} 
where the last inequality is due to $\xi = \xi_1 + \xi_2$ and $s\geq 0$. Then the rest argument is similar to that in (\ref{TEMPF1}) with an extra coefficient $\sigma^{\frac12}$.

\noindent \textbf{Region (2.2).} $\langle L_1\rangle= MAX_1$ or $\langle L_2\rangle= MAX_1$.

We only consider the case $\langle L_1\rangle= MAX_1$. By duality, \eqref{BilinearFirstPart2} is equivalent to
\begin{equation}\label{Duality4}
\left|\int \frac{\xi\langle\xi\rangle^s g(\xi,\tau)f_1(\xi_{1},\tau_1)f_2(\xi_{2},\tau_2)}{\langle\xi_1\rangle^s\langle\xi_2\rangle^s\langle L\rangle^{1/2}\langle L_1\rangle^{1/2}\langle L_2\rangle^{1/2}}\dd\xi_1\dd\tau_1\dd\xi\dd\tau \right|\lesssim \sigma \|f_1\|_{L^2_{\xi_1,\tau_1}}\|f_2\|_{L^2_{\xi_2,\tau_2}}\|g\|_{L^2_{\xi,\tau}}. 
\end{equation}
Similar to the discussion in Region (2.1), we have 
\[\frac{|\xi|\langle\xi\rangle^s}{\langle\xi_1\rangle^s\langle \xi_2\rangle^s} \frac{1}{\langle L_1\rangle^\frac{1}{2}} \lesssim  \sigma^{\frac{1}{2}}.\] 
Then, by H\"older's inequality, we find
\be\label{PlanHold1}
LHS \text{ of } \eqref{Duality4} \lesssim  \sigma^{\frac{1}{2}}\left|\iint F_1(x,t)F_2(x,t)G(x,t) \dd x \dd t\right| \le  \sigma^{\frac{1}{2}}\|F_1\|_{L^2_{x,t}} \|F_2\|_{L^4_{x,t}} \|G\|_{L^4_{x,t}},
\ee
where 
\be\label{TEMPD2}
F_1(x,t) = \mathscr{F}^{-1}\{f({\xi_1,\tau_1})\},\quad{G}(x, t) =\mathscr{F}^{-1}\{\langle L_2\rangle^{-\frac{1}{2}}{f_2}(\xi_2, \tau_2)\},\quad{F_2}(x, t) = \mathscr{F}^{-1}\{\langle L\rangle^{-\frac{1}{2}}{g}(-\xi, -\tau)\}.
\ee
Now we apply Lemma\eqref{lemma4} to (\ref{PlanHold1}) to
obtain
\be\label{TEMPF2}
LHS \text{ of } \eqref{Duality4} \ls \sigma^{\frac{1}{2}} \|F_1\|_{L^2_{x,t}} \|F_2\|_{X^{0,\frac{1}{3}}_{4,\beta_{\sigma}}} \|G\|_{X^{0,\frac{1}{3}}} \le   \sigma^{\frac{1}{2}}\|f_1\|_{L^2_{\xi_1,\tau_1}}\|f_2\|_{L^2_{\xi_2,\tau_2}}\|g\|_{L^2_{\xi,\tau}} \leq RHS \text{ of } \eqref{Duality4}.
\ee

\noindent\textbf{Region (3).} $\xi> E_{\alpha,\beta,s}$, $ |\xi_{1}-c_1\xi|\le \frac{1}{2\sigma}$ or $|\xi_{1}-c_2\xi|\le \frac{1}{2\sigma}$ .

The frequency $\xi$ in (\ref{xi_decom2}) belongs to $\m{Z}/\sigma$, so in order to take advantage of classical results from Diophantine approximation theory to estimate $|\xi_1-x_1|$ and $|\xi_1 - x_2|$, it is helpful to convert $\xi$ to be an integer. Thus, we introduce 
\[
    \wt{\xi} := \sigma \xi \quad \text{and} \quad\wt{\xi}_{1} := \sigma \xi_1,
\] 
so that both $\wt{\xi}$ and $\wt{\xi}_1$ are integers. Thanks to the constraint that $|\xi| > E_{\a,\b,s}$, (\ref{large_xi2}) automatically holds. 
As a consequence, it follows from (\ref{xi_decom2}) that
\be\label{xi_decom_int}
    x_1 = \frac{1}{\sigma} \bigg[ c_1\wt{\xi} + \frac{\lambda}{\wt{\xi}}\bigg] + Q_1(\xi), 
    \qquad
    x_2 = \frac{1}{\sigma} \bigg[ c_2\wt{\xi} - \frac{\lambda}{ \wt{\xi}}\bigg] + Q_2(\xi).
\ee
Hence,
\be\label{distance_xi}
\left\{\begin{array}{l}
|\xi_{1}-x_1| \geq \dfrac{1}{\sigma}\Bigg|\widetilde{\xi}_1-c_1\widetilde{\xi} - \dfrac{\lambda}{\widetilde{\xi}} \Bigg|-|Q_1(\xi)|, \vspace{0.15in}
\\
|\xi_{1}-x_2| \geq \dfrac{1}{\sigma}\Bigg|\widetilde{\xi}_1-c_2\widetilde{\xi} + \dfrac{\lambda}{\widetilde{\xi}} \Bigg|-|Q_2(\xi)|.
\end{array}\right.\ee

Now we focus only on the region where 
\be\label{loc_xi1}
    |\xi_{1} - c_1\xi| \le \frac{1}{2\sigma}
\ee
since the proof for the region \(|\xi_{1} - c_2\xi| \le \frac{1}{2\sigma}\) is analogous. In this region, we first demonstrate the sizes of $\xi_1$, $\xi_2$ and $\xi$ are comparable as long as $E_{\a,\b,s}$ is large enough. 
Since $c_1 + c_2 = 1$ and $\xi_1 + \xi_2 = \xi$, then
\be\label{res_region}
    |\xi_2 - c_2\xi| = |\xi_1 - c_1\xi| \leq \frac{1}{2\sigma}.
\ee
In addition to the constraint (\ref{const_M1}), we further require that 
\be\label{const_M2}
    E_{\a,\b,s} \geq \frac{1}{|c_2|} = \frac{6}{|3-R_\a|}.
\ee
We point out that the denominator $|c_2|$ or $|3-R_\a|$ is nonzero since $\a\neq 1$. Then as $\xi > E_{\a,\b,s}$ and $\sigma\geq 1$, it follows from (\ref{c1c2}) and (\ref{const_M2}) that 
\be\label{aux1}
    |c_2 \xi| \geq \frac{1}{\sigma} \quad\text{and}\quad |c_1 \xi| \geq \frac{1}{\sigma}. 
\ee
Combining (\ref{const_M2}) with (\ref{aux1}) yields 
\be\label{comparable}
    \langle\xi_1\rangle \sim \langle\xi_2\rangle \sim \langle\xi\rangle,
\ee
which implies that for any $s\geq \frac12$,
\be\label{weight_est1}
	\frac{|\xi|\langle \xi\rangle^s}{\langle MAX_1\rangle^{\frac{1}{2}}\langle\xi_{1}\rangle^s\langle\xi_{2}\rangle^s}\lesssim  \frac{|\xi|^{\frac{1}{2}}}{\la MAX_1\ra^{\frac12}} .
\ee

Next, we estimate the size of the resonance function $H_{\sigma}^{\xi}$ which can be decomposed below (see (\ref{Resonance04})):
\be\label{res5}
    H_{\sigma}^{\xi}(\xi_1) = -3\a\xi(\xi_1-x_1)(\xi_1-x_2).
\ee
Since $\xi_1$ is near $c_1\xi$ and $x_2$ is near $c_2\xi$, the distance between $\xi_1$ and $x_2$ should be close to $(c_1-c_2)\xi$. In fact, 
\[\begin{split}
    |\xi_1-x_2| &= |(\xi_1 - c_1\xi) + (c_1-c_2)\xi + (c_2\xi - x_2)| \\
    &\geq (c_1-c_2)\xi - |\xi_1 - c_1\xi| - |c_2\xi - x_2|,
\end{split}\]
it then follows from (\ref{loc_xi1}) and (\ref{loc_x12}) that 
\[
    |\xi_1-x_2| \geq (c_1-c_2)\xi - \frac{1}{\sigma} = \frac{R_\a}{3}\xi - 1 \geq \frac{R_\a}{12}\xi,
\]
where the last inequality is due to $\xi>E_{\a,\b,s}\geq \frac{4}{R_\a}$ as shown in (\ref{const_M1}). Hence, we deduce from (\ref{res5}) that 
\begin{equation}\label{regin3H}
|H_\sigma^{\xi}(\xi_{1})|\gtrsim |\xi|^2|\xi_{1}-x_1|.  
\end{equation}
In what follows, we divide the discussion into three cases.

\noindent\textbf{Case 3.1.  $R_\alpha \in \mathbb{Q}$.}

In this case, $c_1 = \frac{1}{2}+\frac{R_{\alpha}}{6}\in \mathbb{Q}$. Let 
\begin{equation}\label{Prime}
c_1 = \frac{q}{p}, \qquad \text{where} \quad p, q\in\mathbb{Z^+} \text{ and } \gcd(p,q) = 1.
\end{equation}
Then using (\ref{decom_xi2}) and \eqref{distance_xi}, $|\xi_1 - x_1|$ has a lower bound below:
\be\label{diff_xi_x}
|\xi_{1}-x_1| \geq \frac{1}{\sigma}\left|\widetilde{\xi}_1-c_1\widetilde{\xi} - \frac{\lambda}{\widetilde{\xi}}\right|-|Q_1(\xi)|
\geq \frac{1}{\sigma}\Bigg(\bigg|\frac{p\wt\xi_1-q\wt\xi}p-\frac{\lambda}{\wt{\xi}}\bigg|-\frac{12\lambda^2}{R_\a |\widetilde\xi|^3}\Bigg). 
\ee

\begin{itemize}
	\item When $ p\widetilde\xi_{1}-q\widetilde\xi = 0$, we have 
	$ \frac{12\lambda^2}{R_\a |\widetilde\xi|^3} \le \frac{\lam}{2\widetilde{\xi}} $ since $|\xi|>E_{\a,\b,s}$. Hence, it follows from (\ref{diff_xi_x}) that 
	\[
		|\xi_{1}-x_1|\gtrsim \frac{1}{\sigma^2|\xi|}.
	\]

	\item When $p\widetilde\xi_{1}-q\widetilde\xi \neq 0$, we have
\begin{equation*}
|p\widetilde\xi_{1}-q\widetilde\xi| \geq 1.
\end{equation*} 
So it follows from (\ref{diff_xi_x}) that 
\[
	\sigma |\xi_1 - x_1| \geq \bigg|\frac{p\widetilde\xi_1-q\widetilde\xi}{p}-\frac{\lambda}{\widetilde{\xi}}\bigg|-\frac{12\lambda^2}{R_\a |\widetilde\xi|^3} \geq \frac{1}{p} - \frac{|\lam|}{\wt{\xi}} - \frac{12\lambda^2}{R_\a \widetilde\xi^3}.
\]
Since $\xi$ satisfies (\ref{large_xi2}), then one can directly check that 
\begin{equation}\label{RationalLowerBound}
\sigma |\xi_1 - x_1| \geq \frac{1}{p} - \frac{3 |\lam|}{2\wt{\xi}}.
\end{equation}

To obtain a positive lower bound of \eqref{RationalLowerBound}, we
require 
\begin{equation}\label{const_M3}
	E_{\alpha,\beta,s}\geq 2p|\lambda| = \frac{2p|\beta|}{\alpha R_\alpha},
\end{equation}
which depends only on $\alpha$ and $\beta$, then $\wt{\xi} \geq \xi \geq 2p |\lam|$ and
\[
	\frac{1}{p} - \frac{3 |\lambda|}{2\widetilde{\xi}} \gtrsim \frac{1}{\widetilde{\xi}}.
\]
Plugging this into (\ref{RationalLowerBound}) yields 
\[|\xi_1 - x_1| \gs \frac{1}{\sigma \wt{\xi}} = \frac{1}{\sigma^2 \xi}.\]
\end{itemize}

As a summary, no matter whether $p\widetilde\xi_{1}-q\widetilde\xi = 0$ or not, it always holds that $|\xi_1 - x_1| \gs \frac{1}{\sigma^2 \xi}$. As a result,
\begin{equation*}
	|H_\sigma^{\xi}(\xi_{1})|\gtrsim |\xi|^2|\xi_{1}-x_1|\gtrsim \frac{|\xi|}{\sigma^2}.  
\end{equation*}
Consequently, for any $s\geq \frac12$, it follows from (\ref{weight_est1}) that 
\begin{equation*}
\frac{|\xi|\langle \xi\rangle^s}{\langle MAX_1\rangle^{\frac{1}{2}}\langle\xi_{1}\rangle^s\langle\xi_{2}\rangle^s}
\ls \frac{|\xi|^{\frac{1}{2}}}{\la MAX_1\ra^{\frac12}} \ls \sigma.
\end{equation*}
The rest argument is similar to that for (\ref{TEMPF1}) with an extra coefficient $\sigma$.

\noindent\textbf{Case 3.2. $R_\alpha \notin \mathbb{Q}$ and $\max\{\nu_{\lambda}(c_1), \nu_{\lam}(c_2)\} \geq1$.}

In this case, $s^*(\alpha,\beta) = 1$. Then for $s\geq1$, it follows from (\ref{comparable}) that 
\begin{equation*}
\frac{|\xi|\langle\xi \rangle^s}{\langle\xi_1 \rangle^s\langle\xi_2 \rangle^s}\lesssim 1.
\end{equation*}
Hence, the rest computation is same as that for (\ref{TEMPF1}).

\noindent\textbf{Case 3.3. $R_\alpha\notin \mathbb{Q}$ and $\max\{\nu_{\lambda}(c_1), \nu_{\lam}(c_2)\} < 1$.}

Following the idea in \cite{Oh2009}, we apply the Diophantine approximation theory to handle this case. Recalling Definition \ref{Def, mbti2} which requires a lower bound $N$ for $|n|$, so we need to ensure $|\xi| > N$ in order to apply this theory. Meanwhile, it is important to note that we also have the case \(|\xi_{1} - c_2\xi| \le \frac{1}{2\sigma}\), and for this case, one should substitute \(\nu_{\lambda}(c_2)\) for \(\nu_{\lambda}(c_1)\) in the subsequent proof. This substitution directly gives the desired index \eqref{s_0} as follows:
\[
s_{\alpha,\beta} = \max\{\nu_{\lambda}(c_1),\nu_{\lambda}(c_2)\}.
\]
To facilitate the proper use of the index \eqref{s_0} hereafter, we thus define
\begin{equation}\label{LowerBoundofN}
    N(\alpha,\beta)=\max\{N_1,N_2\},
\end{equation}
where $N_i = N_i(c_i,\lambda)$,\ $i=1,2$, is exactly the lower bound \(N_i\) given in the definition of \(\nu_{\lambda}(c_i)\), and it depends only on \(\alpha\) and \(\beta\). Therefore, we require 
\begin{equation}\label{const_M4}
    E_{\alpha,\beta,s}\geq N(\alpha,\beta) + 1
\end{equation}
to guarantee $|\xi| > N(\a,\b)$ whenever $|\xi| \geq E_{\a, \b,s}$.
Since $s > s^{*}(\a,\b)$, we define a positive parameter $\veps$ below:
\[
    \varepsilon = 2s - 1 - s(\a,\b) = 2[s - s^{*}(\a,\b)].
\]
Then $0 < \veps \leq 2s - 1 - \nu_{\lam}(c_1)$. If $s\geq 1$, then similar to Case 3.2, the desired estimate holds automatically. Next, we assume $s^*(\a,\b) < s < 1$, which implies that 
\[
    \nu(c_1) + \veps \leq 2s-1 < 1.
\]

With \eqref{distance_xi} and (\ref{decom_xi2}), it holds that
\begin{equation}\label{xi1minusx1}
\begin{split}
|\xi_{1}-x_1| &\geq 
\frac{|\widetilde\xi|}{\sigma}\bigg|c_1-\frac{\widetilde{\xi}_1}{\widetilde\xi}+\frac{\lambda}{\widetilde{\xi}^2}\bigg| - |Q_1(\xi)| \\
&\geq |\xi| \bigg(\bigg|c_1-\frac{\widetilde{\xi}_1}{\widetilde\xi}+\frac{\lambda}{\widetilde{\xi}^2}\bigg| - \frac{12\lambda^2}{R_\a|\widetilde\xi|^4} \bigg).
\end{split}\end{equation}
According to the definition of $\nu_{\lam}(c_1)$, there exist a positive number $K_1 = K_1(c_1, \nu_{\lam}(c_1)+\veps, \lam)$, which only depends on $\a$, $\b$ and $s$, and a positive integer $N_1 = N_1(c_1, \lam)$, such that 
\[ 
    \Big| c_1 - \frac{m}{n} + \frac{\lam}{n^2} \Big| \geq \frac{K_1}{|n|^{2 + \nu_{\lam}(c_1)+\veps}}, \qquad \forall\, (m,n)\in\m{Z}^2, \, |n| > N_1.
\]
In (\ref{xi1minusx1}), $(\wt{\xi}_1, \wt{\xi})\in\m{Z}^2$ and $\wt{\xi} \geq \xi > N_1$, so we can plug the above property into (\ref{xi1minusx1}) to find 
\begin{equation}\label{xi1minusx2}
|\xi_{1}-x_1| \geq |\xi|\left(\frac{K_1}{|\widetilde{\xi}|^{2+\nu_{\lambda}(c_1) + \varepsilon}} - \frac{12\lambda^2}{R_\a|\widetilde\xi|^4} \right).
\end{equation}
Now we further impose a lower bound for $\xi$ by requiring 
\be\label{const_M5}
    E_{\a,\b,s} \geq \frac{24 \lam^2}{R_\a K_1},
\ee
then we have $\wt{\xi} \geq \xi \geq \frac{24 \lam^2}{R_\a K_1}$, which implies that 
\[
    \frac{12\lambda^2}{R_\a|\widetilde\xi|^4} \leq \frac12 \frac{K_1}{|\widetilde{\xi}|^{3}} \leq \frac12 \frac{K_1}{|\widetilde{\xi}|^{2+\nu_{\lambda}(c_1) + \varepsilon}},
\]
where the last inequality is due to the fact that $\nu_{\lambda}(c_1) + \varepsilon \leq 1$. Plugging this estimate into (\ref{xi1minusx2}) leads to 
\begin{equation}\label{xi1minusx3}
\begin{split}
|\xi_{1}-x_1| \geq 
\frac12 |\xi| \frac{ K_1}{|\widetilde{\xi}|^{2+\nu_{\lambda}(c_1) + \varepsilon}} \gs \frac{|\xi|^{-1-\nu_{\lambda}(c_1)-\varepsilon}}{\sigma^{2+\nu_{\lambda(c_1)}+\varepsilon}}.
\end{split}
\end{equation}
Together with  \eqref{regin3H} and \eqref{xi1minusx3}, it holds that
\begin{equation}\label{32a}
|H_\sigma^{\xi}(\xi_{1})|\gtrsim |\xi|^2|\xi_{1}-x_1|\gtrsim \frac{|\xi|^{1-\nu_{\lambda}(c_1)-\varepsilon}}{\sigma^{2+\nu_{\lambda}(c_1)+\varepsilon}}.
\end{equation}
As a summary, by choosing $E_{\a,\b,s}$ to satisfy (\ref{const_M1}), (\ref{const_M2}), (\ref{const_M3}), (\ref{const_M4}) and (\ref{const_M5}), we justify the lower bound estimate (\ref{32a}) for the resonance function $|H_\sigma^{\xi}(\xi_{1})|$ in Region (3) for the Case 3.3. 

Finally, thanks to the lower bound estimate (\ref{32a}), the remaining argument is the same as that in Oh's paper (see the argument starting from equation (46) in Case 3 in Part I in the proof of Proposition 3.7 in \cite{Oh2009}).

\end{proof}

\subsubsection{Proof of the auxiliary part of the bilinear estimate}

\begin{proof}[Proof of \eqref{Bilinear3, p2}]

For this part, it suffices to prove that:
\begin{equation}\label{Part23}
\|\langle L\rangle^{-\frac{1}{2}}\mathscr{B}_{s}(f_1,f_2)\|_{L_\xi^2(\mathbb{Z}/\sigma)L_\tau^1(\mathbb{R})}\lesssim \sigma \|f_1\|_{L_{\xi_{1},\tau_1}^2}\|f_2\|_{L_{\xi_{2},\tau_2}^2}.
\end{equation}

We first denote $E_{\a,\b,s}$ to be the same constant as that in the proof of (\ref{Bilinear3, p1}). Then we proceed by dividing the region into several subregions.

\noindent\textbf{Region (1).} $|\xi|\le E_{\alpha,\beta,s}$, or $|\xi| > E_{\alpha,\beta,s}$ with $MAX_1 = \langle L_1\rangle$ or $MAX_1 = \langle L_2 \rangle$.

The proof for this region is identical to that for Region (1) and Region (2) in the proof of \eqref{Bilinear1, p2} in Section \ref{Sec, pos_beta_nr}, so we omit the details here.

\noindent\textbf{Region (2).} $|\xi| > E_{\alpha,\beta,s}$ and $MAX_1 = \langle L \rangle$.

\noindent\textbf{Region (2.1).} $\langle L_1\rangle \geq \frac{1}{2}\langle L\rangle^{6\varepsilon}$ or $\langle L_2\rangle \geq \frac{1}{2}\langle L\rangle^{6\varepsilon}$, where $\veps$ is some small number in $(0, \frac{1}{100})$.

In this case, we have the following inequality:
\[
    \frac{\la L \ra^{\veps}}{\la L \ra^{1/2} \la L_1 \ra^{1/2}\la L_2 \ra^{1/2}}  \ls \frac{1}{\la L \ra^{1/2} \la L_1 \ra^{1/3}\la L_2 \ra^{1/3}}.
\]
The purpose of the above inequality is to eliminate the term $\la L \ra^{\veps}$ in the numerator by paying the price of lowering the powers of $\la L_1\ra$ and $\la L_2 \ra$ from $1/2$ to $1/3$. We point out that the power $1/3$ suffices to obtain the desired result due to Lemma \ref{lemma4}. For example, the estimate \eqref{TEMPF1} is still valid if the terms $\langle L_1\rangle^{\frac{1}{2}}$ and $\la L_2\ra^{\frac{1}{2}}$ in \eqref{TEMPF1} and \eqref{TEMPD1} are replaced with $\la L_1\ra^{\frac{1}{3}}$ and $\la L_2\ra^{\frac{1}{3}}$ respectively.

\noindent\textbf{Region (2.2).} $\langle L_1\rangle \le \frac{1}{2}\langle L\rangle^{6\varepsilon}$ and $\langle L_2\rangle \le \frac{1}{2}\langle L\rangle^{6\varepsilon}$.

Recalling (\ref{res_fn_general}) where
\begin{equation*}
H_\sigma^\xi(\xi_1) =L_1+L_2-L,
\end{equation*} 
so $| H_\sigma^\xi(\xi_1) | \sim |L| \gs (|L_1| + |L_2|)^{\frac{1}{6\veps}}$, which implies that 
\begin{equation*}
	L = -H_\sigma^\xi(\xi_1) - (L_1 + L_2) = -H_\sigma^\xi(\xi_1)+o(|H_\sigma^\xi(\xi_1)|^{10\varepsilon}).
\end{equation*}
Let 
\[
	\Omega(\xi) = \{\eta \in \mathbb{R} : \eta = -H_\sigma^\xi(\xi_1)+o(|H_\sigma^\xi(\xi_1)|^{10\varepsilon}) \text{ for some } \xi_1 \in \mathbb{Z}/\sigma\}.
\] 
Then following the strategy in the proof of (\ref{Bilinear1, p2}) in Case 1 for Region (3.2), it suffices to prove 
\begin{equation}\label{Count3}
|\Omega(\xi) \cap \{| \eta | \sim M\}| \lesssim \sigma M^{\frac{2}{3}},
\end{equation} 
for all $|\xi| > E_{\a,\b,s}$ and for all dyadic $M \ge 1$.

To prove \eqref{Count3}, without loss of generality, we assume $\xi$ is positive and $|\xi_1| \ge |\xi_2|$ since $H_\sigma^\xi(\xi_1)$ is symmetric in $\xi_1$ and $\xi_2$. In addition, for any $\eta\in\Omega(\xi)$ with $|\eta| \sim M$, it is readily seen that $|H_\sigma^\xi(\xi_1)| \sim M$. For any such $\xi_1$, there holds
\be\label{eta_number}
\left|\left\{\eta \in \mathbb{R} : |\eta| \sim M, \eta = -H_\sigma^\xi(\xi_1)+o(|H_\sigma^\xi(\xi_1)|^{10\varepsilon})\right\}\right| \ls M^{10\varepsilon}. 
\ee
Now for any fixed $\xi$ with $|\xi| > E_{\a,\b,s}$, we estimate the number $N_{\xi}$ of possible values of $\xi_1 \in \mathbb{Z}/\sigma$ such that
$ |H_\sigma^\xi(\xi_1)| \sim M$, that is,
\[
	N_{\xi} := \#\{\xi_{1}\in \mathbb{Z}/\sigma: | H_{\sigma}^{\xi}(\xi_1) |\sim M\}.
\]
Thanks to (\ref{eta_number}) and the choice $\veps<\frac{1}{100}$, we know
\[\text{LHS of (\ref{Count3})} \ls N_{\xi} M^{10\veps} \leq N_{\xi} M^{\frac{1}{10}}.\]
So in order to justify (\ref{Count3}), it suffices to prove 
\be\label{eta_count}
	N_\xi \ls \sigma M^{\frac23-\frac{1}{10}}.
\ee

Since we assume $\xi$ to be positive and $|\xi_{1}|\geq |\xi_{2}|$, then $\xi_1\geq \frac{\xi}{2}$. Thus, we divide the domain into three subregions according to whether they are close to $c_1\xi$.

\noindent\textbf{Region (2.2a)} $\frac{\xi}{2}\le \xi_{1}\le c_1\xi-1$.

Recalling (\ref{H_factor}) which shows
\be\label{res_cpt_main}
H_{\sigma}^{\xi}(\xi_1) = -3\a\xi (\xi_1 - c_1\xi)(\xi_1-c_2\xi) + \frac{\b\xi}{\sigma^2} =  -3\a\xi \bigg(\Big(\xi_1-\frac{\xi}{2}\Big)^2 + \frac{\a-4}{12\a}\xi^2-\frac{\b}{3\a \sigma^2}\bigg).
\ee
Noticing that $H_{\sigma}^{\xi}(\xi_1)$, as a function in $\xi_1$, is decreasing in $[\xi/2,c_1\xi-1]$. Therefore, we only need to compute the values of $H_{\sigma}^{\xi}(\xi_1)$ at the endpoints \(\xi/2\) and \(c_1\xi - 1\) to determine its range. Direct computation yields
\begin{equation*}
H_\sigma^\xi\left(\frac{\xi}{2}\right) = 3\alpha\xi \left(\frac{4-\a}{12\a}\xi^2+\frac{\beta}{3\alpha \sigma^2}\right) 
=3\alpha\xi \left(\frac{R_\a^2}{36}\xi^2 + \frac{\beta}{3\alpha \sigma^2}\right)
\end{equation*}
and
\begin{equation*}
H_\sigma^\xi(c_1\xi-1) = 3\alpha\xi\left((c_1-c_2)\xi - 1+\frac{\beta}{3\alpha\sigma^2}\right),
\end{equation*}
where $|\xi| > E_{\alpha,\beta,s}$. Thanks to (\ref{const_M1}), we have 
\[
	\frac{R_\a^2}{36} \xi^2 \geq \frac{2|\beta|}{3\alpha} \qquad\text{and}\qquad (c_1 - c_2)\xi \geq 2 + \frac{2|\b|}{3\a},
\]
so both $H_\sigma^\xi(\frac{\xi}{2})$ and $H_\sigma^\xi(c_1\xi-1)$ are positive, and
\[
\left|H_\sigma^\xi\left(\frac{\xi}{2}\right)\right| \sim |\xi|^3
\quad\text{and}\quad
\left|H_\sigma^\xi(c_1\xi-1)\right| \sim |\xi|^2,
\]
which implies that $|H_\sigma^\xi(\xi_1)| \sim |\xi|^p$ for some $p \in [2,3]$, i.e. $|\xi| \sim |H_\sigma^\xi(\xi_1)|^{1/p} \sim M^{1/p}$.
Now for any fixed $\xi$ such that $|\xi| > E_{\a,\b,s}$ and $|\xi| \sim M^{1/p}$, the length of the interval $[\xi/2, c_1\xi-1]$ is of size $M^{1/p} \leq M^{1/2}$, so there are at most $\sigma M^{1/2}$ possible values of $\xi_1$ in this interval, which justifies (\ref{eta_count}).

\noindent\textbf{Region (2.2b)}. $c_1\xi - 1 \leq \xi_1 \leq c_1\xi+1$.

In this case, there are at most $2\sigma$ many possible values of $\xi_1$, which also satisfies (\ref{eta_count}).

\noindent\textbf{Region (2.2c)}. $\xi_1 \ge c_1\xi+1$.

In this case, it follows from (\ref{res_cpt_main}) that $H_\sigma^\xi(\xi_1)$ is decreasing and 
\[
H_\sigma^\xi(c_1\xi+1) = -3\a \xi \left((c_1-c_2)\xi+1-\frac{\b}{3\a \sigma^2}\right),
\]
where $|\xi|>E_{\a,\b,s}$. Recalling that $(c_1 - c_2)\xi \geq 2 + \frac{2|\b|}{3\a}$, so we have $H_\sigma^\xi(c_1\xi+1) < 0$ and $|H_\sigma^\xi(c_1\xi+1)| \gs |\xi|^2$. Meanwhile, since $H_\sigma^\xi(\xi_1)$ is decreasing on $[c_1\xi+1, \infty)$, we know $H_\sigma^\xi(\xi_1) < 0$ and $M\sim|H^{\xi}_{\sigma}(\xi_1)|\gtrsim |\xi|^2$. Note that $H_\sigma^\xi(\xi_1)$ can be rewritten as
\begin{equation*}
H_\sigma^\xi(\xi_1) = -3\alpha\xi\Big(\xi_1-\frac{\xi}{2}\Big)^2 +\frac{\alpha R_\alpha^2}{12}\xi^3+\frac{\xi\beta}{\sigma^2}.
\end{equation*}
Let $\xi\sim N$ be fixed, where $N$ is a dyadic number. 
Denote 
$h^\xi(\xi_{1}) = -3\alpha\xi(\xi_{1}-\frac{\xi}{2})^2$.
Then 
\[
    h^\xi(\xi_{1}) = H_\sigma^\xi(\xi_1) - \frac{\alpha R_\alpha^2}{12}\xi^3 - \frac{\xi\beta}{\sigma^2},
\]
which implies that 
\[
    |h^\xi(\xi_{1})| \leq |H_\sigma^\xi(\xi_1)| + \frac{\alpha R_\alpha^2}{6}\xi^3 \leq |H_\sigma^\xi(\xi_1)| + 2\xi^3.
\]
Therefore, 
\begin{equation}\label{Numberxi}
\#\{\xi_1 \in \mathbb{Z}/\sigma: \xi_1 \ge c_1\xi+1 \text{ and } |H_\sigma^\xi(\xi_1)| \sim M\} \le \#\{\xi_1 \in \mathbb{Z}/\sigma: \xi_1 \ge c_1\xi+1 \text{ and } |h^\xi(\xi_1)|\ls M + N^3\}.
\end{equation}
Recalling that $h^\xi(\xi_{1}) = -3\alpha\xi(\xi_{1}-\frac{\xi}{2})^2$ and $M\sim|H^{\xi}_{\sigma}(\xi_1)|\gtrsim |\xi|^2 \sim N^2$, then
\[
\text{RHS of \eqref{Numberxi} } \lesssim \sigma\left(\frac{M+N^3}{N}\right)^{\frac{1}{2}}\ls \sigma M^{\frac{1}{2}},
\]
which verifies (\ref{eta_count}).
\end{proof}

Thus, the justification of (\ref{Bilinear3}) is finished. Next, we will illustrate why (\ref{Bilinear4}) needs the extra assumption that the function $w_1$ has zero mean.  
Recalling the resonance function \eqref{H_tilde_root}:
\[
 \wt{H}^{\xi_1}(\xi) := 3\a\xi_1 \bigg[ (\xi - c_1 \xi_1)(\xi - c_2 \xi_1) - \frac{\b}{3\a} \bigg].
\]
When $\xi_1$ is fixed, we regard $\wt{H}^{\xi_1}(\xi)$ as a function of $\xi$, which is similar to the resonance function \eqref{H_res_root} for \eqref{Bilinear3}. Hence the proof of the second bilinear estimate \eqref{Bilinear4} are
analogous to \eqref{Bilinear3}. The major difference here is the extra singularity induced by $\xi_1$ for \eqref{Bilinear4}. To ensure the estimate is valid when $\xi_1 = 0$, the assumption that $\mathscr{F}_{x}{w_1}(0, t) = 0$ for any $t$, i.e., the mean value of $w_1$ is zero, is necessary.

Hence, Proposition \ref{Prop2} has been established.

\section{The Ill-Posedness Results}
\label{Sec, ip}
Recalling that a solution map being $C^k$ ($k \ge 1$) means that there exists $T>0$ such that the map from the initial data $(u_0,v_0) \in {H}^s(\mathbb{T})\times{H}^s(\mathbb{T})$ to the local solution $(u,v) \in C([0,T];{H}^s(\mathbb{T})\times{H}^s(\mathbb{T}))$ is $C^k$. Taking $(u_0,v_0)=(\delta\phi,\delta\psi)$, so that \eqref{mMajdaBiello} becomes
\begin{equation}\label{DMajdaBiello}
\left\{
\begin{aligned}
	&u_t+u_{xxx}+vv_x=0, \\
	&v_t+\alpha v_{xxx}+\beta v_x+(uv)_x=0, \\
	&(u_0,v_0) = (\delta\phi(x),\delta\psi(x)),
\end{aligned}\right. \qquad x\in\mathbb{T}, \, t\in\mathbb{R},
\end{equation}
where $\delta \ge 0$ and $(\phi,\psi) \in {H}^s(\mathbb{T})\times{H}^s(\mathbb{T})$. 

Denote the solution of \eqref{DMajdaBiello} to be $( u(x, t,\delta), v(x, t,\delta) )$. Then it follows from the Duhamel's principle that
\begin{equation}\label{solnform}
\left\{\begin{aligned}
	u(x, t,\delta) &= \delta S(t)\phi(x) - \frac{1}{2} \int_{0}^{t} S(t-t')\partial_x(v^2)(x, t', \delta) \, dt', \\
	v(x, t,\delta) &= \delta S_{\alpha,\beta}(t)\psi(x) - \int_{0}^{t} S_{\alpha,\beta}(t-t')\partial_x(uv)(x, t', \delta) \, dt',
\end{aligned}\right.
\end{equation}
where $S(t)$ and $S_{\a,\b}(t)$ are the semigroup operators defined as in (\ref{semigroup op}).

When $\delta=0$, the initial data $(u_0, v_0)$ in (\ref{DMajdaBiello}) is $(0,0)$ and the unique solution is also $(0,0)$, which means $(u(t,0),v(t,0))=(0,0)$. Then by taking derivative of $(u(x,t,\delta), v(x,t,\delta))$ with respect to $\delta$ at $0$, it follows from (\ref{solnform}) that
\be\label{phi1psi1}
\left\{
\begin{aligned}
	\partial_\delta u(x,t,0) &= [S(t)\phi](x) =: \phi_1(x,t), \\
	\partial_\delta v(x,t,0) &= [S_{\alpha,\beta}(t)\psi](x) =: \psi_1(x,t).
\end{aligned}\right.
\ee
For convenience, let \[P_{\alpha,\beta}(\eta) = \alpha\eta^3-\beta\eta,\] then
\[
\left\{\begin{aligned}
	& \mathscr{F}_{x}\phi_1(\xi, t) = \mathscr{F}_{x} [S(t)\phi](\xi) = e^{itP_{1,0}(\xi)}\widehat\phi(\xi), \\
	&\mathscr{F}_{x}\psi_1(\xi,t) = \mathscr{F}_{x}[S_{\alpha,\beta}(t)\psi](\xi)  = e^{itP_{\alpha,\beta}(\xi)}\widehat \psi(\xi),
\end{aligned}\right.
\]
where $\mathscr{F}_{x}$ represents the Fourier transform with respect to the spatial variable $x$. By taking the second and third derivatives of $(u(x,t,\delta), v(x,t,\delta))$ in terms of $\delta$ at $0$, we have
\begin{equation}\label{phi2psi2}
\left\{
\begin{aligned}
	\partial_\delta^2 u(x,t,0) &= -\int_0^t S(t-t')\partial_x(\psi_1^2)(x,t') \, dt' =: \phi_2(x,t), \\
	\partial_\delta^2 v(x,t,0) &= -2\int_0^t S_{\alpha,\beta}(t-t')\partial_x(\phi_1\psi_1)(x,t') \,dt' =: \psi_2(x,t),
\end{aligned}\right.
\end{equation}
and
\begin{equation}\label{phi3psi3}
\left\{
\begin{aligned}
	\partial_\delta^3 u(x,t,0) &= -3\int_0^t S(t-t')\partial_x(\psi_1\psi_2)(x,t') \,dt' =: \phi_3(x,t), \\
	\partial_\delta^3 v(x,t,0) &= -3\int_0^t S_{\alpha,\beta}(t-t')\partial_x(\phi_1\psi_2+\phi_2\psi_1)(x,t') \,dt' =: \psi_3(x,t).
\end{aligned}\right.
\end{equation}
Note that if the solution map is $C^k$, then there exists a time $T>0$ and a constant $C$ such that
\begin{equation}\label{Mainidea}
\sup_{0\le t \le T} \Vert(\phi_k,\psi_k)(\cdot,t)\Vert_{{H}^s(\mathbb{T})\times{H}^s(\mathbb{T})} \le C\Vert(\phi,\psi)\Vert_{{H}^s(\mathbb{T})\times{H}^s(\mathbb{T})}^k, \quad k=1,2,3.
\end{equation}

\subsection{Proof of Theorem \ref{thm3}}
\noindent {\bf Case 1}: $\alpha = 4$ and $\beta = 3n^2$ for some $n\in \mathbb{N}$.

In this case, we will prove $s \geq 1$ when the solution map is at least $C^2$.
For any positive integer $N \geq 10 |n| + 10$ such that both $\frac{N+n}{2}$ and $\frac{N-n}{2}$ are integers, we define 
\[
    \phi(x) = 0, \quad \psi(x) = \frac{1}{\pi N^{s}} \Big[\cos\Big(\frac{N+n}{2} x\Big) + \cos\Big(\frac{N-n}{2} x\Big)\Big], \quad \forall\, x\in\m{R}.
\]
As a result, 
\begin{equation}\label{Intial_Ill_4_3n}
\widehat\phi(\xi) = 0, \quad 
\widehat\psi(\xi) = \frac{1}{N^s}\left(\delta\Big(|\xi|-\frac{N+n}{2}\Big)+\delta\Big(|\xi|-\frac{N-n}{2}\Big)\right), \quad\forall \, \xi \in\m{Z},
\end{equation}
where $\delta$ represent the Dirac delta function. Since $N \geq 10|n|+10$, it is readily seen that 
$\|(\phi,\psi)\|_{H^s \times H^s}\sim 1$.
Meanwhile, since the solution map is assumed to be $C^2$, then it follows from (\ref{Mainidea}) with $k=2$ that 
\be\label{phi2bdd}
	\sup_{0\le t \le T} \|\phi_2(\cdot, t) \|_{H^s(\m{T})}
	\ls \Vert(\phi,\psi)\Vert_{{H}^s(\m{T})\times{H}^s(\m{T})}^2 \lesssim 1,
\ee
where $T$ is some fixed positive number. On the other hand, based on the definition of $\phi_2$ in (\ref{phi2psi2}), for any $t\in (0,T)$, the Fourier transform of $\phi_2(\cdot, t)$ in the spatial variable is given below:
\begin{equation}\label{phi2}
\mathscr{F}_{x}\phi_2(\xi, t) = -\xi e^{itP_{1,0}(\xi)}\int_{\mathbb{Z}}\int_0^t e^{it'G(\xi_1,\xi-\xi_{1},-\xi)}\widehat\psi(\xi_1)\widehat\psi(\xi-\xi_{1})\,dt' \,d\xi_{1},
\end{equation}
where the function $G$ is defined on $\Gamma_3$ as follows:  
\be\label{G}
	G(\eta_1,\eta_2,\eta_3) = P_{\alpha,\beta}(\eta_1)+P_{\alpha,\beta}(\eta_2)+P_{1,0}(\eta_3), \quad \forall\, (\eta_1, \eta_2, \eta_3) \in \Gamma_3.
\ee
When $\a=4$, we plug $\eta_2 = -(\eta_1+\eta_3)$ into the above formula to obtain
\be\label{G_cpt}
    G(\eta_1,\eta_2,\eta_3) = -3\eta_3 \Big((\eta_3 + 2\eta_1)^2-\frac{\b}{3}\Big).
\ee

Next, we estimate $|\F_{x}\phi_2(N, t)|$ which can be written as follows due to (\ref{phi2}):
\[
	|\mathscr{F}_{x}\phi_2(N, t)| = N \bigg| \int_{\mathbb{Z}}\int_0^t e^{it'G(\xi_1, N-\xi_{1}, -N)} \widehat\psi(\xi_1) \widehat\psi(N-\xi_{1})\,dt' \,d\xi_{1} \bigg|.
\]
Since the support of $\wh{\psi}$ is the set of four points: $\big\{\pm\frac{N+n}{2}, \pm\frac{N-n}{2}\big\}$, then $\xi_1$ has to be either $\frac{N+n}{2}$ or $\frac{N-n}{2}$ so that both $\xi_1$ and $N-\xi_1$ are in the support. Therefore, 
\[\begin{split}
	|\F_{x}\phi_2(N, t)| &= \frac{N}{2\pi} \bigg| \int_0^t \Big[ e^{it'G(\frac{N+n}{2}, \frac{N-n}{2}, -N)} + e^{it'G(\frac{N-n}{2}, \frac{N+n}{2}, -N)} \Big] \frac{1}{N^{2s}}\,dt' \bigg|.
\end{split}\]
Since $\b=3n^2$, then it follows from (\ref{G_cpt}) that
\begin{equation*}
G\Big(\frac{N+n}{2},\frac{N-n}{2}, -N\Big) = 0 = G\Big(\frac{N-n}{2},\frac{N+n}{2}, -N\Big),
\end{equation*}
which implies that 
\[
    |\F_{x}\phi_2(N, t)| = \frac{N}{2\pi} \frac{t}{N^{2s}} = \frac{t}{2\pi} N^{1-2s}.
\]
Hence, 
\begin{equation*}
\|\phi_2(\cdot,t)\|_{H^s(\m{T})}^2 =\int_\mathbb{Z} \langle \xi\rangle^{2s} |\mathscr{F}_x\phi_2(\xi,t)|^2 \dd\xi\geq \frac{1}{2\pi} N^{2s}|\F_{x} \phi_2(N,t)|^2 \geq \frac{t^2}{8\pi^3} N^{2-2s}.
\end{equation*}
Meanwhile, since $\|\phi_2(\cdot,t)\|_{H^s(\m{T})} \ls 1$ in (\ref{phi2bdd}), then we conclude $s\geq 1$ by sending $N\to\infty$.

\bigskip

\noindent{\bf Case 2}: $\alpha = 4$ and $\beta\neq 3n^2$ for any $n\in \mathbb{N}$.

In this case, we will prove $s \geq \frac12$ when the solution map is at least $C^3$. For any positive integer $N$, we define
\[
\phi(x) = 0,\quad \psi(x) = \frac{1}{\pi N^s}\cos(Nx),\quad \forall\, x\in \m{R},
\]
then
\begin{equation*}
	\widehat\phi = 0,\quad 
	\widehat\psi = \frac{1}{N^s}\delta(|\xi|-N), \quad \forall\, \xi\in\m{Z}.
\end{equation*}
which implies that $\|(\phi,\psi)\|_{H^s(\m{T}) \times H^s(\m{T})}\sim 1$. 
Meanwhile, since the solution map is assumed to be $C^3$, then it follows from (\ref{Mainidea}) with $k=3$ that 
\be\label{psi3bdd}
\sup_{0\le t \le T} \|\psi_3(\cdot, t) \|_{H^s(\m{T})}
\ls \Vert(\phi,\psi)\Vert_{{H}^s(\m{T})\times{H}^s(\m{T})}^3 \lesssim 1,
\ee
where $T$ is some fixed positive number. On the other hand, since $\phi = 0$, both $\phi_1$ and $\psi_2$ are identical zero functions based on (\ref{phi1psi1}) and (\ref{phi2psi2}). As a result, it follows from (\ref{phi3psi3}) that for any $t\in (0,T)$,
\[
\psi_3(x,t) = -3\int_0^t S_{\alpha,\beta}(t-t')\partial_x(\phi_2\psi_1)(x,t') \,dt'.
\]

Next, we compute $\F_{x} \psi_3(\xi,t)$. Based on the definition of the semigroup operator $S_{\a,\b}$ in (\ref{semigroup op}) and the above expression for $\psi_3(x,t)$, we find 
\[
\F_{x} \psi_3(\xi,t) = -3 \int_{0}^{t} e^{i P_{\a,\b}(\xi)(t-t')} \xi \F_{x}(\phi_2 \psi_1)(\xi,t') \dd t'.
\]
Noting that the polynomial $P_{\a,\b}$ is an odd function, so $\F_{x} \psi_3(\xi,t)$ can be rewritten as
\be\label{psi3F}
\F_{x} \psi_3(\xi,t) = -3\xi e^{i P_{\a,\b}(\xi)t} \int_{0}^{t} e^{i P_{\a,\b}(-\xi)t'} \F_{x}(\phi_2 \psi_1)(\xi,t') \dd t'.
\ee
Since $\psi_1(x,t) = S_{\a,\b}(t)\psi(x)$, then 
\be\label{Fconv}\begin{split}
	\F_{x}(\phi_2 \psi_1)(\xi,t') &= \int_{\m{Z}} \F_{x}\phi_2(\xi_1, t') \F_{x}\psi_1(\xi - \xi_1, t') \,d\xi_1 \\
	&= \int_{\m{Z}} \F_{x}\phi_2(\xi_1, t') e^{i P_{\a,\b}(\xi-\xi_1) t'} \wh{\psi}(\xi-\xi_1) \,d\xi_1.
\end{split}\ee
According to (\ref{phi2psi2}), 
\be\label{phi2F}\begin{split}
	\F_{x}\phi_2(\xi_1, t') &= -\int_{0}^{t'} e^{i P_{1,0}(\xi_1)(t'-\tau)} \xi_1 \F_{x}(\psi_1^2)(\xi_1, \tau) \,d\tau \\
	&= -\xi_1 e^{i P_{1,0}(\xi_1)t'}\int_{0}^{t'} e^{i P_{1,0}(-\xi_1)\tau} \F_{x}(\psi_1^2)(\xi_1,\tau) \,d\tau,
\end{split}\ee
where 
\be\label{psi1SqF}\begin{split}
	\F_{x}(\psi_1^2)(\xi_1, \tau) &= \int_{\m{Z}} \F_{x}\psi_1(\xi_2, \tau) \F_{x}\psi_1(\xi_1 - \xi_2, \tau) \,d\xi_2 \\
	&= \int_{\m{Z}} e^{i [P_{\a,\b}(\xi_2) + P_{\a,\b}(\xi_1-\xi_2)] \tau} \wh{\psi}(\xi_2)\wh{\psi}(\xi_1 - \xi_2) \,d\xi_2.
\end{split}\ee
Plugging (\ref{psi1SqF}) into (\ref{phi2F}) yields 
\be\label{phi2F2}
\F_{x} \phi_2(\xi_1, t') = -\xi_1 e^{i P_{1,0}(\xi_1)t'}\int_{0}^{t'} \int_{\m{Z}} e^{i G(\xi_2, \xi_1-\xi_2,-\xi_1) \tau} \wh{\psi}(\xi_2) \wh{\psi}(\xi_1-\xi_2) \,d\xi_2 \,d\tau,
\ee
where $G$ is the function defined in (\ref{G}). Now putting (\ref{phi2F2}) into (\ref{Fconv}) leads to 
\[\begin{split}
	&\F_{x}(\phi_2 \psi_1)(\xi,t') \\
	\quad= & -\int_{\m{Z}} \xi_1 e^{i [P_{\a,\b}(\xi-\xi_1) + P_{1,0}(\xi_1)] t'} \wh{\psi}(\xi-\xi_1)
	\bigg( \int_{0}^{t'} \int_{\m{Z}} e^{i G(\xi_2, \xi_1-\xi_2,-\xi_1) \tau} \wh{\psi}(\xi_2) \wh{\psi}(\xi_1-\xi_2) \,d\xi_2\,d\tau \bigg)\,d\xi_1.
\end{split}\]
Finally, combining the above formula with \eqref{psi3F} gives
\be\label{psi3}\begin{split}
	\F_{x} \psi_3(\xi,t) &= -3\xi e^{i P_{\a,\b}(\xi)t}
	\int_{0}^{t} \int_{\m{Z}} \xi_1 e^{i G(-\xi,\xi-\xi_1,\xi_1) t'} \wh{\psi}(\xi-\xi_1) \\
	&\qquad 
	\bigg( \int_{0}^{t'} \int_{\m{Z}} e^{i G(\xi_2, \xi_1-\xi_2,-\xi_1) \tau} \wh{\psi}(\xi_2) \wh{\psi}(\xi_1-\xi_2) \,d\xi_2\,d\tau \bigg)\,d\xi_1 \,dt'.
\end{split}\ee

Next, we fix $\xi = N$ in $\F_{x}\psi_3(\xi, t)$ and estimate $\F_{x}\psi_3(N, t)$. Then it follows from (\ref{psi3}) that 
\[
    |\F_{x} \psi_3(N, t)| 
    =3N 
    \bigg|\int_{\m{Z}}\int_{\m{Z}} \Psi_1(\xi_1, \xi_2) 
    \bigg( \int_{0}^{t} e^{i G(-N,N-\xi_1,\xi_1) t'}
    \int_{0}^{t'} e^{i G(\xi_2, \xi_1-\xi_2,-\xi_1) \tau} \,d\tau \,dt'\bigg)
    \dd\xi_2 \dd\xi_1\bigg|,
\]
where 
\begin{equation*}
    \Psi_1(\xi_1, \xi_2) := \xi_1 \wh{\psi}(N-\xi_1) \wh{\psi}(\xi_2) \wh{\psi}(\xi_1-\xi_2).
\end{equation*}
Define $\mathcal{D}$ to be the support of $\Psi_1$, that is
\[\mathcal{D} := \bigl\{ (\xi_1,\xi_2)\in\m{Z}^2 :\Psi_1(\xi_1,\xi_2) \neq 0 \bigr\}.\]
Then we only need to consider the points in $\mathcal{D}$ to compute $\F_{x} \psi_3(N, t)$, so we require $\xi_1\neq 0$ and 
\[
\xi_2, \, N - \xi_1,\, \xi_1-\xi_2 \in \text{supp}\ \widehat\psi = \{N,-N\}.
\]
Thus $\mathcal{D}$ contains exactly one point:
\[
\mathcal{D} = \{(2N,N)\}.
\]
Based on $\mathcal{D}$, we have
\begin{equation}\label{Temp219}
    \begin{aligned}
    |\mathscr{F}_x\psi_3(N,t)| = \frac{3N}{4\pi^2 N^{3s}} \bigg|\sum_{(\xi_1,\xi_2)=(2N,N)} \xi_{1} 
    \bigg(\int_{0}^{t} e^{i G(-N,N-\xi_1,\xi_1) t'}
    \int_{0}^{t'} e^{i G(\xi_2, \xi_1-\xi_2,-\xi_1) \tau} \,d\tau \,dt'\bigg) \bigg|.
    \end{aligned}
\end{equation}

Recall (\ref{G_cpt}) which shows that for any $(\eta_1, \eta_2, \eta_3)\in\Gamma_3$,
\[\begin{split}
G(\eta_1,\eta_2,\eta_3) &= -3\eta_3 \Big((\eta_3 + 2\eta_1)^2-\frac{\b}{3}\Big),
\end{split}\]
so
\[
    G(-N,-N,2N) = 2\b N \quad\text{and}\quad G(N,N,-2N) = -2\beta N.
\]
Therefore
\[
\begin{aligned}
	\text{RHS of }\eqref{Temp219} &\sim N^{2-3s}\bigg|\int_{0}^t e^{2\b Ni t'}\int_{0}^{t'}e^{-2\b Ni \tau} \dd \tau\dd t'\bigg|\\
	& = \frac{N^{1-3s}}{2\b} \bigg| \int_{0}^{t}\big( 1 - e^{2\b N i t'}\big) \,dt' \bigg|\geq \frac{N^{1-3s}}{2\b} \Big(t-\frac{1}{|\b|N}\Big).
\end{aligned}
\]
As a result, by fixing $t = T$ and requiring $N > \frac{2}{|\b|T}$,
\be\label{AbEs}
    |\mathscr{F}_x\psi_3(N,T)| \gs N^{1-3s} T.
\ee

Combining \eqref{AbEs} with \eqref{psi3bdd}, we deduce
\begin{equation*}
1\gtrsim \|\psi_3(\cdot, T)\|_{H^s(\m{T})} \gtrsim  \left(N^{2s}\big|\mathscr{F}_x\psi_3(N, T)\big|^2\right)^{\frac{1}{2}}\gtrsim N^{1-2s}T.
\end{equation*}
Sending $N\to\infty$, we conclude that \(s\ge\frac12\) is necessary.

\subsection{Proof of Theorem \ref{thm4}}
In this proof, $\a\in (0,4)\setminus\{1\}$ and $\b\neq 0$. We will justify this theorem in three cases. 

\noindent{\bf Case 1}: $R_\a\in\mathbb{Q}$.

In this case, we will prove $s\geq \frac12$ when the solution map is at least $C^3$. Since \(R_\alpha\in\mathbb{Q}\), we have \(c_1,c_2\in\mathbb{Q}\), where $c_1,c_2$ are given in \eqref{c1c2}: 
\[ c_1 = \frac12 + \frac{R_\a}{6}, \quad c_2 = \frac12 - \frac{R_\a}{6}, \quad R_\a = \sqrt{12/\a - 3}.\] 
Hence, we may choose integers $N$ such that \(c_1N, c_2N \in\mathbb{Z}\).

We define
\[
\phi(x) = 0,\quad
\psi(x) = \frac{1}{\pi N^s}\big[\cos( c_1 Nx) +\cos (c_2 Nx)\big] ,\quad \forall\, x \in \mathbb{R},
\]
then
\begin{equation}\label{Intial_Ill_5}
	\widehat\phi(\xi) = 0, \quad \widehat\psi(\xi) = \frac{1}{N^s}\big[ \delta (|\xi| - c_1 N) + \delta (|\xi| - c_2 N) \big], \quad \forall\, \xi\in\m{Z}.
\end{equation}
To simplify the notation, we define  
\[
N_1 := c_1N \quad \text{and} \quad N_2 := c_2N.
\]
The choice of $(\phi,\psi)$ implies $\|(\phi,\psi)\|_{H^s(\m{T}) \times H^s(\m{T})}\sim 1$.


Next, we fix $\xi = N_1$ and estimate $|\F_{x}\psi_3(N_1, t)|$. The value of $\xi$ is chosen to be $N_1$ so that \eqref{Temp4} is satisfied. Then it follows from (\ref{psi3}) that 
\be\label{psi3F_res}
|\F_{x} \psi_3(N_1, t)| 
=3N_1 
\bigg|\int_{\m{Z}}\int_{\m{Z}} \Psi_2(\xi_1, \xi_2) 
\bigg( \int_{0}^{t} e^{i G(-N_1,N_1-\xi_1,\xi_1) t'}
\int_{0}^{t'} e^{i G(\xi_2, \xi_1-\xi_2,-\xi_1) \tau} \,d\tau \,dt'\bigg)
\dd\xi_2 \dd\xi_1\bigg|,
\ee
where the function $G$ is as defined in (\ref{G}) and 
\be\label{Psi2}
	\Psi_2(\xi_1, \xi_2) := \xi_1 \wh{\psi}(N_1-\xi_1) \wh{\psi}(\xi_2) \wh{\psi}(\xi_1-\xi_2).
\ee
For any $(\eta_1,\eta_2,\eta_3)\in\Gamma_3$, we substitute $\xi_2$ with $-(\xi_1+\xi_3)$ to obtain 
\be\label{G_gen}\begin{split}
    G(\eta_1,\eta_2,\eta_3) &= P_{\alpha,\beta}(\eta_1)+P_{\alpha,\beta}(\eta_2)+P_{1,0}(\eta_3) \\
    &= -3\a\eta_3 \Big( \eta_1^2 + \eta_3 \eta_1 + \frac{\a-1}{3\a}\eta_3^2 \Big) + \b \eta_3\\
    &= -3\a\eta_3 (\eta_1 + c_1 \eta_3)(\eta_1 + c_2 \eta_3) + \b \eta_3.
\end{split}\ee
Denote $\mathcal{D}$ as the support of $\Psi_2$, that is, 
\be\label{suppD}
    \mathcal{D} :=  \bigl\{ (\xi_1,\xi_2)\in\m{Z}^2 :\Psi_2(\xi_1,\xi_2) \neq 0 \bigr\}.
\ee
Then we only need to consider the points in $\mathcal{D}$ to compute $\F_{x} \psi_3(N_1, t)$, so we require $\xi_1\neq 0$ and 
\[
\xi_2, \, N_1 - \xi_1,\, \xi_1-\xi_2 \in \text{supp}\ \widehat\psi.
\]
Since $c_1+c_2=1$, which implies that $N_1+N_2=N$, the qualified values of $\xi_1$ and $\xi_2$ can be summarized in the following table, where ``$\checkmark$'' indicates that
$(\xi_1,\xi_2)\in \mathcal{D}$.
\begin{table}[H]
	\centering
	\renewcommand{\arraystretch}{1.2} 
	\begin{tabular}{|c|c|c|c|c|}
		\hline
		\diagbox[height=2.6em, width=5.6em]{%
			\hspace{0.8em}$\xi_1$\hspace{1.8em}%
		}{%
			\hspace{1.8em}$\xi_2$\hspace{0.8em}%
		} 
		& $N_1$ & $N_2$ & $-N_1$ & $-N_2$ \\
		\hline
		$2N_1$    & \checkmark &   & &  \\
		\hline
		$N$       &  \checkmark&\checkmark &   & \\
		\hline
		$N_1-N_2$ & \checkmark &       &   & \checkmark\\
		\hline 
	\end{tabular}
	\caption{Values of $\xi_1,\xi_2$ for which $(\xi_1, \xi_2)\in\mathcal{D}$ } \label{tabsupportxi1xi2}
\end{table}

Based on $\mathcal{D}$, we have
\be\label{psi3F_sim}
\begin{aligned}
	|\mathscr{F}_x\psi_3(N_1,t)| = \frac{3N_1}{4\pi^2 N^{3s}} \bigg|\sum_{(\xi_1,\xi_2)\in \mathcal{D}} \xi_{1} 
	\bigg(\int_{0}^{t} e^{i G(-N_1,N_1-\xi_1,\xi_1) t'}
	\int_{0}^{t'} e^{i G(\xi_2, \xi_1-\xi_2,-\xi_1) \tau} \,d\tau \,dt'\bigg) \bigg|.
\end{aligned}
\ee
Meanwhile, it follows from (\ref{G_gen}) that 
\be\label{G_decomp1}
G(\xi_2,\xi_1-\xi_2,-\xi_1) = 3\a\xi_1 (\xi_2 - c_1 \xi_1)(\xi_2 - c_2 \xi_1) - \b\xi_1.
\ee
Since $G(\xi_2,\xi_1-\xi_2,-\xi_1) \neq 0$ for any $(\xi_1, \xi_2)\in\mathcal{D}$, then
\[
\int_{0}^{t'} e^{i G(\xi_2, \xi_1-\xi_2,-\xi_1) \tau} \,d\tau = \frac{-i}{G(\xi_2, \xi_1-\xi_2,-\xi_1)} \Big[ e^{i G(\xi_2, \xi_1-\xi_2,-\xi_1) t'} - 1 \Big],
\]
and 
\be\label{Temp6}
	|\mathscr{F}_x\psi_3(N_1,t)| = \frac{3N_1}{4\pi^2 N^{3s}} \Bigg|\sum_{(\xi_1,\xi_2)\in \mathcal{D}} \frac{\xi_{1}}{G(\xi_{2},\xi_{1}-\xi_{2},-\xi_{1})} \int_{0}^{t} \mathcal{G}(\xi_1,\xi_2,t') \dd t'\Bigg|,
\ee
where we set 
\be\label{G_script}
    \mathcal{G}(\xi_1,\xi_2,t') := e^{it'[G(-N_1,N_1-\xi_{1},\xi_{1})+G(\xi_{2},\xi_{1}-\xi_{2},-\xi_{1})]} - e^{it'G(-N_1,N_1-\xi_{1},\xi_{1})}.
\ee
Since $|\mathcal{G}(\xi_1,\xi_2,t')|$ is at most a constant, 
the term $\frac{\xi_{1}}{G(\xi_{2},\xi_{1}-\xi_{2},-\xi_{1})}$ plays a more important role. Let 
\be\label{F}
    F(\xi_1, \xi_2) = \frac{\xi_1}{G(\xi_{2},\xi_{1}-\xi_{2},-\xi_{1})}, \quad\forall\,(\xi_1,\xi_2)\in\mathcal{D}.
\ee
According to the decomposition (\ref{G_decomp1}), we have 
\[F(\xi_1, \xi_2) = \frac{1}{3\a  (\xi_2 - c_1\xi_1)(\xi_2 - c_2\xi_1) - \b}.\]
Thus, the distance between $\xi_2$ and both $c_1\xi_1$ and $c_2\xi_1$ determine the size of $F$.

According to Table \ref{tabsupportxi1xi2}, we divide $\mathcal{D}$ into two parts: $\mathcal{D} = \mathcal{D}_1 \cup \mathcal{D}_2$, where 
\[
\mathcal{D}_1 := \{(N,N_1),(N,N_2)\}, \qquad \mathcal{D}_2 := \mathcal{D} \setminus \mathcal{D}_1.
\]
When $(\xi_1,\xi_2)\in\mathcal{D}_1$, no matter $(\xi_1,\xi_2) = (N,N_1)$ or $(N,N_2)$, it always holds that 
\begin{equation*}
	F(\xi_1, \xi_2)= -\frac{1}{\b} \quad\text{and}\quad 
    \mathcal{G}(\xi_1,\xi_2,t') = 1 - e^{it'G(-N_1, -N_2,N)},
\end{equation*}
where we used the facts that $N_1+N_2=N$ and 
\begin{equation}\label{Temp4}
	G(-N_1,-N_2,N) = -G(N_1,N_2,-N) = -G(N_2,N_1,-N) = \b N.
\end{equation}
As a result, 
\[
I_1(t) := \sum_{(\xi_1,\xi_2)\in\mathcal{D}_1} F(\xi_1, \xi_2)\int_{0}^{t} \mathcal{G}(\xi_1,\xi_2,t') \,dt' \\
= -\frac{2}{\beta} \int_{0}^{t}\big(1 - e^{it'\b N}\big) \,dt'.
\]
In addition,
\[\begin{split}
	\bigg| \int_{0}^{t}\big(1 - e^{it'\b N}\big) \,dt' \bigg| &\geq t - \bigg|\int_{0}^{t} e^{it'\b N} \,dt'\bigg| \geq t - \frac{2}{|\b| N}.
\end{split}\]
Therefore, by fixing $t = T$ and requiring $N > \frac{4}{|\b|T}$,  we have 
\be\label{I1}
|I_1(T)| \geq \frac{T}{|\b|}.
\ee

On the other hand, for any $(\xi_1,\xi_2)\in \mathcal{D}_2 = \mathcal{D}\setminus \mathcal{D}_1$, we have the following:
\begin{equation}\label{Temp521}
	|F(\xi_1,\xi_2)| =\frac{1}{|3\a  (\xi_1-c_1\xi_2)(\xi_1-c_2\xi_2)-\b|}\leq \frac{C}{N^2},
\end{equation}
as long as $N$ is large enough, where $C$ is some constant which only depends on $\a$ and $\b$. Denote 
\[
    I_2(t) :=\sum_{(\xi_1,\xi_2)\in\mathcal{D}_2} F(\xi_1, \xi_2)\int_{0}^{t} \mathcal{G}(\xi_1,\xi_2,t') \,dt'.
\]
Since $\mathcal{D}_2$ only contains three points and $|\mathcal{G}(\xi_1,\xi_2,t')|\leq 2$, we conclude 
\be\label{I2}
    |I_2(T)| \leq \frac{CT}{N^2}.
\ee

Plugging (\ref{I1}) and (\ref{I2}) into (\ref{Temp6}) yields
\begin{equation*}
	\begin{split}
		|\mathscr{F}_x\psi_3(N_1,T)| = \frac{3N_1}{4\pi^2 N^{3s}} |I_1(T) + I_2(T)| &\sim N^{1-3s} |I_1(T) + I_2(T)| \\
		&\geq N^{1-3s} (|I_1(T)| - |I_2(T)|) \\
		&\geq N^{1-3s} \Big( \frac{1}{|\b|} - C N^{-2} \Big) T.
	\end{split}
\end{equation*}
Hence, by requiring $N$ is larger than $\sqrt{2C|\b|}$, we obtain 
\be\label{AbEs0}
|\mathscr{F}_x\psi_3(N_1,T)| \gs N^{1-3s} T.
\ee
Combining \eqref{AbEs0} with \eqref{Mainidea}, where $k=3$, we deduce
\begin{equation*}
	1\gtrsim \|\psi_3(\cdot, T)\|_{H^s(\m{T})} \gtrsim  \left(N_1^{2s}\big|\mathscr{F}_x\psi_3(N_{1}, T)\big|^2\right)^{\frac{1}{2}}\gtrsim N^{1-2s}T.
\end{equation*}
Sending $N\to\infty$ shows that \(s\ge\frac12\) is necessary.

\noindent{\bf Case 2}: $R_\a\notin \mathbb{Q}$ and $s_{\a,\b}< 1$.

In this case, we will prove $s \geq s^*(\alpha,\beta)$, where 
$s^*(\alpha,\beta) = \frac{1+s_{\a,\b}}{2}$ is defined in $(\ref{ci_res})$, when the solution map is at least $C^3$.
Under the assumption $s_{\alpha,\beta}< 1$, both $\nu_{\lambda}(c_1)<1$ and $\nu_{\lambda}(c_2)< 1$ hold due to (\ref{s_0}), where $c_1$ and $c_2$ are the numbers in (\ref{c1c2}), and 
$\lam = \frac{\b}{\a R_\a}$.
The main structure of the proof is analogous to that of Case 1, so we will omit details that have already been shown in Case 1. 

We define 
\be\label{Intial_Ill_6}
\phi = 0,\quad
\psi = \frac{1}{\pi N^s}\bigg[\cos\Big( [[c_1N+\frac{\lambda}{N}]]x\Big) +\cos \Big([[c_2N-\frac{\lambda}{N}]]x\Big)\bigg], \quad \forall x \in \mathbb{R},
\ee
where the notation $[[\cdot]]$ means 
\begin{equation}\label{nearint}
	[[y]] = \left\{
	\begin{aligned}
		&\text{the nearest even integer to $y$}, &\quad \text{if } \quad y\in \mathbb{Z}+1/2,\\
		&\text{the nearest integer to $y$}, &\quad \text{if } \quad y\notin \mathbb{Z}+1/2.
	\end{aligned}\right.
\end{equation}
For example, $[[2.5]] = 2$, $[[3.5]] = 4$, $[[2.1]] = 2$, $[[2.9]] = 3$. In the setting of (\ref{nearint}), one can easily justify the following identity for any integer $N$ and any real number $\lambda$.
\be\label{int}
[[c_1N+\frac{\lambda}{N}]]+[[c_2N-\frac{\lambda}{N}]] = N.
\ee 
To simplify notation and align with the earlier proof, we define
\be\label{N1N2}
N_1 := [[c_1N+\frac{\lambda}{N}]] \quad \text{and} \quad N_2 := [[c_2N-\frac{\lambda}{N}]].
\ee
then
\begin{equation*}
	\widehat{\phi}(\xi) = 0,\quad
	\widehat{\psi}(\xi)
	= \frac{1}{\pi N^s}\Big[\,
	\delta\bigl(|\xi| - N_1\bigr)
	+
	\delta\bigl(|\xi| - N_2\bigr)
	\,\Big],
	\quad \forall\,\xi\in\mathbb{Z}.
\end{equation*}

Fix $\xi = N_1$,  we estimate $|\F_{x}\psi_3(N_1, t)|$. 
Similar to Case 1, we compute $|\F_{x} \psi_3(N_1,t)|$ as (\ref{psi3F_res}), where $\Psi_2$ is defined in (\ref{Psi2}). We also adopt the notation $\mathcal{D}$ in (\ref{suppD}) for the support of $\Psi_2$, it turns out that Table \ref{tabsupportxi1xi2} keeps the same regarding the elements in $\mathcal{D}$, with the only difference being the new definitions of $N_1$ and $N_2$ in (\ref{N1N2}). The difficulties of (\ref{N1N2}) are twofold, one is the irrationality of $c_1 N$ and $c_2 N$, and the other one is the involvement of the term $\frac{\lam}{N}$. According to (\ref{Temp6}) in Case 1, 
\[
    |\mathscr{F}_x\psi_3(N_1,t)| = \frac{3N_1}{4\pi^2 N^{3s}} \Bigg|\sum_{(\xi_1,\xi_2)\in \mathcal{D}} F(\xi_1, \xi_2) \int_{0}^{t} \mathcal{G}(\xi_1,\xi_2,t') \dd t'\Bigg|,
\]
where (see (\ref{F}) and (\ref{G_script}))
\[
F(\xi_1,\xi_2) = \frac{\xi_1}{G(\xi_{2},\xi_{1}-\xi_{2},-\xi_{1})},
\]
and 
\[
    \mathcal{G}(\xi_1,\xi_2,t') = e^{it'[G(-N_1,N_1-\xi_{1},\xi_{1})+G(\xi_{2},\xi_{1}-\xi_{2},-\xi_{1})]} - e^{it'G(-N_1,N_1-\xi_{1},\xi_{1})}.
\]
We also split $\mathcal{D}$ as $\mathcal{D}_1 \cup \mathcal{D}_2$, where $\mathcal{D}_1 := \{(N,N_1),(N,N_2)\}$ and $\mathcal{D}_2 = \mathcal{D} \setminus \mathcal{D}_1$. Then 
\be\label{psi3_split}
    |\mathscr{F}_x\psi_3(N_1,t)| = \frac{3N_1}{4\pi^2 N^{3s}} |I_1(t) + I_2(t)|,
\ee
where 
\[
    I_1(t) = \sum_{(\xi_1,\xi_2)\in\mathcal{D}_1} F(\xi_1, \xi_2)\int_{0}^{t} \mathcal{G}(\xi_1,\xi_2,t') \,dt', \quad
    I_2(t) = \sum_{(\xi_1,\xi_2)\in\mathcal{D}_2} F(\xi_1, \xi_2)\int_{0}^{t} \mathcal{G}(\xi_1,\xi_2,t') \,dt'.
\]

We first estimate $I_1(t)$. For any $(\xi_1,\xi_2)\in \mathcal{D}_1$, the values of $F(\xi_1, \xi_2)$ and $\mathcal{G}(\xi_1,\xi_2,t')$ keep the same no matter $(\xi_1,\xi_2) = (N,N_1)$ or $(N,N_2)$, where we used the fact that $G$ is symmetric with respect to its first two variables, so
\be\label{I1sim}
    I_1(t) = 2 F(N, N_1) \int_{0}^{t} \mathcal{G}(N,N_1,t')\,dt' = \frac{2N}{G(N_1, N_2,-N)} \int_{0}^{t} \mathcal{G}(N,N_1,t') \,dt'.
\ee
Based on the definition (\ref{G_gen}) of $G$, 
\be\label{Temp536}
\begin{split}
G(N_1,N_2,-N) &= 3\a N (N_1 - c_1 N)(N_1 - c_2 N) - \b N \\
&= 3\alpha N(N_1 - x_1)(N_1 - x_2),
\end{split}\ee
where
\begin{equation*}
	x_1 = \frac{1}{2}N + \frac{1}{6}\sqrt{(R_{\alpha}N)^2+\frac{12\b}{\alpha}}, 
	\qquad 
	x_2 = \frac{1}{2}N - \frac{1}{6}\sqrt{(R_{\alpha}N)^2+\frac{12\b}{\alpha}}, 
	\qquad 
	R_{\alpha} = \sqrt{12/\alpha-3}.
\end{equation*}
We expand $x_1$ and $x_2$ in terms of the order of $\xi_1$ as follows:
\[
x_1 = c_1 N + \frac{\lam}{N} + Q_1(N), \qquad
x_2 = c_2 N - \frac{\lam}{N} + Q_2(N),
\qquad 
\lam = \frac{\b}{\a R_\a}, 
\]
with
\be\label{Temp539}
|Q_{j}(N)| \leq \frac{12\lam^2}{R_\a N^3} = O(N^{-3}), \qquad\, j=1,2.
\ee
Then it follows from \eqref{Temp536} that with \eqref{Temp539} yields
\begin{equation}\label{Gbdd}
	\begin{aligned}
		|G(N_1, N_2, -N)| &= 3\a N\Big|N_1-c_1 N - \frac{\lambda}{N}-Q_1(N)\Big|\Big|N_1-c_2N+\frac{\lambda}{N}-Q_2(N)\Big| \\
		&\sim N^2 \Big| N_1 - c_1 N - \frac{\lambda}{N} -Q_1(N)\Big| \\
        &= N^3 \Big| c_1-\frac{N_1}{N} + \frac{\lambda}{N^2} + \frac{Q_1(N)}{N}\Big|.
	\end{aligned}
\end{equation}
On the one hand, we combine (\ref{Temp539}) and (\ref{Gbdd}) together to obtain 
\be\label{Gbdd2}
	|G(N_1, N_2, -N)| \ls N^3\Big| c_1-\frac{N_1}{N} + \frac{\lambda}{N^2}\Big| + \frac{1}{N}.
\ee
For any $\varepsilon>0$, by Definition \ref{Def, bim} and the choice of $N_1$, there exist infinitely many $N$ such that
\begin{equation*}
	\Big|c_1-\frac{N_1}{N}+\frac{\lambda}{ N^2}\Big|\leq \frac{1}{N^{\mu_{\lambda}(c_1)-\varepsilon}}.
\end{equation*}
Since $\lam = \frac{\b}{\a R_\a}\neq 0$, it follows from 
Proposition \ref{muequnuplus2} that $\nu_{\lambda}(c_1) = \mu_{\lambda}(c_1) -2$, so 
\be\label{dist_est}
	\Big|c_1-\frac{N_1}{N}+\frac{\lambda}{ N^2}\Big|\leq \frac{1}{N^{\nu_{\lambda}(c_1)+2-\varepsilon}}.
\ee
Combining (\ref{Gbdd2}) with (\ref{dist_est}), we conclude that there exist infinitely many $N$ such that 
\be\label{G_ubd1}
	|G(N_1, N_2, -N)|\ls N^{1-\nu_{\lambda}(c_1)+\varepsilon} + \frac{1}{N},
\ee
Since $\nu_{\lam}(c_1) < 1$, then we know $|G(N_1, N_2, -N)| \ls N^{1-\nu_{\lambda}(c_1)+\varepsilon}$ and
\begin{equation}\label{Fbdd}
    |F(N,N_1)| = \frac{2N}{|G(N_1,N_2,-N)|}\gs N^{\nu_{\lam}(c_1)-\varepsilon}.
\end{equation}
On the other hand, 
\[
\begin{split}
\mathcal{G}(N,N_1,t') &= e^{it'[G(-N_1,N_1-N,N)+G(N_1,N-N_1,-N)]} - e^{it'G(-N_1,N_1-N,N)} \\
&= 1 - e^{-it'G(N_1,N_2,-N)},
\end{split}
\]
which implies
\be\label{int_small}
\begin{split}
\bigg| \int_0^t \mathcal{G}(N,N_1,t')  \,dt' \bigg|
\ge t - \frac{2}{|G(N_1,N_2,-N)|}.
\end{split}
\ee
Based on Definition \ref{def_nu}, for $\veps' := \frac{1-\nu_{\lam}(c_1)}{2}>0$, there exists $K>0$ such that 
\be\label{c1_lbdd}
    \bigg|c_1 - \frac{N_1}{N} + \frac{\lambda}{N^2}\bigg| \geq \frac{K}{N^{2+\nu_{\lam}(c_1)+\veps'}},
\ee
for all sufficiently large $N$. Then it follows from \eqref{Gbdd}, \eqref{Temp539} and (\ref{c1_lbdd}) that 
\be\label{G_lbdd}\begin{split}
    |G(N_1,N_2,-N)| &\gs N^3\bigg|c_1-\frac{N_1}{N}+\frac{\lambda}{N^2}\bigg| - N^2 |Q_1(N)| \\
    & \gtrsim N^{1-\nu_{\lambda}(c_1)-\varepsilon'} - \frac{1}{N} \sim N^{\frac{1-\nu_{\lam}(c_1)}{2}},
\end{split}\ee
for all sufficiently large $N$. Now we fix $t=T$ and then apply \eqref{int_small} and \eqref{G_lbdd} to obtain
\be\label{int_G_lbd}
\begin{split}
    \bigg| \int_0^t \mathcal{G}(N,N_1,t')  \,dt' \bigg| \ge T - C N^{-\frac{1-\nu_\lam(c_1)}{2}} \geq \frac{T}{2},
\end{split}
\ee
for sufficiently large $N$.
Substituting \eqref{Fbdd} and (\ref{int_G_lbd}) into (\ref{I1sim}) yields 
\be\label{I1alpha}
    |I_1(T)| \gs N^{\nu_{\lambda}(c_1)-\varepsilon}\, T.
\ee

Next, we estimate $I_2(T)$. For any $(\xi_1,\xi_2)\in \mathcal{D}_2$, we know
\begin{equation*}
	|F(\xi_1,\xi_2)| = \frac{1}{|3\a (\xi_2-c_1\xi_1)(\xi_2-c_2\xi_1)-\beta|}\sim \frac{1}{N^2}.
\end{equation*}
Consequently, 
\begin{equation}\label{I2alpha}
	|I_2(T)|  \lesssim \sum_{(\xi_1,\xi_2)\in \mathcal{D}_2} N^{-2}\int_0^{T} |\mathcal{G}(\xi_1,\xi_2,t')| \dd t' \lesssim N^{-2}T,
\end{equation}

Finally, based on \eqref{I1alpha} and \eqref{I2alpha}, it follows from (\ref{psi3_split}) that
\[
	\big| \mathscr{F}_x\psi_3(N_1, T) \big| \gtrsim N^{1-3s}(|I_1(T)| - |I_2(T)|) \gtrsim N^{1-3s+\nu_{\lambda}(c_1)-\varepsilon}T.
\]
Hence, we apply (\ref{Mainidea}) with $k=3$ to conclude that
\begin{equation*}
	1\gtrsim \|\psi_3(\cdot,T)\|_{H^s}\gtrsim  \left(N_1^{2s}\big|\mathscr{F}_x\psi_3(N_1,T) \big|^2\right)^{\frac{1}{2}} \gtrsim N^{-2s+1+\nu_{\lambda}(c_1)-\varepsilon} T.
\end{equation*}
As $N\to\infty$, it follows that $s\geq \frac{1}{2}+\frac{\nu_\lambda(c_1)}{2}-\frac{\varepsilon}{2}$. Since $\varepsilon>0$ is arbitrary, we know $s\geq \frac{1}{2}+\frac{\nu_\lambda(c_1)}{2}$. Similarly, we can also verify that
$s \geq \frac{1}{2} + \frac{\nu_\lambda(c_2)}{2} $. Ultimately, we obtain $s\geq \frac{1+s_{\a,\b}}{2}$.

\noindent{\bf Case 3}: $R_\a\notin \mathbb{Q}$ and $s_{\a,\b}\geq 1$.

In this case, we will prove that $s\geq 1$ when the solution map is at least $C^2$. Since $s_{\alpha,\beta} := \max\{\nu_{\lam}(c_1), \nu_{\lam}(c_2)\}\geq 1$, without loss of generality, we assume that \( \nu_{\lambda}(c_1) \geq 1 \) since the other case \( \nu_{\lambda}(c_2) \geq 1 \) can be handled similarly.
We take the same initial data as in Case 2, given by \eqref{Intial_Ill_6}, where $N_1,N_2$ are defined by \eqref{N1N2}. Recalling $\mathscr{F}_x\phi_2$ in \eqref{phi2}, we fix $\xi = N$ and estimate $|\mathscr{F}_x \phi_2(N,t) |$ which can be represented as follows:
\[
	|\mathscr{F}_{x}\phi_2(N, t)| = N \bigg| \int_{\mathbb{Z}}\int_0^t e^{it'G(\xi_1, N-\xi_{1}, -N)} \widehat\psi(\xi_1) \widehat\psi(N-\xi_{1})\,dt' \,d\xi_{1} \bigg|.
\] 
Since both $\xi_1$ and $N-\xi_1$ must be in the support of $\wh{\psi}$, then $\xi_1$ has to be either $N_1$ or $N_2$. By adding these two cases, we have 
\begin{equation}\label{phi2Bdd}
	|\F_{x}\phi_2(N, t)| = \frac{N}{2\pi} \bigg| \int_0^t 2e^{it'G(N_1, N_2, -N)}   \frac{1}{N^{2s}}\,dt' \bigg| = \frac{2N^{1-2s}}{\pi}\frac{\Big|\sin\Bigl(\frac{G(N_1,N_2,-N)t}{2}\Bigr)\Big|}{\big|G(N_1,N_2,-N)\big|},
\end{equation}
where in the first equality, we used the fact that $G(\eta_1,\eta_2,\eta_3)$ is symmetric with respect to its first two variables.

For any $\varepsilon>0$, it follows from \eqref{G_ubd1} that
\[
    |G(N_1, N_2, -N)|\ls N^{1-\nu_{\lambda}(c_1)+\varepsilon} + \frac{1}{N} \ls N^{\veps},
\]
where the last inequality is due to $\nu_{\lambda}(c_1) \geq 1$. By choosing $t = t_{N} := N^{-2\veps} T$, then 
\begin{equation*}
\frac{G(N_1,N_2,-N) t_N}{2} \lesssim N^{-\varepsilon}\,T \to 0,
\qquad\text{as }N\to +\infty.
\end{equation*}
Noting that $\frac{\sin x}{x}\geq\frac{1}{2}$ for small $|x|$, hence, for sufficiently large $N$,
\begin{equation}\label{sim}
\frac{\Big|\sin\Bigl(\frac{G(N_1,N_2,-N) t_N}{2}\Bigr)\Big|}{\big|G(N_1,N_2,-N)\big|}\geq \frac{1}{2}\cdot\frac{t_N}{2} = \frac{1}{4}N^{-2\veps}T.
\end{equation}
Therefore, by taking $N$ sufficiently large so that \eqref{sim} is valid, it then follows from (\ref{phi2Bdd}) that 
\be\label{AbEs2}
|\mathscr{F}_x\phi_2(N,t_N)| \gtrsim N^{1-2s-2\veps} T.
\ee
Finally, since $t_N\in [0,T]$, we apply (\ref{Mainidea}) with $k=2$ to obtain
\begin{equation*}
1\gtrsim \|\phi_2(\cdot, t_N)\|_{H^s(\mathbb{T})}
\gtrsim
\bigl(N^{2s}\big|\mathscr{F}_x\phi_2(N, t_N)\big|^2\bigr)^{\frac12}
\gtrsim N^{1-s-\veps}T.
\end{equation*}
Letting $N\to\infty$, we conclude that $s\geq 1-\veps$,
which implies $s\geq 1$ since $\veps$ is arbitrary.

\section*{Acknowledgments}
X.Yang is supported by National Natural Science Foundation of China (No. 12401299), Natural Science Foundation of Jiangsu Province (No. BK20241260), Scientific Research Center of Applied Mathematics of Jiangsu Province (No. BK20233002).


\bigskip

\thanks{(K. Wang) School of Mathematics, Southeast University, Nanjing, Jiangsu 211189, China} 

\thanks{Email: ke.wang.math@seu.edu.cn}

\medskip

\thanks{(X. Yang) School of Mathematics, Southeast University, Nanjing, Jiangsu 211189, China} 

\thanks{Email: xinyang@seu.edu.cn}

\end{document}